\title{Rainbow vertex connection of digraphs}

\author{
Hui Lei$^1$, Shasha Li$^2$, Henry Liu$^3$\footnote{Corresponding author} , Yongtang Shi$^1$\\
\\
\normalsize $^1$Center for Combinatorics and LPMC\\
\normalsize Nankai University, Tianjin 300071, China\\
\normalsize leihui0711@163.com, shi@nankai.edu.cn\\
\\
\normalsize $^2$Ningbo Institute of Technology\\
\normalsize Zhejiang University, Ningbo 315100, China\\
\normalsize lss@nit.zju.edu.cn\\
\\
\normalsize $^3$School of Mathematics and Statistics\\
\normalsize Central South University, Changsha 410083, China\\
\normalsize henry-liu@csu.edu.cn\\
}
\date{15 July 2017}

\documentclass[11pt]{article}
\usepackage{amsmath,amssymb,amscd,amsthm}
\usepackage{xcolor}
\newdimen\unit\newdimen\psep\newcount\nd\newcount\ndx\newbox\dotb\newbox\ptbox
\newdimen\dx\newdimen\dy\newdimen\dxx\newdimen\dyy\newdimen\hgt
\newdimen\xoff\newdimen\yoff
\newcommand\clap[1]{\hbox to 0pt{\hss{#1}\hss}}
\newcommand\vdisk[1]{{\font\dotf=cmr10 scaled #1\dotf.}}
\newcommand\varline[2]{\setbox\dotb\hbox{\vdisk{#1}}\xoff=-.5\wd\dotb
\wd\dotb=0pt\yoff=-.5\ht\dotb\psep=#2\ht\dotb}
\newcommand\varpt[1]{\setbox\ptbox\clap{\vdisk{#1}}\setbox\ptbox
\hbox{\raise-.5\ht\ptbox\box\ptbox}}
\newcommand\cpt{\copy\ptbox}
\newcommand\point[3]{\rlap{\kern#1\unit\raise#2\unit\hbox{#3}}}
\newcommand\setnd[4]{\dx=#3\unit\advance\dx-#1\unit\divide\dx by\psep
\dy=#4\unit\advance\dy-#2\unit\divide\dy by\psep \multiply\dx
by\dx\multiply\dy by\dy\advance\dx\dy\nd=1\advance\dx-1sp
\loop\ifnum\dx>0\advance\dx-\nd sp\advance\nd1\advance\dx-\nd
sp\repeat}
\newcommand\dl[4]{{\setnd{#1}{#2}{#3}{#4}\dline{#1}{#2}{#3}{#4}\nd}}
\newcommand\dline[5]{{\nd=#5\hgt=#2\unit\dx=#3\unit\advance\dx-#1\unit
\divide\dx by\nd\dy=#4\unit\advance\dy-#2\unit\divide\dy by\nd
\advance\hgt\yoff\rlap{\kern#1\unit\kern\xoff\loop\ifnum\nd>1\advance\nd-1
\advance\hgt\dy\kern\dx\raise\hgt\copy\dotb\repeat}}}
\newcommand\ellipse[4]{\qellip{#1}{#2}{#3}{#4}\qellip{#1}{#2}{#3}{-#4}%
\qellip{#1}{#2}{-#3}{#4}\qellip{#1}{#2}{-#3}{-#4}}
\newcommand\qellip[4]{{\setnd{0}{0}{#3}{#4}\dx=\unit\dy=0pt\raise\yoff\rlap{%
\kern#1\unit\kern\xoff\raise#2\unit\hbox{\loop\ifnum\dx>0\rlap{\kern#3\dx
\raise#4\dy\copy\dotb}\hgt=\dx\divide\hgt
by\nd\advance\dy\hgt\hgt=\dy \divide\hgt
by\nd\advance\dx-\hgt\repeat\rlap{\raise#4\dy\copy\dotb}}}}}
\newcommand\bez[6]{{\setnd{#1}{#2}{#3}{#4}\ndx=\nd\setnd{#3}{#4}{#5}{#6}
\ifnum\ndx>\nd\nd=\ndx\fi\dx=#3\unit\advance\dx-#1\unit\dy=#4\unit
\advance\dy-#2\unit\dxx=#5\unit\advance\dxx-#1\unit\dyy=#6\unit\advance
\dyy-#2\unit\advance\dxx-2\dx\advance\dyy-2\dy\divide\dxx
by\nd\divide\dyy
by\nd\advance\dx.25\dxx\advance\dy.25\dyy\divide\dx
by\nd\divide\dy by\nd \multiply\nd
by2\dx=100\dx\dy=100\dy\dxx=100\dxx\dyy=100\dyy\divide\dxx by\nd
\divide\dyy
by\nd\hgt=#2\unit\raise\yoff\rlap{\kern#1\unit\kern\xoff
\raise\hgt\copy\dotb\loop\ifnum\nd>0\advance\nd-1\advance\hgt0.01\dy
\kern0.01\dx\raise\hgt\copy\dotb\advance\dx\dxx\advance\dy\dyy\repeat}}}
\newcommand\ptu[3]{\point{#1}{#2}{\cpt\raise1ex\clap{$\scriptstyle{#3}$}}}
\newcommand\ptd[3]{\point{#1}{#2}{\cpt\raise-1.8ex\clap{$\scriptstyle{#3}$}}}
\newcommand\ptr[3]{\point{#1}{#2}{\cpt\raise-.4ex\rlap{$\ \scriptstyle{#3}$}}}
\newcommand\ptl[3]{\point{#1}{#2}{\cpt\raise-.4ex\llap{$\scriptstyle{#3}\ $}}}
\newcommand\ptlu[3]{\point{#1}{#2}{\raise.8ex\clap{$\scriptstyle{#3}$}}}
\newcommand\ptld[3]{\point{#1}{#2}{\raise-1.6ex\clap{$\scriptstyle{#3}$}}}
\newcommand\ptlr[3]{\point{#1}{#2}{\raise-.4ex\rlap{$\,\scriptstyle{#3}$}}}
\newcommand\ptll[3]{\point{#1}{#2}{\raise-.4ex\llap{$\scriptstyle{#3}\,$}}}
\newcommand\pt[2]{\point{#1}{#2}{\cpt}}

\newcommand\thnline{\varline{400}{.6}}

\varpt{2500}\thnline\unit=2em

\newtheorem{theorem}                   {Theorem}
\newtheorem{thm}             [theorem] {Theorem}

\newtheorem{lem}             [theorem] {Lemma}

\newtheorem{prop}             [theorem] {Proposition}
\newtheorem{claim}           [theorem] {Claim}

\newtheorem{prob}            [theorem] {Problem}

\def\diam{\textup{diam}}

\def\lrG{\overset{\textup{\hspace{0.02cm}\tiny$\leftrightarrow$}}{\phantom{\in}}\hspace{-0.3cm}G}

\def\lrWn{\overset{\textup{\hspace{0.09cm}\tiny$\leftrightarrow$}}{\phantom{\in}}\hspace{-0.34cm}W_n}
\def\rC{\overset{\textup{\hspace{0.05cm}\tiny$\rightarrow$}}{\phantom{\in}}\hspace{-0.32cm}C}

\def\lrP{\overset{\textup{\hspace{0.05cm}\tiny$\leftrightarrow$}}{\phantom{\in}}\hspace{-0.32cm}P}
\def\lrK{\overset{\textup{\hspace{0.08cm}\tiny$\leftrightarrow$}}{\phantom{\in}}\hspace{-0.34cm}K}
\def\lrC{\overset{\textup{\hspace{0.05cm}\tiny$\leftrightarrow$}}{\phantom{\in}}\hspace{-0.32cm}C}

\begin{document}
\maketitle

\begin{abstract}
An edge-coloured path is \emph{rainbow} if its edges have distinct colours. An edge-coloured connected graph is said to be \emph{rainbow connected} if any two vertices are connected by a rainbow path, and \emph{strongly rainbow connected} if any two vertices are connected by a rainbow geodesic. The (\emph{strong}) \emph{rainbow connection number} of a connected graph is the minimum number of colours needed to make the graph (strongly) rainbow connected. These two graph parameters were introduced by Chartrand, Johns, McKeon and Zhang in 2008. As an extension, Krivelevich and Yuster proposed the concept of \emph{rainbow vertex-connection}. The topic of rainbow connection in graphs drew much attention and various similar parameters were introduced, mostly dealing with undirected graphs. 

Dorbec, Schiermeyer, Sidorowicz and Sopena extended the concept of the rainbow connection to digraphs. In this paper, we consider the (strong) rainbow vertex-connection number of digraphs. Results on the (strong) rainbow vertex-connection number of biorientations of graphs, cycle digraphs, circulant digraphs and tournaments are presented.
\\
\\
{\bf Keywords:} Rainbow connection; rainbow vertex-connection; digraphs; tournaments
\end{abstract}

\section{Introduction}
In this paper, we consider graphs and digraphs which are finite and simple. That is, we do not permit the existence of loops, multiple edges (for graphs), and multiple directed arcs (for digraphs). For any undefined terms about graphs and digraphs, we refer the reader to the book of Bollob\'as \cite{BB98}.

The concept of rainbow connection in graphs was introduced by Chartrand et al.~\cite{CJMZ2008}. An edge-coloured path is \emph{rainbow} if its edges have distinct colours. For a connected graph $G$, an edge-colouring is \emph{rainbow connected} if any two vertices of $G$ are connected by a rainbow path. The \emph{rainbow connection number} of $G$, denoted by $rc(G)$, is the smallest possible number of colours in a rainbow connected edge-colouring of $G$. An edge-colouring of $G$ is \emph{strongly rainbow connected} if for any two vertices $u$ and $v$, there exists a rainbow $u-v$ geodesic, i.e., a rainbow $u-v$ path of minimum length. The \emph{strong rainbow connection number} of $G$, denoted by $src(G)$, is the smallest possible number of colours in a strongly rainbow connected edge-colouring of $G$. Clearly, we have diam$(G)\le rc(G)\le src(G)$, where diam$(G)$ is the \emph{diameter} of $G$. As an analogous setting for vertex-colourings, Krivelevich and Yuster \cite{KY2010} proposed the concept of rainbow vertex-connection. A vertex-coloured path is \emph{rainbow} if its internal vertices have distinct colours. A vertex-colouring of $G$ is \emph{rainbow vertex-connected} if any two vertices are connected by a rainbow path. The \emph{rainbow vertex-connection number} of $G$, denoted by $rvc(G)$, is the smallest possible number of colours in a rainbow vertex-connected vertex-colouring of $G$. Likewise, a vertex-colouring of $G$ is \emph{strongly rainbow vertex-connected} if for any two vertices $u$ and $v$, there exists a rainbow $u-v$ geodesic. The \emph{strong rainbow vertex-connection number} of $G$, denoted by $srvc(G)$, is the smallest possible number of colours in a strongly rainbow vertex-connected vertex-colouring of $G$. The parameter $srvc(G)$ was introduced by Li, Mao and Shi \cite{LMS2012}. An easy observation is that $\textup{diam}(G)-1\le rvc(G)\le srvc(G)$, and both equalities hold if diam$(G)=1$ or $2$. For more results on rainbow vertex-connection, we refer to \cite{LS,LiuMSrvc}. It was also shown that computing the rainbow (vertex-)connection number of an arbitrary graph is NP-hard \cite{CLRTY2008, CFMY,CLS2011}. For more results on the rainbow connection and rainbow vertex-connection of graphs, we refer to the survey \cite{LSS2013} and the book \cite{LS2012}.

Recently, Dorbec et al.~\cite{PIE2014} extended the concept of rainbow connection to digraphs. Given a digraph $D$, a \emph{directed path}, or simply a \emph{path} $P$ in $D$, is a sequence of vertices $x_0, x_1,\dots, x_\ell$ in $D$ such that $x_{i-1}x_i$ is an arc of $D$ for every $1\le i\le \ell$. $P$ is also called an \emph{$x_0-x_\ell$ path}, and its \emph{length} is the number of arcs $\ell$. An arc-coloured path is \emph{rainbow} if its arcs have distinct colours. Let $D$ be a strongly connected digraph, i.e., for any ordered pair of vertices $(u,v)$ in $D$, there exists a $u-v$ path. An arc-colouring of $D$ is \emph{rainbow connected} if for any ordered pair of vertices $(u,v)$, there is a rainbow $u-v$ path. The \emph{rainbow connection number} of $D$, denoted by $\overset{\rightarrow}{rc}(D)$, is the smallest possible number of colours in a rainbow connected arc-colouring of $D$. An arc-colouring of $D$ is \emph{strongly rainbow connected} if for any ordered pair of vertices $(u,v)$, there is a rainbow $u-v$ geodesic, i.e., a rainbow $u-v$ path of minimum length. The \emph{strong rainbow connection number} of $D$, denoted by $\overset{\rightarrow}{src}(D)$, is the smallest possible number of colours in a strongly rainbow connected arc-colouring of $D$. The function $\overset{\rightarrow}{src}(D)$ was introduced by Alva-Samos and Montellano-Ballesteros \cite{JJ2015}. We have diam$(D)\le \overset{\rightarrow}{rc}(D)\le\overset{\rightarrow}{src}(D)$. In \cite{JJ2015,JJJ2015}, the authors also studied the (strong) rainbow connection number of some classes of digraphs. 

In this paper, we extend the concept of rainbow vertex-connection to digraphs. A vertex-coloured path in a digraph is \emph{rainbow} if its internal vertices have distinct colours. A vertex-colouring of $D$ is \emph{rainbow vertex-connected} if for any ordered pair of vertices $(u,v)$ in $D$, there is a rainbow $u-v$ path. The \emph{rainbow vertex-connection number} of $D$, denoted by $\overset{\rightarrow}{rvc}(D)$, is the smallest possible number of colours in a rainbow vertex-connected vertex-colouring of $D$. Likewise, a vertex-colouring of $D$ is \emph{strongly rainbow vertex-connected} if for any ordered pair of vertices $(u,v)$, there exists a rainbow $u-v$ geodesic. The \emph{strong rainbow vertex-connection number} of $D$, denoted by $\overset{\rightarrow}{srvc}(D)$, is the smallest possible number of colours in a strongly rainbow vertex-connected vertex-colouring of $D$. 

We shall present some results regarding the functions $\overset{\rightarrow}{rvc}(D)$ and $\overset{\rightarrow}{srvc}(D)$. In Section \ref{gensect}, some basic results for general digraphs are proved. In Section \ref{specsect}, we consider the two parameters for the biorientations of some graphs, cycle digraphs, and circulant digraphs. Finally in Section \ref{tournamentssect}, we shall prove some results for tournaments. Many of our results will be the (strong) rainbow vertex-connection versions of some of the results in \cite{JJ2015,JJJ2015,PIE2014}.

\section{Definitions, remarks and basic results}\label{gensect}

We begin with some definitions about digraphs. For a digraph $D$, its vertex and arc sets are denoted by $V(D)$ and $A(D)$. For $u, v \in V(D)$, the \emph{distance} from $u$ to $v$ (i.e., the length of a shortest $u-v$ path) in $D$ is denoted by $d(u,v)$. If $uv, vu\in A(D)$, then we say that $uv$ and $vu$ are \emph{symmetric arcs}. If $uv\in A(D)$ and $vu\not\in A(D)$, then $uv$ is an \emph{asymmetric arc}. If $uv\in A(D)$ is an arc, then $v$ is an \emph{out-neighbour} of $u$, and $u$ is an \emph{in-neighbour} of $v$. The \emph{out-degree} (resp.~\emph{in-degree}) of $u$ is the number of out-neighbours (resp.~in-neighbours) of $u$. A \emph{tournament} is a digraph where every two vertices have one asymmetric arc joining them. Given a graph $G$, its \emph{biorientation} is the digraph $\lrG$ obtained by replacing each edge $uv$ of $G$ by the pair of symmetric arcs $uv$ and $vu$. 

For $X\subset V(D)$, the subdigraph of $D$ induced by $X$ is denoted by $D[X]$. Let $H$ be another digraph, and $u\in V(D)$. We define $D_{u\to H}$ to be the digraph obtained from $D$ and $H$ by replacing the vertex $u$ with a copy of $H$, and replacing each arc $xu$ (resp. $ux$) in $D$ by all the arcs $xv$ (resp.~$vx$) for $v\in V(H)$. We say that $D_{u\to H}$ is \emph{obtained from $D$ by expanding $u$ to $H$}. Note that the digraph obtained from $D$ by expanding every vertex to $H$ is also known as the \emph{lexicographic product} $D\circ H$.

Let $K_n$ and $P_n$ denote the complete graph and the simple path of order $n$, respectively. For $n\ge 3$, let $C_n$ and $\rC_n$ denote the cycle and directed cycle of order $n$, i.e., we may let $V(\rC_n)=\{v_0,\dots,v_{n-1}\}$ and $A(\rC_n)=\{v_0v_1,v_1v_2,\dots,v_{n-2}v_{n-1},v_{n-1}v_0\}$. For a directed cycle $\rC$ and $u,v\in V(\rC)$, we write $u\rC v$ for the $u-v$ path which uses the arcs of $\rC$.

Now, we shall present some remarks and basic results for the rainbow vertex-connection and strong rainbow vertex-connection numbers, for general digraphs. We first note that in a rainbow vertex-connected colouring of a strongly connected digraph $D$, there must be a path between some two vertices with at least diam$(D)-1$ colours. Thus, we have the following proposition.
\begin{prop}\label{pro1}
Let $D$ be a strongly connected digraph of order $n$ and let $\textup{diam}(D)$ be the diameter of $D$. Then 

\begin{equation}
\textup{diam}(D)-1 \le \overset{\rightarrow}{rvc}(D)\le \overset{\rightarrow}{srvc}(D)\le n.\label{pro1eq}
\end{equation}
\end{prop}

It is easy to see that the bioriented paths $\lrP_n$, for $n\ge 2$, form an infinite family of graphs where we have equalities in the first two inequalities in (\ref{pro1eq}). Also, it is not difficult to see that for every directed cycle $\rC_n$ with $n\ge 5$, we have equalities in the last two inequalities in (\ref{pro1eq}). These two results will be included in Theorem \ref{biorthm} and Proposition \ref{cyclepro}.

\begin{thm}\label{pro2}
Let $D$ be a non-trivial, strongly connected digraph.
\begin{enumerate}
\item[(a)] The following are equivalent.
\begin{enumerate}
\item[(i)] $D=\lrK_n$ for some $n\ge 2$.
\item[(ii)] \textup{diam}$(D)=1$.
\item[(iii)] $\overset{\rightarrow}{srvc}(D)=0$.
\item[(iv)] $\overset{\rightarrow}{rvc}(D)=0$.
\item[(v)]$\overset{\rightarrow}{src}(D)=1$.
\item[(vi)] $\overset{\rightarrow}{rc}(D)=1$.
\end{enumerate}
\item[(b)] 
\begin{enumerate}
\item[(i)] $\overset{\rightarrow}{srvc}(D) = 1$, if and only if $\overset{\rightarrow}{rvc}(D) = 1$, if and only if \textup{diam}$(D)=2$. 
\item[(ii)] $\overset{\rightarrow}{srvc}(D) = 2$ if and only if $\overset{\rightarrow}{rvc}(D) = 2$.
\item[(iii)] $\overset{\rightarrow}{src}(D)=2$ if and only if $\overset{\rightarrow}{rc}(D)=2$.
\end{enumerate}
Moreover, either of the conditions of (iii) implies any of the conditions of (i).
\end{enumerate}
\end{thm}

\begin{proof}
(a) In \cite{JJ2015}, Theorem 2.2(a), it was proved that (i), (v) and (vi) are equivalent. Clearly, we have (ii) $\Rightarrow$ (i) $\Rightarrow$ (iii). If (iii) holds, then by (\ref{pro1eq}), we have $0\le \overset{\rightarrow}{rvc}(D)\le \overset{\rightarrow}{srvc}(D)=0$, so that (iv) holds. Finally, again by (\ref{pro1eq}), we clearly have (iv) $\Rightarrow$ (ii).\\[1ex]
\indent (b)(i) If $\overset{\rightarrow}{srvc}(D) = 1$, then by (a) and (\ref{pro1eq}), we have $1\le \overset{\rightarrow}{rvc}(D)\le\overset{\rightarrow}{srvc}(D)=1$, and hence $\overset{\rightarrow}{rvc}(D) = 1$. Next, suppose that $\overset{\rightarrow}{rvc}(D) = 1$. Then (\ref{pro1eq}) implies diam$(D)\le 2$, and (a) implies diam$(D)\neq 1$, so that diam$(D)=2$. Now, suppose that diam$(D)=2$. Then the vertex-colouring of $D$ where all vertices have the same colour is strongly rainbow vertex-connected. Indeed, for any $u,v\in V(D)$, either $uv\in A(D)$, or $uv\not\in A(D)$ and there is a $u-v$ path of length $2$, which is also a rainbow $u-v$ geodesic. Thus, $\overset{\rightarrow}{srvc}(D) \le 1$. Also, (\ref{pro1eq}) implies $\overset{\rightarrow}{srvc}(D) \ge 1$, and hence, we have $\overset{\rightarrow}{srvc}(D)=1$.\\[1ex]
\indent (b)(ii) Suppose first that $\overset{\rightarrow}{srvc}(D) = 2$. Then $\overset{\rightarrow}{rvc}(D)\le 2$ by (\ref{pro1eq}). Clearly $\overset{\rightarrow}{rvc}(D) \neq 0$ by (a), and $\overset{\rightarrow}{rvc}(D) \neq 1$ by (b)(i). Thus $\overset{\rightarrow}{rvc}(D) = 2$. Conversely, suppose that $\overset{\rightarrow}{rvc}(D) = 2$. Then by (\ref{pro1eq}), we have $\overset{\rightarrow}{srvc}(D) \ge 2$ and $\textup{diam}(D)\le 3$. We may take a rainbow vertex-connected colouring of $D$, using $\overset{\rightarrow}{rvc}(D) = 2$ colours. Let $u,v \in V(D)$. If $d(u,v)\in\{1,2\}$, then any $u-v$ geodesic is clearly rainbow. If $d(u, v) = 3$, then since any $u-v$ path of length at least 4 cannot be rainbow, there must exist a rainbow $u-v$ path of length 3, which is also a $u-v$ geodesic. Thus, the vertex-colouring is also strongly rainbow vertex-connected. We have $\overset{\rightarrow}{srvc}(D)\le 2$, so that $\overset{\rightarrow}{srvc}(D) = 2$.\\[1ex]
\indent Finally, (b)(iii) was proved in \cite{JJ2015}, Theorem 2.2(b). To see the ``moreover'' part, suppose that  $\overset{\rightarrow}{rc}(D)=2$. Then diam$(D)\le\overset{\rightarrow}{rc}(D)=2$, and again (a) implies diam$(D)=2$.
\end{proof}

We remark that in Theorem \ref{pro2}(b), no other implication exists between the conditions of (i) to (iii). Obviously, the conditions of (i) and those of (ii) are mutually exclusive. Thus by the ``moreover'' part, the conditions of (ii) and those of (iii) are also mutually exclusive. Now, there are infinitely many examples of digraphs $D$ where the conditions of (i) hold, but those of (iii) do not hold. For example, let $u$ be a vertex of the directed cycle $\rC_3$, and let $D_n$ be the digraph on $n\ge 3$ vertices, obtained from $\rC_3$ by expanding $u$ to $\lrK_{n-2}$. Then, we have $\overset{\rightarrow}{srvc}(D_n)=\overset{\rightarrow}{rvc}(D_n)=1$ and diam$(D_n)=2$, but $\overset{\rightarrow}{src}(D_n)=\overset{\rightarrow}{rc}(D_n)=3$.

Alva-Samos and Montellano-Ballesteros \cite{JJ2015} remarked that for a connected graph $G$,
\begin{equation}
\overset{\rightarrow}{rc}(\lrG)\le rc(G)\quad\textup{and}\quad\overset{\rightarrow}{src}(\lrG)\le src(G).\label{ASMBrmk}
\end{equation}
Furthermore, for each inequality, there are infinitely many graphs where equality holds, and also with the difference between the two parameters arbitrarily large. For example, for $n\ge 4$, we have $rc(C_n)=\overset{\rightarrow}{rc}(\lrC_n)=src(C_n)=\overset{\rightarrow}{src}(\lrC_n)=\lceil\frac{n}{2}\rceil$. Also, for $n\ge 2$, we have $\overset{\rightarrow}{rc}(\lrK_{1,n})=\overset{\rightarrow}{src}(\lrK_{1,n})=2$ and $rc(K_{1,n})=src(K_{1,n})=n$, where $K_{1,n}$ is the star with $n$ edges. However, for rainbow vertex-connection, we have the following proposition.

\begin{prop}\label{pro3}
For a connected graph $G$, we have $rvc(G)=\overset{\rightarrow}{rvc}(\lrG)$ and $srvc(G)=\overset{\rightarrow}{srvc}(\lrG)$.
\end{prop}

\begin{proof}
First, observe that for vertices $u,v\in V(G)=V(\lrG)$, a $u-v$ geodesic in $G$ corresponds to $u-v$ and $v-u$ geodesics in $\lrG$. The proposition follows since a (strongly) rainbow vertex-connected colouring of $G$ is also a (strongly) rainbow vertex-connected colouring of $\lrG$, and conversely.
\end{proof}

Next, Alva-Samos and Montellano-Ballesteros \cite{JJ2015} remarked that, if $D$ and $H$ are strongly connected digraphs such that $H$ is a spanning subdigraph of $D$, then $\overset{\rightarrow}{rc}(D)\le \overset{\rightarrow}{rc}(H)$. The analogous statement for the rainbow vertex-connection number also clearly holds.

\begin{prop}\label{pro4}
Let $D$ and $H$ be strongly connected digraphs such that, $H$ is a spanning subdigraph of $D$. Then $\overset{\rightarrow}{rvc}(D)\le \overset{\rightarrow}{rvc}(H)$.
\end{prop}

Alva-Samos and Montellano-Ballesteros also showed in \cite{JJ2015}, Lemma 1, that the same is not true for the strong rainbow connection number, when they presented an example of such digraphs $D$ and $H$ with $\overset{\rightarrow}{src}(D)>\overset{\rightarrow}{src}(H)$. However, in their example, the arc-colouring of the digraph $H$ is \emph{not} strongly rainbow connected. In the next lemma, we will rectify this problem, and also show the analogous result for the strong rainbow vertex-connection number.


\begin{lem}\label{lem5}
There are strongly connected digraphs $D$ and $H$ such that, $H$ is a spanning subdigraph of $D$, and $\overset{\rightarrow}{src}(D)>\overset{\rightarrow}{src}(H)$. A similar statement holds for the function $\overset{\rightarrow}{srvc}$.
\end{lem}

\begin{proof}
Let $H_1$ (resp.~$H_2$) be the digraph consisting of the solid arcs as in Figure 1(a) (resp.~(b)), and $D_1$ (resp.~$D_2$) be the digraph obtained by adding the dotted arc.

It is not hard to see that the arc-colouring for $H_1$ in Figure 1(a) is strongly rainbow connected, where for the eight pairs of symmetric arcs in the middle, the two arcs in each pair have the same colour. Thus $\overset{\rightarrow}{src}(H_1)\le 6$. In fact, we have $\overset{\rightarrow}{src}(H_1)=6$, since $\overset{\rightarrow}{src}(H_1)\ge\diam(H_1)=6$. Now, we will show that $\overset{\rightarrow}{src}(D_1)\geq 7$. Suppose that we have a strongly rainbow connected colouring $c$ of $D_1$, using at most six colours, say colours $1,2,3,4,5,6$. For $1 \le i \le 4$, we see that the arcs $u_iv_i$ must have distinct colours. Otherwise we can find two vertices such that any geodesic in $D_1$ connecting them is not rainbow. Hence, we may assume that $c(u_iv_i) = i$, for $1 \le i \le 4$. For the same reason, the arcs $v_1x,xy$ and $yu_3$ must each have a colour different from $1,2,3,4$, and they must also have distinct colours, since $v_1xyu_3$ is the unique $v_1-u_3$ geodesic in $D_1$. However, only the colours $5$ and $6$ are available for these three arcs, a contradiction. Hence $\overset{\rightarrow}{src}(D_1)\geq 7$, and we have $\overset{\rightarrow}{src}(D_1)>\overset{\rightarrow}{src}(H_1)$.

We may apply a fairly similar argument for the digraphs $D_2$ and $H_2$. Again, we can easily check that the vertex-colouring for $H_2$ in Figure 1(b) is strongly rainbow vertex-connected, and that $\diam(H_2)=9$. Thus $\overset{\rightarrow}{srvc}(H_2)=8$. Now suppose that there is a strongly rainbow vertex-connected colouring of $D_2$, using colours $1,2,\dots,8$. Notice that any two cut-vertices must receive distinct colours, and thus without loss of generality, we may assume that $w_1,w_2,w_3,w_4,x,y$ have colours $1,2,3,4,5,6$. Next, we see that $v_1,u_3$ and $z$ must each have a colour different from $1,2,\dots,6$, otherwise we can find two vertices such that any geodesic in $D_2$ connecting them is not rainbow. Moreover, $v_1,u_3$ and $z$ must have distinct colours, since $u_1v_1w_1xzyw_3u_3v_3$ is the unique $u_1-v_3$ geodesic in $D_2$. However, only the colours $7$ and $8$ are available for these three vertices, a contradiction. Hence $\overset{\rightarrow}{srvc}(D_2)\geq 9>8=\overset{\rightarrow}{srvc}(H_2)$.\\[2ex]
\[ \unit = 0.6cm
\pt{0}{0.8}\pt{0}{-0.8}\pt{0}{2.4}\pt{0}{-2.4}\pt{3.9}{0}\pt{-3.9}{0}
\pt{4.9}{2.4}\pt{6.3}{1}\pt{4.9}{-2.4}\pt{6.3}{-1}\pt{-4.9}{2.4}\pt{-6.3}{1}\pt{-4.9}{-2.4}\pt{-6.3}{-1}
\varline{500}{0.6}
\bez{0}{2.4}{2.17}{1.88}{3.9}{0}\bez{0}{2.4}{1.83}{0.92}{3.9}{0}
\bez{0}{0.8}{2.06}{0.9}{3.9}{0}\bez{0}{0.8}{1.94}{0.01}{3.9}{0}
\bez{0}{-2.4}{2.17}{-1.88}{3.9}{0}\bez{0}{-2.4}{1.83}{-0.92}{3.9}{0}
\bez{0}{-0.8}{2.06}{-0.9}{3.9}{0}\bez{0}{-0.8}{1.94}{-0.01}{3.9}{0}
\bez{0}{2.4}{-2.17}{1.88}{-3.9}{0}\bez{0}{2.4}{-1.83}{0.92}{-3.9}{0}
\bez{0}{0.8}{-2.06}{0.9}{-3.9}{0}\bez{0}{0.8}{-1.94}{0.01}{-3.9}{0}
\bez{0}{-2.4}{-2.17}{-1.88}{-3.9}{0}\bez{0}{-2.4}{-1.83}{-0.92}{-3.9}{0}
\bez{0}{-0.8}{-2.06}{-0.9}{-3.9}{0}\bez{0}{-0.8}{-1.94}{-0.01}{-3.9}{0}
\dl{3.9}{0}{4.9}{2.4}\dl{3.9}{0}{6.3}{1}\dl{4.9}{2.4}{6.3}{1}
\dl{3.9}{0}{4.9}{-2.4}\dl{3.9}{0}{6.3}{-1}\dl{4.9}{-2.4}{6.3}{-1}
\dl{-3.9}{0}{-4.9}{2.4}\dl{-3.9}{0}{-6.3}{1}\dl{-4.9}{2.4}{-6.3}{1}
\dl{-3.9}{0}{-4.9}{-2.4}\dl{-3.9}{0}{-6.3}{-1}\dl{-4.9}{-2.4}{-6.3}{-1}
\varline{850}{5.5}
\bez{-3.9}{0}{0}{-5.84}{3.9}{0}
\thnline
\dl{2.22}{1.44}{2.07}{1.67}\dl{2.18}{1.47}{2.07}{1.67}\dl{2.15}{1.49}{2.07}{1.67}\dl{2.12}{1.51}{2.07}{1.67}\dl{2.1}{1.53}{2.07}{1.67}\dl{2.08}{1.54}{2.07}{1.67}\dl{2.06}{1.55}{2.07}{1.67}
\dl{2.22}{1.44}{1.95}{1.48}\dl{2.18}{1.47}{1.95}{1.48}\dl{2.15}{1.49}{1.95}{1.48}\dl{2.12}{1.51}{1.95}{1.48}\dl{2.1}{1.53}{1.95}{1.48}\dl{2.08}{1.54}{1.95}{1.48}\dl{2.06}{1.55}{1.95}{1.48}
\dl{1.78}{1.12}{1.95}{0.9}\dl{1.82}{1.09}{1.95}{0.9}\dl{1.86}{1.07}{1.95}{0.9}\dl{1.89}{1.05}{1.95}{0.9}\dl{1.92}{1.04}{1.95}{0.9}\dl{1.94}{1.03}{1.95}{0.9}\dl{1.96}{1.02}{1.95}{0.9}
\dl{1.78}{1.12}{2.05}{1.09}\dl{1.82}{1.09}{2.05}{1.09}\dl{1.86}{1.07}{2.05}{1.09}\dl{1.89}{1.05}{2.05}{1.09}\dl{1.92}{1.04}{2.05}{1.09}\dl{1.94}{1.03}{2.05}{1.09}\dl{1.96}{1.02}{2.05}{1.09}\dl{1.96}{1.02}{2.05}{1.09}
\dl{2.12}{0.63}{1.89}{0.8}\dl{2.07}{0.64}{1.89}{0.8}\dl{2.03}{0.65}{1.89}{0.8}\dl{2}{0.65}{1.89}{0.8}\dl{1.97}{0.66}{1.89}{0.8}\dl{1.94}{0.66}{1.89}{0.8}\dl{1.92}{0.67}{1.89}{0.8}
\dl{2.12}{0.63}{1.84}{0.57}\dl{2.07}{0.64}{1.84}{0.57}\dl{2.03}{0.65}{1.84}{0.57}\dl{2}{0.65}{1.84}{0.57}\dl{1.97}{0.66}{1.84}{0.57}\dl{1.94}{0.66}{1.84}{0.57}\dl{1.92}{0.67}{1.84}{0.57}
\dl{1.81}{0.23}{2.04}{0.06}\dl{1.86}{0.22}{2.04}{0.06}\dl{1.9}{0.21}{2.04}{0.06}\dl{1.93}{0.21}{2.04}{0.06}\dl{1.96}{0.2}{2.04}{0.06}\dl{1.99}{0.2}{2.04}{0.06}\dl{2.01}{0.19}{2.04}{0.06}
\dl{1.81}{0.23}{2.09}{0.3}\dl{1.86}{0.22}{2.09}{0.3}\dl{1.9}{0.21}{2.09}{0.3}\dl{1.93}{0.21}{2.09}{0.3}\dl{1.96}{0.2}{2.09}{0.3}\dl{1.99}{0.2}{2.09}{0.3}\dl{2.01}{0.19}{2.09}{0.3}
\dl{2.03}{-0.18}{1.8}{-0.35}\dl{1.98}{-0.19}{1.8}{-0.35}\dl{1.94}{-0.2}{1.8}{-0.35}\dl{1.91}{-0.2}{1.8}{-0.35}\dl{1.88}{-0.21}{1.8}{-0.35}\dl{1.85}{-0.21}{1.8}{-0.35}\dl{1.83}{-0.22}{1.8}{-0.35}
\dl{2.03}{-0.18}{1.75}{-0.12}\dl{1.98}{-0.19}{1.75}{-0.12}\dl{1.94}{-0.2}{1.75}{-0.12}\dl{1.91}{-0.2}{1.75}{-0.12}\dl{1.88}{-0.21}{1.75}{-0.12}\dl{1.85}{-0.21}{1.75}{-0.12}\dl{1.83}{-0.22}{1.75}{-0.12}
\dl{1.88}{-0.68}{2.11}{-0.51}\dl{1.93}{-0.67}{2.11}{-0.51}\dl{1.97}{-0.66}{2.11}{-0.51}\dl{2}{-0.66}{2.11}{-0.51}\dl{2.03}{-0.65}{2.11}{-0.51}\dl{2.06}{-0.65}{2.11}{-0.51}\dl{2.08}{-0.64}{2.11}{-0.51}
\dl{1.88}{-0.68}{2.16}{-0.74}\dl{1.93}{-0.67}{2.16}{-0.74}\dl{1.97}{-0.66}{2.16}{-0.74}\dl{2}{-0.66}{2.16}{-0.74}\dl{2.03}{-0.65}{2.16}{-0.74}\dl{2.06}{-0.65}{2.16}{-0.74}\dl{2.08}{-0.64}{2.16}{-0.74}
\dl{2.03}{-1.57}{2.18}{-1.34}\dl{2.07}{-1.54}{2.18}{-1.34}\dl{2.1}{-1.52}{2.18}{-1.34}\dl{2.13}{-1.5}{2.18}{-1.34}\dl{2.15}{-1.48}{2.18}{-1.34}\dl{2.17}{-1.47}{2.18}{-1.34}\dl{2.19}{-1.46}{2.18}{-1.34}
\dl{2.03}{-1.57}{2.3}{-1.53}\dl{2.07}{-1.54}{2.3}{-1.53}\dl{2.1}{-1.52}{2.3}{-1.53}\dl{2.13}{-1.5}{2.3}{-1.53}\dl{2.15}{-1.48}{2.3}{-1.53}\dl{2.17}{-1.47}{2.3}{-1.53}\dl{2.19}{-1.46}{2.3}{-1.53}
\dl{1.96}{-1.02}{1.83}{-1.23}\dl{1.93}{-1.02}{1.83}{-1.23}\dl{1.9}{-1.04}{1.83}{-1.23}\dl{1.88}{-1.06}{1.83}{-1.23}\dl{1.86}{-1.08}{1.83}{-1.23}\dl{1.84}{-1.09}{1.83}{-1.23}\dl{1.82}{-1.1}{1.83}{-1.23}\dl{1.8}{-1.11}{1.83}{-1.23}
\dl{1.96}{-1.02}{1.71}{-1.04}\dl{1.93}{-1.02}{1.71}{-1.04}\dl{1.9}{-1.04}{1.71}{-1.04}\dl{1.88}{-1.06}{1.71}{-1.04}\dl{1.86}{-1.08}{1.71}{-1.04}\dl{1.84}{-1.09}{1.71}{-1.04}\dl{1.82}{-1.1}{1.71}{-1.04}\dl{1.8}{-1.11}{1.71}{-1.04}
\dl{-2.22}{-1.44}{-2.07}{-1.67}\dl{-2.18}{-1.47}{-2.07}{-1.67}\dl{-2.15}{-1.49}{-2.07}{-1.67}\dl{-2.12}{-1.51}{-2.07}{-1.67}\dl{-2.1}{-1.53}{-2.07}{-1.67}\dl{-2.08}{-1.54}{-2.07}{-1.67}\dl{-2.06}{-1.55}{-2.07}{-1.67}
\dl{-2.22}{-1.44}{-1.95}{-1.48}\dl{-2.18}{-1.47}{-1.95}{-1.48}\dl{-2.15}{-1.49}{-1.95}{-1.48}\dl{-2.12}{-1.51}{-1.95}{-1.48}\dl{-2.1}{-1.53}{-1.95}{-1.48}\dl{-2.08}{-1.54}{-1.95}{-1.48}\dl{-2.06}{-1.55}{-1.95}{-1.48}
\dl{-1.78}{-1.12}{-1.95}{-0.9}\dl{-1.82}{-1.09}{-1.95}{-0.9}\dl{-1.86}{-1.07}{-1.95}{-0.9}\dl{-1.89}{-1.05}{-1.95}{-0.9}\dl{-1.92}{-1.04}{-1.95}{-0.9}\dl{-1.94}{-1.03}{-1.95}{-0.9}\dl{-1.96}{-1.02}{-1.95}{-0.9}
\dl{-1.78}{-1.12}{-2.05}{-1.09}\dl{-1.82}{-1.09}{-2.05}{-1.09}\dl{-1.86}{-1.07}{-2.05}{-1.09}\dl{-1.89}{-1.05}{-2.05}{-1.09}\dl{-1.92}{-1.04}{-2.05}{-1.09}\dl{-1.94}{-1.03}{-2.05}{-1.09}\dl{-1.96}{-1.02}{-2.05}{-1.09}\dl{-1.96}{-1.02}{-2.05}{-1.09}
\dl{-2.12}{-0.63}{-1.89}{-0.8}\dl{-2.07}{-0.64}{-1.89}{-0.8}\dl{-2.03}{-0.65}{-1.89}{-0.8}\dl{-2}{-0.65}{-1.89}{-0.8}\dl{-1.97}{-0.66}{-1.89}{-0.8}\dl{-1.94}{-0.66}{-1.89}{-0.8}\dl{-1.92}{-0.67}{-1.89}{-0.8}
\dl{-2.12}{-0.63}{-1.84}{-0.57}\dl{-2.07}{-0.64}{-1.84}{-0.57}\dl{-2.03}{-0.65}{-1.84}{-0.57}\dl{-2}{-0.65}{-1.84}{-0.57}\dl{-1.97}{-0.66}{-1.84}{-0.57}\dl{-1.94}{-0.66}{-1.84}{-0.57}\dl{-1.92}{-0.67}{-1.84}{-0.57}
\dl{-1.81}{-0.23}{-2.04}{-0.06}\dl{-1.86}{-0.22}{-2.04}{-0.06}\dl{-1.9}{-0.21}{-2.04}{-0.06}\dl{-1.93}{-0.21}{-2.04}{-0.06}\dl{-1.96}{-0.2}{-2.04}{-0.06}\dl{-1.99}{-0.2}{-2.04}{-0.06}\dl{-2.01}{-0.19}{-2.04}{-0.06}
\dl{-1.81}{-0.23}{-2.09}{-0.3}\dl{-1.86}{-0.22}{-2.09}{-0.3}\dl{-1.9}{-0.21}{-2.09}{-0.3}\dl{-1.93}{-0.21}{-2.09}{-0.3}\dl{-1.96}{-0.2}{-2.09}{-0.3}\dl{-1.99}{-0.2}{-2.09}{-0.3}\dl{-2.01}{-0.19}{-2.09}{-0.3}
\dl{-2.03}{0.18}{-1.8}{0.35}\dl{-1.98}{0.19}{-1.8}{0.35}\dl{-1.94}{0.2}{-1.8}{0.35}\dl{-1.91}{0.2}{-1.8}{0.35}\dl{-1.88}{0.21}{-1.8}{0.35}\dl{-1.85}{0.21}{-1.8}{0.35}\dl{-1.83}{0.22}{-1.8}{0.35}
\dl{-2.03}{0.18}{-1.75}{0.12}\dl{-1.98}{0.19}{-1.75}{0.12}\dl{-1.94}{0.2}{-1.75}{0.12}\dl{-1.91}{0.2}{-1.75}{0.12}\dl{-1.88}{0.21}{-1.75}{0.12}\dl{-1.85}{0.21}{-1.75}{0.12}\dl{-1.83}{0.22}{-1.75}{0.12}
\dl{-1.88}{0.68}{-2.11}{0.51}\dl{-1.93}{0.67}{-2.11}{0.51}\dl{-1.97}{0.66}{-2.11}{0.51}\dl{-2}{0.66}{-2.11}{0.51}\dl{-2.03}{0.65}{-2.11}{0.51}\dl{-2.06}{0.65}{-2.11}{0.51}\dl{-2.08}{0.64}{-2.11}{0.51}
\dl{-1.88}{0.68}{-2.16}{0.74}\dl{-1.93}{0.67}{-2.16}{0.74}\dl{-1.97}{0.66}{-2.16}{0.74}\dl{-2}{0.66}{-2.16}{0.74}\dl{-2.03}{0.65}{-2.16}{0.74}\dl{-2.06}{0.65}{-2.16}{0.74}\dl{-2.08}{0.64}{-2.16}{0.74}
\dl{-2.03}{1.57}{-2.18}{1.34}\dl{-2.07}{1.54}{-2.18}{1.34}\dl{-2.1}{1.52}{-2.18}{1.34}\dl{-2.13}{1.5}{-2.18}{1.34}\dl{-2.15}{1.48}{-2.18}{1.34}\dl{-2.17}{1.47}{-2.18}{1.34}\dl{-2.19}{1.46}{-2.18}{1.34}
\dl{-2.03}{1.57}{-2.3}{1.53}\dl{-2.07}{1.54}{-2.3}{1.53}\dl{-2.1}{1.52}{-2.3}{1.53}\dl{-2.13}{1.5}{-2.3}{1.53}\dl{-2.15}{1.48}{-2.3}{1.53}\dl{-2.17}{1.47}{-2.3}{1.53}\dl{-2.19}{1.46}{-2.3}{1.53}
\dl{-1.96}{1.02}{-1.83}{1.23}\dl{-1.93}{1.02}{-1.83}{1.23}\dl{-1.9}{1.04}{-1.83}{1.23}\dl{-1.88}{1.06}{-1.83}{1.23}\dl{-1.86}{1.08}{-1.83}{1.23}\dl{-1.84}{1.09}{-1.83}{1.23}\dl{-1.82}{1.1}{-1.83}{1.23}\dl{-1.8}{1.11}{-1.83}{1.23}
\dl{-1.96}{1.02}{-1.71}{1.04}\dl{-1.93}{1.02}{-1.71}{1.04}\dl{-1.9}{1.04}{-1.71}{1.04}\dl{-1.88}{1.06}{-1.71}{1.04}\dl{-1.86}{1.08}{-1.71}{1.04}\dl{-1.84}{1.09}{-1.71}{1.04}\dl{-1.82}{1.1}{-1.71}{1.04}\dl{-1.8}{1.11}{-1.71}{1.04}
%
\dl{4.38}{1.13}{4.45}{1.05}\dl{4.39}{1.16}{4.45}{1.05}\dl{4.4}{1.19}{4.45}{1.05}\dl{4.41}{1.22}{4.45}{1.05}\dl{4.43}{1.27}{4.45}{1.05}\dl{4.45}{1.32}{4.45}{1.05}\dl{4.47}{1.37}{4.45}{1.05}
\dl{4.38}{1.13}{4.26}{1.14}\dl{4.39}{1.16}{4.26}{1.14}\dl{4.4}{1.19}{4.26}{1.14}\dl{4.41}{1.22}{4.26}{1.14}\dl{4.43}{1.27}{4.26}{1.14}\dl{4.45}{1.32}{4.26}{1.14}\dl{4.47}{1.37}{4.26}{1.14}
\dl{4.93}{0.43}{5.27}{0.42}\dl{5}{0.45}{5.27}{0.42}\dl{5.05}{0.47}{5.27}{0.42}\dl{5.08}{0.49}{5.27}{0.42}\dl{5.11}{0.5}{5.27}{0.42}\dl{5.14}{0.51}{5.27}{0.42}\dl{5.17}{0.52}{5.27}{0.42}
\dl{4.93}{0.43}{5.17}{0.67}\dl{5}{0.45}{5.17}{0.67}\dl{5.05}{0.47}{5.17}{0.67}\dl{5.08}{0.49}{5.17}{0.67}\dl{5.11}{0.5}{5.17}{0.67}\dl{5.14}{0.51}{5.17}{0.67}\dl{5.17}{0.52}{5.17}{0.67}
\dl{5.73}{1.57}{5.43}{1.7}\dl{5.67}{1.63}{5.43}{1.7}\dl{5.63}{1.67}{5.43}{1.7}\dl{5.6}{1.7}{5.43}{1.7}\dl{5.57}{1.73}{5.43}{1.7}\dl{5.55}{1.75}{5.43}{1.7}\dl{5.53}{1.77}{5.43}{1.7}
\dl{5.73}{1.57}{5.6}{1.88}\dl{5.67}{1.63}{5.6}{1.88}\dl{5.63}{1.67}{5.6}{1.88}\dl{5.6}{1.7}{5.6}{1.88}\dl{5.57}{1.73}{5.6}{1.88}\dl{5.55}{1.75}{5.6}{1.88}\dl{5.53}{1.77}{5.6}{1.88}
\dl{4.42}{-1.27}{4.35}{-1.35}\dl{4.41}{-1.24}{4.35}{-1.35}\dl{4.4}{-1.21}{4.35}{-1.35}\dl{4.39}{-1.18}{4.35}{-1.35}\dl{4.37}{-1.13}{4.35}{-1.35}\dl{4.35}{-1.08}{4.35}{-1.35}\dl{4.33}{-1.03}{4.35}{-1.35}
\dl{4.42}{-1.27}{4.54}{-1.26}\dl{4.41}{-1.24}{4.54}{-1.26}\dl{4.4}{-1.21}{4.54}{-1.26}\dl{4.39}{-1.18}{4.54}{-1.26}\dl{4.37}{-1.13}{4.54}{-1.26}\dl{4.35}{-1.08}{4.54}{-1.26}\dl{4.33}{-1.03}{4.54}{-1.26}
\dl{5.27}{-0.57}{4.93}{-0.58}\dl{5.2}{-0.55}{4.93}{-0.58}\dl{5.15}{-0.53}{4.93}{-0.58}\dl{5.12}{-0.51}{4.93}{-0.58}\dl{5.09}{-0.5}{4.93}{-0.58}\dl{5.06}{-0.49}{4.93}{-0.58}\dl{5.03}{-0.48}{4.93}{-0.58}
\dl{5.27}{-0.57}{5.03}{-0.33}\dl{5.2}{-0.55}{5.03}{-0.33}\dl{5.15}{-0.53}{5.03}{-0.33}\dl{5.12}{-0.51}{5.03}{-0.33}\dl{5.09}{-0.5}{5.03}{-0.33}\dl{5.06}{-0.49}{5.03}{-0.33}\dl{5.03}{-0.48}{5.03}{-0.33}
\dl{5.47}{-1.83}{5.77}{-1.7}\dl{5.53}{-1.77}{5.77}{-1.7}\dl{5.57}{-1.73}{5.77}{-1.7}\dl{5.6}{-1.7}{5.77}{-1.7}\dl{5.63}{-1.67}{5.77}{-1.7}\dl{5.65}{-1.65}{5.77}{-1.7}\dl{5.67}{-1.63}{5.77}{-1.7}
\dl{5.47}{-1.83}{5.6}{-1.52}\dl{5.53}{-1.77}{5.6}{-1.52}\dl{5.57}{-1.73}{5.6}{-1.52}\dl{5.6}{-1.7}{5.6}{-1.52}\dl{5.63}{-1.67}{5.6}{-1.52}\dl{5.65}{-1.65}{5.6}{-1.52}\dl{5.67}{-1.63}{5.6}{-1.52}
\dl{-4.42}{1.27}{-4.35}{1.35}\dl{-4.41}{1.24}{-4.35}{1.35}\dl{-4.4}{1.21}{-4.35}{1.35}\dl{-4.39}{1.18}{-4.35}{1.35}\dl{-4.37}{1.13}{-4.35}{1.35}\dl{-4.35}{1.08}{-4.35}{1.35}\dl{-4.33}{1.03}{-4.35}{1.35}
\dl{-4.42}{1.27}{-4.54}{1.26}\dl{-4.41}{1.24}{-4.54}{1.26}\dl{-4.4}{1.21}{-4.54}{1.26}\dl{-4.39}{1.18}{-4.54}{1.26}\dl{-4.37}{1.13}{-4.54}{1.26}\dl{-4.35}{1.08}{-4.54}{1.26}\dl{-4.33}{1.03}{-4.54}{1.26}
\dl{-5.27}{0.57}{-4.93}{0.58}\dl{-5.2}{0.55}{-4.93}{0.58}\dl{-5.15}{0.53}{-4.93}{0.58}\dl{-5.12}{0.51}{-4.93}{0.58}\dl{-5.09}{0.5}{-4.93}{0.58}\dl{-5.06}{0.49}{-4.93}{0.58}\dl{-5.03}{0.48}{-4.93}{0.58}
\dl{-5.27}{0.57}{-5.03}{0.33}\dl{-5.2}{0.55}{-5.03}{0.33}\dl{-5.15}{0.53}{-5.03}{0.33}\dl{-5.12}{0.51}{-5.03}{0.33}\dl{-5.09}{0.5}{-5.03}{0.33}\dl{-5.06}{0.49}{-5.03}{0.33}\dl{-5.03}{0.48}{-5.03}{0.33}
\dl{-5.47}{1.83}{-5.77}{1.7}\dl{-5.53}{1.77}{-5.77}{1.7}\dl{-5.57}{1.73}{-5.77}{1.7}\dl{-5.6}{1.7}{-5.77}{1.7}\dl{-5.63}{1.67}{-5.77}{1.7}\dl{-5.65}{1.65}{-5.77}{1.7}\dl{-5.67}{1.63}{-5.77}{1.7}
\dl{-5.47}{1.83}{-5.6}{1.52}\dl{-5.53}{1.77}{-5.6}{1.52}\dl{-5.57}{1.73}{-5.6}{1.52}\dl{-5.6}{1.7}{-5.6}{1.52}\dl{-5.63}{1.67}{-5.6}{1.52}\dl{-5.65}{1.65}{-5.6}{1.52}\dl{-5.67}{1.63}{-5.6}{1.52}
\dl{-4.38}{-1.13}{-4.45}{-1.05}\dl{-4.39}{-1.16}{-4.45}{-1.05}\dl{-4.4}{-1.19}{-4.45}{-1.05}\dl{-4.41}{-1.22}{-4.45}{-1.05}\dl{-4.43}{-1.27}{-4.45}{-1.05}\dl{-4.45}{-1.32}{-4.45}{-1.05}\dl{-4.47}{-1.37}{-4.45}{-1.05}
\dl{-4.38}{-1.13}{-4.26}{-1.14}\dl{-4.39}{-1.16}{-4.26}{-1.14}\dl{-4.4}{-1.19}{-4.26}{-1.14}\dl{-4.41}{-1.22}{-4.26}{-1.14}\dl{-4.43}{-1.27}{-4.26}{-1.14}\dl{-4.45}{-1.32}{-4.26}{-1.14}\dl{-4.47}{-1.37}{-4.26}{-1.14}
\dl{-4.93}{-0.43}{-5.27}{-0.42}\dl{-5}{-0.45}{-5.27}{-0.42}\dl{-5.05}{-0.47}{-5.27}{-0.42}\dl{-5.08}{-0.49}{-5.27}{-0.42}\dl{-5.11}{-0.5}{-5.27}{-0.42}\dl{-5.14}{-0.51}{-5.27}{-0.42}\dl{-5.17}{-0.52}{-5.27}{-0.42}
\dl{-4.93}{-0.43}{-5.17}{-0.67}\dl{-5}{-0.45}{-5.17}{-0.67}\dl{-5.05}{-0.47}{-5.17}{-0.67}\dl{-5.08}{-0.49}{-5.17}{-0.67}\dl{-5.11}{-0.5}{-5.17}{-0.67}\dl{-5.14}{-0.51}{-5.17}{-0.67}\dl{-5.17}{-0.52}{-5.17}{-0.67}
\dl{-5.73}{-1.57}{-5.43}{-1.7}\dl{-5.67}{-1.63}{-5.43}{-1.7}\dl{-5.63}{-1.67}{-5.43}{-1.7}\dl{-5.6}{-1.7}{-5.43}{-1.7}\dl{-5.57}{-1.73}{-5.43}{-1.7}\dl{-5.55}{-1.75}{-5.43}{-1.7}\dl{-5.53}{-1.77}{-5.43}{-1.7}
\dl{-5.73}{-1.57}{-5.6}{-1.88}\dl{-5.67}{-1.63}{-5.6}{-1.88}\dl{-5.63}{-1.67}{-5.6}{-1.88}\dl{-5.6}{-1.7}{-5.6}{-1.88}\dl{-5.57}{-1.73}{-5.6}{-1.88}\dl{-5.55}{-1.75}{-5.6}{-1.88}\dl{-5.53}{-1.77}{-5.6}{-1.88}
\dl{0.13}{-2.91}{-0.07}{-2.91}
\dl{0.13}{-2.91}{-0.13}{-2.81}\dl{0.09}{-2.91}{-0.13}{-2.81}\dl{0.05}{-2.91}{-0.13}{-2.81}\dl{0.01}{-2.91}{-0.13}{-2.81}\dl{-0.03}{-2.91}{-0.13}{-2.81}\dl{-0.07}{-2.91}{-0.13}{-2.81}
\dl{0.13}{-2.91}{-0.13}{-3.01}\dl{0.09}{-2.91}{-0.13}{-3.01}\dl{0.05}{-2.91}{-0.13}{-3.01}\dl{0.01}{-2.91}{-0.13}{-3.01}\dl{-0.03}{-2.91}{-0.13}{-3.01}\dl{-0.07}{-2.91}{-0.13}{-3.01}
\point{-7}{0.85}{\small $u_1$}\point{-5.2}{-2.9}{\small $u_2$}\point{6.45}{-1.1}{\small $u_4$}\point{4.67}{2.63}{\small $u_3$}
\point{-5.2}{2.63}{\small $v_1$}\point{-7}{-1.1}{\small $v_2$}\point{4.67}{-2.9}{\small $v_4$}\point{6.45}{0.85}{\small $v_3$}
\point{-3.95}{0.42}{\small $x$}\point{3.7}{0.42}{\small $y$}
\ptlu{-5.8}{1.6}{1}\ptlr{-4.38}{1.3}{5}\ptld{-5.45}{0.7}{6}
\ptld{-5.8}{-1.6}{2}\ptlr{-4.38}{-1.3}{6}\ptlu{-5.45}{-0.7}{5}
\ptlu{5.8}{1.6}{3}\ptll{4.33}{1.3}{6}\ptld{5.45}{0.7}{5}
\ptld{5.8}{-1.6}{4}\ptll{4.33}{-1.3}{5}\ptlu{5.45}{-0.7}{6}
\ptld{-1.15}{2.6}{3}\ptld{-1.15}{1.35}{4}\ptlu{-1.15}{-1.62}{4}\ptlu{-1.15}{-0.5}{3}
\ptld{1.15}{2.6}{1}\ptld{1.15}{1.35}{1}\ptlu{1.15}{-1.62}{2}\ptlu{1.15}{-0.5}{2}
\ptlu{-10.7}{2.1}{\textup{\normalsize (a)}}
\]\\[-4ex]

\[ \unit = 0.6cm
\pt{1.3}{0.8}\pt{1.3}{-0.8}\pt{1.3}{2.4}\pt{1.3}{-2.4}\pt{-1.3}{0.8}\pt{-1.3}{-0.8}\pt{-1.3}{2.4}\pt{-1.3}{-2.4}
\pt{3.9}{0}\pt{-3.9}{0}\pt{6.15}{1.3}\pt{6.15}{-1.3}\pt{-6.15}{1.3}\pt{-6.15}{-1.3}
\pt{8.28}{2.79}\pt{8.62}{0.5}\pt{8.28}{-2.79}\pt{8.62}{-0.5}\pt{-8.28}{2.79}\pt{-8.62}{0.5}\pt{-8.28}{-2.79}\pt{-8.62}{-0.5}
\varline{500}{0.6}
\bez{-1.3}{2.4}{0}{2.8}{1.3}{2.4}\bez{-1.3}{2.4}{0}{2}{1.3}{2.4}
\bez{-1.3}{-2.4}{0}{-2.8}{1.3}{-2.4}\bez{-1.3}{-2.4}{0}{-2}{1.3}{-2.4}
\bez{-1.3}{0.8}{0}{1.2}{1.3}{0.8}\bez{-1.3}{0.8}{0}{0.4}{1.3}{0.8}
\bez{-1.3}{-0.8}{0}{-1.2}{1.3}{-0.8}\bez{-1.3}{-0.8}{0}{-0.4}{1.3}{-0.8}
\bez{1.3}{2.4}{2.77}{1.48}{3.9}{0}\bez{1.3}{2.4}{2.43}{0.92}{3.9}{0}
\bez{1.3}{0.8}{2.66}{0.7}{3.9}{0}\bez{1.3}{0.8}{2.54}{0.1}{3.9}{0}
\bez{1.3}{-2.4}{2.77}{-1.48}{3.9}{0}\bez{1.3}{-2.4}{2.43}{-0.92}{3.9}{0}
\bez{1.3}{-0.8}{2.66}{-0.7}{3.9}{0}\bez{1.3}{-0.8}{2.54}{-0.1}{3.9}{0}
\bez{-1.3}{2.4}{-2.77}{1.48}{-3.9}{0}\bez{-1.3}{2.4}{-2.43}{0.92}{-3.9}{0}
\bez{-1.3}{0.8}{-2.66}{0.7}{-3.9}{0}\bez{-1.3}{0.8}{-2.54}{0.1}{-3.9}{0}
\bez{-1.3}{-2.4}{-2.77}{-1.48}{-3.9}{0}\bez{-1.3}{-2.4}{-2.43}{-0.92}{-3.9}{0}
\bez{-1.3}{-0.8}{-2.66}{-0.7}{-3.9}{0}\bez{-1.3}{-0.8}{-2.54}{-0.1}{-3.9}{0}
\bez{6.15}{1.3}{6.55}{0}{6.15}{-1.3}\bez{6.15}{1.3}{5.75}{0}{6.15}{-1.3}
\bez{3.9}{0}{4.83}{1}{6.15}{1.3}\bez{3.9}{0}{5.23}{0.3}{6.15}{1.3}
\bez{3.9}{0}{4.83}{-1}{6.15}{-1.3}\bez{3.9}{0}{5.23}{-0.3}{6.15}{-1.3}
\bez{-6.15}{1.3}{-6.55}{0}{-6.15}{-1.3}\bez{-6.15}{1.3}{-5.75}{0}{-6.15}{-1.3}
\bez{-3.9}{0}{-4.83}{1}{-6.15}{1.3}\bez{-3.9}{0}{-5.23}{0.3}{-6.15}{1.3}
\bez{-3.9}{0}{-4.83}{-1}{-6.15}{-1.3}\bez{-3.9}{0}{-5.23}{-0.3}{-6.15}{-1.3}
\dl{8.28}{2.79}{8.62}{0.5}\dl{8.28}{2.79}{6.15}{1.3}\dl{8.62}{0.5}{6.15}{1.3}
\dl{-8.28}{2.79}{-8.62}{0.5}\dl{-8.28}{2.79}{-6.15}{1.3}\dl{-8.62}{0.5}{-6.15}{1.3}
\dl{8.28}{-2.79}{8.62}{-0.5}\dl{8.28}{-2.79}{6.15}{-1.3}\dl{8.62}{-0.5}{6.15}{-1.3}
\dl{-8.28}{-2.79}{-8.62}{-0.5}\dl{-8.28}{-2.79}{-6.15}{-1.3}\dl{-8.62}{-0.5}{-6.15}{-1.3}
\varline{850}{5.5}
\bez{-3.9}{0}{-2.98}{-4.84}{1.3}{-2.4}
\thnline
\dl{0.13}{2.6}{-0.13}{2.7}\dl{0.09}{2.6}{-0.13}{2.7}\dl{0.05}{2.6}{-0.13}{2.7}\dl{0.01}{2.6}{-0.13}{2.7}\dl{-0.03}{2.6}{-0.13}{2.7}\dl{-0.07}{2.6}{-0.13}{2.7}
\dl{0.13}{2.6}{-0.13}{2.5}\dl{0.09}{2.6}{-0.13}{2.5}\dl{0.05}{2.6}{-0.13}{2.5}\dl{0.01}{2.6}{-0.13}{2.5}\dl{-0.03}{2.6}{-0.13}{2.5}\dl{-0.07}{2.6}{-0.13}{2.5}
\dl{0.13}{1}{-0.13}{1.1}\dl{0.09}{1}{-0.13}{1.1}\dl{0.05}{1}{-0.13}{1.1}\dl{0.01}{1}{-0.13}{1.1}\dl{-0.03}{1}{-0.13}{1.1}\dl{-0.07}{1}{-0.13}{1.1}
\dl{0.13}{1}{-0.13}{0.9}\dl{0.09}{1}{-0.13}{0.9}\dl{0.05}{1}{-0.13}{0.9}\dl{0.01}{1}{-0.13}{0.9}\dl{-0.03}{1}{-0.13}{0.9}\dl{-0.07}{1}{-0.13}{0.9}
\dl{0.13}{-0.6}{-0.13}{-0.7}\dl{0.09}{-0.6}{-0.13}{-0.7}\dl{0.05}{-0.6}{-0.13}{-0.7}\dl{0.01}{-0.6}{-0.13}{-0.7}\dl{-0.03}{-0.6}{-0.13}{-0.7}\dl{-0.07}{-0.6}{-0.13}{-0.7}
\dl{0.13}{-0.6}{-0.13}{-0.5}\dl{0.09}{-0.6}{-0.13}{-0.5}\dl{0.05}{-0.6}{-0.13}{-0.5}\dl{0.01}{-0.6}{-0.13}{-0.5}\dl{-0.03}{-0.6}{-0.13}{-0.5}\dl{-0.07}{-0.6}{-0.13}{-0.5}
\dl{0.13}{-2.2}{-0.13}{-2.3}\dl{0.09}{-2.2}{-0.13}{-2.3}\dl{0.05}{-2.2}{-0.13}{-2.3}\dl{0.01}{-2.2}{-0.13}{-2.3}\dl{-0.03}{-2.2}{-0.13}{-2.3}\dl{-0.07}{-2.2}{-0.13}{-2.3}
\dl{0.13}{-2.2}{-0.13}{-2.1}\dl{0.09}{-2.2}{-0.13}{-2.1}\dl{0.05}{-2.2}{-0.13}{-2.1}\dl{0.01}{-2.2}{-0.13}{-2.1}\dl{-0.03}{-2.2}{-0.13}{-2.1}\dl{-0.07}{-2.2}{-0.13}{-2.1}
\dl{-0.13}{2.2}{0.13}{2.3}\dl{-0.09}{2.2}{0.13}{2.3}\dl{-0.05}{2.2}{0.13}{2.3}\dl{-0.01}{2.2}{0.13}{2.3}\dl{0.03}{2.2}{0.13}{2.3}\dl{0.07}{2.2}{0.13}{2.3}
\dl{-0.13}{2.2}{0.13}{2.1}\dl{-0.09}{2.2}{0.13}{2.1}\dl{-0.05}{2.2}{0.13}{2.1}\dl{-0.01}{2.2}{0.13}{2.1}\dl{0.03}{2.2}{0.13}{2.1}\dl{0.07}{2.2}{0.13}{2.1}
\dl{-0.13}{0.6}{0.13}{0.7}\dl{-0.09}{0.6}{0.13}{0.7}\dl{-0.05}{0.6}{0.13}{0.7}\dl{-0.01}{0.6}{0.13}{0.7}\dl{0.03}{0.6}{0.13}{0.7}\dl{0.07}{0.6}{0.13}{0.7}
\dl{-0.13}{0.6}{0.13}{0.5}\dl{-0.09}{0.6}{0.13}{0.5}\dl{-0.05}{0.6}{0.13}{0.5}\dl{-0.01}{0.6}{0.13}{0.5}\dl{0.03}{0.6}{0.13}{0.5}\dl{0.07}{0.6}{0.13}{0.5}
\dl{-0.13}{-1}{0.13}{-0.9}\dl{-0.09}{-1}{0.13}{-0.9}\dl{-0.05}{-1}{0.13}{-0.9}\dl{-0.01}{-1}{0.13}{-0.9}\dl{0.03}{-1}{0.13}{-0.9}\dl{0.07}{-1}{0.13}{-0.9}
\dl{-0.13}{-1}{0.13}{-1.1}\dl{-0.09}{-1}{0.13}{-1.1}\dl{-0.05}{-1}{0.13}{-1.1}\dl{-0.01}{-1}{0.13}{-1.1}\dl{0.03}{-1}{0.13}{-1.1}\dl{0.07}{-1}{0.13}{-1.1}
\dl{-0.13}{-2.6}{0.13}{-2.5}\dl{-0.09}{-2.6}{0.13}{-2.5}\dl{-0.05}{-2.6}{0.13}{-2.5}\dl{-0.01}{-2.6}{0.13}{-2.5}\dl{0.03}{-2.6}{0.13}{-2.5}\dl{0.07}{-2.6}{0.13}{-2.5}
\dl{-0.13}{-2.6}{0.13}{-2.7}\dl{-0.09}{-2.6}{0.13}{-2.7}\dl{-0.05}{-2.6}{0.13}{-2.7}\dl{-0.01}{-2.6}{0.13}{-2.7}\dl{0.03}{-2.6}{0.13}{-2.7}\dl{0.07}{-2.6}{0.13}{-2.7}
\dl{2.82}{1.22}{2.67}{1.49}\dl{2.77}{1.27}{2.67}{1.49}\dl{2.73}{1.31}{2.67}{1.49}\dl{2.7}{1.34}{2.67}{1.49}\dl{2.68}{1.36}{2.67}{1.49}\dl{2.66}{1.38}{2.67}{1.49}
\dl{2.82}{1.22}{2.55}{1.34}\dl{2.77}{1.27}{2.55}{1.34}\dl{2.73}{1.31}{2.55}{1.34}\dl{2.7}{1.34}{2.55}{1.34}\dl{2.68}{1.36}{2.55}{1.34}\dl{2.66}{1.38}{2.55}{1.34}
\dl{2.38}{1.18}{2.53}{0.91}\dl{2.43}{1.13}{2.53}{0.91}\dl{2.47}{1.09}{2.53}{0.91}\dl{2.5}{1.06}{2.53}{0.91}\dl{2.52}{1.04}{2.53}{0.91}\dl{2.54}{1.02}{2.53}{0.91}
\dl{2.38}{1.18}{2.65}{1.06}\dl{2.43}{1.13}{2.65}{1.06}\dl{2.47}{1.09}{2.65}{1.06}\dl{2.5}{1.06}{2.65}{1.06}\dl{2.52}{1.04}{2.65}{1.06}\dl{2.54}{1.02}{2.65}{1.06}
\dl{2.77}{0.5}{2.53}{0.68}\dl{2.7}{0.52}{2.53}{0.68}\dl{2.64}{0.54}{2.53}{0.68}\dl{2.6}{0.55}{2.53}{0.68}\dl{2.57}{0.56}{2.53}{0.68}\dl{2.55}{0.57}{2.53}{0.68}
\dl{2.77}{0.5}{2.47}{0.5}\dl{2.7}{0.52}{2.47}{0.5}\dl{2.64}{0.54}{2.47}{0.5}\dl{2.6}{0.55}{2.47}{0.5}\dl{2.57}{0.56}{2.47}{0.5}
\dl{2.55}{0.57}{2.47}{0.5}
\dl{2.43}{0.3}{2.67}{0.12}\dl{2.5}{0.28}{2.67}{0.12}\dl{2.56}{0.26}{2.67}{0.12}\dl{2.6}{0.25}{2.67}{0.12}\dl{2.63}{0.24}{2.67}{0.12}\dl{2.65}{0.23}{2.67}{0.12}
\dl{2.43}{0.3}{2.73}{0.3}\dl{2.5}{0.28}{2.73}{0.3}\dl{2.56}{0.26}{2.73}{0.3}\dl{2.6}{0.25}{2.73}{0.3}\dl{2.63}{0.24}{2.73}{0.3}
\dl{2.65}{0.23}{2.73}{0.3}
\dl{2.7}{-0.21}{2.46}{-0.39}\dl{2.63}{-0.23}{2.46}{-0.39}\dl{2.57}{-0.25}{2.46}{-0.39}\dl{2.53}{-0.26}{2.46}{-0.39}\dl{2.5}{-0.27}{2.46}{-0.39}\dl{2.48}{-0.28}{2.46}{-0.39}
\dl{2.7}{-0.21}{2.4}{-0.21}\dl{2.63}{-0.23}{2.4}{-0.21}\dl{2.57}{-0.25}{2.4}{-0.21}\dl{2.53}{-0.26}{2.4}{-0.21}\dl{2.5}{-0.27}{2.4}{-0.21}\dl{2.48}{-0.28}{2.4}{-0.21}
\dl{2.5}{-0.59}{2.74}{-0.41}\dl{2.57}{-0.57}{2.74}{-0.41}\dl{2.63}{-0.55}{2.74}{-0.41}\dl{2.67}{-0.54}{2.74}{-0.41}\dl{2.7}{-0.53}{2.74}{-0.41}\dl{2.72}{-0.52}{2.74}{-0.41}
\dl{2.49}{-0.59}{2.8}{-0.59}\dl{2.57}{-0.57}{2.8}{-0.59}\dl{2.63}{-0.55}{2.8}{-0.59}\dl{2.67}{-0.54}{2.8}{-0.59}\dl{2.7}{-0.53}{2.8}{-0.59}\dl{2.72}{-0.52}{2.8}{-0.59}
\dl{2.6}{-1.42}{2.75}{-1.15}\dl{2.64}{-1.38}{2.75}{-1.15}\dl{2.68}{-1.34}{2.75}{-1.15}\dl{2.71}{-1.31}{2.75}{-1.15}\dl{2.74}{-1.28}{2.75}{-1.15}\dl{2.76}{-1.26}{2.75}{-1.15}
\dl{2.6}{-1.42}{2.87}{-1.3}\dl{2.64}{-1.38}{2.87}{-1.3}\dl{2.68}{-1.34}{2.87}{-1.3}\dl{2.71}{-1.31}{2.87}{-1.3}\dl{2.74}{-1.28}{2.87}{-1.3}\dl{2.76}{-1.26}{2.87}{-1.3}
\dl{2.6}{-0.98}{2.45}{-1.25}\dl{2.56}{-1.02}{2.45}{-1.25}\dl{2.52}{-1.06}{2.45}{-1.25}\dl{2.49}{-1.09}{2.45}{-1.25}\dl{2.46}{-1.12}{2.45}{-1.25}\dl{2.44}{-1.14}{2.45}{-1.25}
\dl{2.6}{-0.98}{2.33}{-1.1}\dl{2.56}{-1.02}{2.33}{-1.1}\dl{2.52}{-1.06}{2.33}{-1.1}\dl{2.49}{-1.09}{2.33}{-1.1}\dl{2.46}{-1.12}{2.33}{-1.1}\dl{2.44}{-1.14}{2.33}{-1.1}\dl{2.44}{-1.14}{2.33}{-1.1}
\dl{5.95}{0.13}{6.05}{-0.13}\dl{5.95}{0.06}{6.05}{-0.13}\dl{5.95}{0.01}{6.05}{-0.13}\dl{5.95}{-0.03}{6.05}{-0.13}\dl{5.95}{-0.05}{6.05}{-0.13}\dl{5.95}{-0.07}{6.05}{-0.13}
\dl{5.95}{0.13}{5.85}{-0.13}\dl{5.95}{0.06}{5.85}{-0.13}\dl{5.95}{0.01}{5.85}{-0.13}\dl{5.95}{-0.03}{5.85}{-0.13}\dl{5.95}{-0.05}{5.85}{-0.13}\dl{5.95}{-0.07}{5.85}{-0.13}
\dl{5.04}{0.88}{4.87}{0.67}\dl{5}{0.86}{4.87}{0.67}\dl{4.96}{0.84}{4.87}{0.67}\dl{4.92}{0.82}{4.87}{0.67}\dl{4.89}{0.8}{4.87}{0.67}\dl{4.87}{0.78}{4.87}{0.67}
\dl{5.04}{0.88}{4.77}{0.85}\dl{5}{0.86}{4.77}{0.85}\dl{4.96}{0.84}{4.77}{0.85}\dl{4.92}{0.82}{4.77}{0.85}\dl{4.89}{0.8}{4.77}{0.85}\dl{4.87}{0.78}{4.77}{0.85}
\dl{5.01}{0.42}{5.18}{0.63}\dl{5.05}{0.44}{5.18}{0.63}\dl{5.09}{0.46}{5.18}{0.63}\dl{5.13}{0.48}{5.18}{0.63}\dl{5.16}{0.5}{5.18}{0.63}\dl{5.18}{0.52}{5.18}{0.63}
\dl{5.01}{0.42}{5.28}{0.45}\dl{5.05}{0.44}{5.28}{0.45}\dl{5.09}{0.46}{5.28}{0.45}\dl{5.13}{0.48}{5.28}{0.45}\dl{5.16}{0.5}{5.28}{0.45}\dl{5.18}{0.52}{5.28}{0.45}
\dl{5.24}{-0.54}{5.07}{-0.33}\dl{5.2}{-0.52}{5.07}{-0.33}\dl{5.16}{-0.5}{5.07}{-0.33}\dl{5.12}{-0.48}{5.07}{-0.33}\dl{5.09}{-0.46}{5.07}{-0.33}\dl{5.07}{-0.44}{5.07}{-0.33}
\dl{5.24}{-0.54}{4.97}{-0.51}\dl{5.2}{-0.52}{4.97}{-0.51}\dl{5.16}{-0.5}{4.97}{-0.51}\dl{5.12}{-0.48}{4.97}{-0.51}\dl{5.09}{-0.46}{4.97}{-0.51}\dl{5.07}{-0.44}{4.97}{-0.51}
\dl{4.81}{-0.76}{4.98}{-0.97}\dl{4.85}{-0.78}{4.98}{-0.97}\dl{4.89}{-0.8}{4.98}{-0.97}\dl{4.93}{-0.82}{4.98}{-0.97}\dl{4.96}{-0.84}{4.98}{-0.97}\dl{4.98}{-0.86}{4.98}{-0.97}
\dl{4.81}{-0.76}{5.08}{-0.79}\dl{4.85}{-0.78}{5.08}{-0.79}\dl{4.89}{-0.8}{5.08}{-0.79}\dl{4.93}{-0.82}{5.08}{-0.79}\dl{4.96}{-0.84}{5.08}{-0.79}\dl{4.98}{-0.86}{5.08}{-0.79}
\dl{6.35}{-0.13}{6.45}{0.13}\dl{6.35}{-0.06}{6.45}{0.13}\dl{6.35}{-0.01}{6.45}{0.13}\dl{6.35}{0.03}{6.45}{0.13}\dl{6.35}{0.05}{6.45}{0.13}\dl{6.35}{0.07}{6.45}{0.13}
\dl{6.35}{-0.13}{6.25}{0.13}\dl{6.35}{-0.06}{6.25}{0.13}\dl{6.35}{-0.01}{6.25}{0.13}\dl{6.35}{0.03}{6.25}{0.13}\dl{6.35}{0.05}{6.25}{0.13}\dl{6.35}{0.07}{6.25}{0.13}
\dl{7.31}{2.11}{7.05}{2.06}\dl{7.27}{2.08}{7.05}{2.06}\dl{7.24}{2.05}{7.05}{2.06}\dl{7.21}{2.03}{7.05}{2.06}\dl{7.18}{2.01}{7.05}{2.06}\dl{7.15}{2}{7.05}{2.06}
\dl{7.31}{2.11}{7.17}{1.89}\dl{7.27}{2.08}{7.17}{1.89}\dl{7.24}{2.05}{7.17}{1.89}\dl{7.21}{2.03}{7.17}{1.89}\dl{7.18}{2.01}{7.17}{1.89}\dl{7.15}{2}{7.17}{1.89}
\dl{7.26}{0.94}{7.52}{0.97}\dl{7.32}{0.92}{7.52}{0.97}\dl{7.36}{0.91}{7.52}{0.97}\dl{7.39}{0.9}{7.52}{0.97}\dl{7.42}{0.89}{7.52}{0.97}\dl{7.45}{0.88}{7.52}{0.97}
\dl{7.26}{0.94}{7.46}{0.77}\dl{7.32}{0.92}{7.46}{0.77}\dl{7.36}{0.91}{7.46}{0.77}\dl{7.39}{0.9}{7.46}{0.77}\dl{7.42}{0.89}{7.46}{0.77}\dl{7.45}{0.88}{7.46}{0.77}
\dl{8.44}{1.7}{8.53}{1.75}\dl{8.44}{1.68}{8.53}{1.75}\dl{8.45}{1.65}{8.53}{1.75}\dl{8.45}{1.62}{8.53}{1.75}\dl{8.46}{1.59}{8.53}{1.75}\dl{8.46}{1.55}{8.53}{1.75}\dl{8.47}{1.5}{8.53}{1.75}
\dl{8.44}{1.7}{8.33}{1.71}\dl{8.44}{1.68}{8.33}{1.71}\dl{8.45}{1.65}{8.33}{1.71}\dl{8.45}{1.62}{8.33}{1.71}\dl{8.46}{1.59}{8.33}{1.71}\dl{8.46}{1.55}{8.33}{1.71}\dl{8.47}{1.5}{8.33}{1.71}
\dl{7.12}{-1.98}{7.38}{-2.03}\dl{7.16}{-2.01}{7.38}{-2.03}\dl{7.19}{-2.04}{7.38}{-2.03}\dl{7.22}{-2.06}{7.38}{-2.03}\dl{7.25}{-2.08}{7.38}{-2.03}\dl{7.28}{-2.09}{7.38}{-2.03}
\dl{7.12}{-1.98}{7.26}{-2.2}\dl{7.16}{-2.01}{7.26}{-2.2}\dl{7.19}{-2.04}{7.26}{-2.2}\dl{7.22}{-2.06}{7.26}{-2.2}\dl{7.25}{-2.08}{7.26}{-2.2}\dl{7.28}{-2.09}{7.26}{-2.2}
\dl{7.51}{-0.86}{7.25}{-0.83}\dl{7.45}{-0.88}{7.25}{-0.83}\dl{7.41}{-0.89}{7.25}{-0.83}\dl{7.38}{-0.9}{7.25}{-0.83}\dl{7.35}{-0.91}{7.25}{-0.83}\dl{7.32}{-0.92}{7.25}{-0.83}
\dl{7.51}{-0.86}{7.31}{-1.03}\dl{7.45}{-0.88}{7.31}{-1.03}\dl{7.41}{-0.89}{7.31}{-1.03}\dl{7.38}{-0.9}{7.31}{-1.03}\dl{7.35}{-0.91}{7.31}{-1.03}\dl{7.32}{-0.92}{7.31}{-1.03}
\dl{8.46}{-1.58}{8.37}{-1.53}\dl{8.46}{-1.6}{8.37}{-1.53}\dl{8.45}{-1.63}{8.37}{-1.53}\dl{8.45}{-1.66}{8.37}{-1.53}\dl{8.44}{-1.69}{8.37}{-1.53}\dl{8.44}{-1.73}{8.37}{-1.53}\dl{8.43}{-1.78}{8.37}{-1.53}
\dl{8.46}{-1.58}{8.57}{-1.57}\dl{8.46}{-1.6}{8.57}{-1.57}\dl{8.45}{-1.63}{8.57}{-1.57}\dl{8.45}{-1.66}{8.57}{-1.57}\dl{8.44}{-1.69}{8.57}{-1.57}\dl{8.44}{-1.73}{8.57}{-1.57}\dl{8.43}{-1.78}{8.57}{-1.57}
\dl{-2.82}{-1.22}{-2.67}{-1.49}\dl{-2.77}{-1.27}{-2.67}{-1.49}\dl{-2.73}{-1.31}{-2.67}{-1.49}\dl{-2.7}{-1.34}{-2.67}{-1.49}\dl{-2.68}{-1.36}{-2.67}{-1.49}\dl{-2.66}{-1.38}{-2.67}{-1.49}
\dl{-2.82}{-1.22}{-2.55}{-1.34}\dl{-2.77}{-1.27}{-2.55}{-1.34}\dl{-2.73}{-1.31}{-2.55}{-1.34}\dl{-2.7}{-1.34}{-2.55}{-1.34}\dl{-2.68}{-1.36}{-2.55}{-1.34}\dl{-2.66}{-1.38}{-2.55}{-1.34}
\dl{-2.38}{-1.18}{-2.53}{-0.91}\dl{-2.43}{-1.13}{-2.53}{-0.91}\dl{-2.47}{-1.09}{-2.53}{-0.91}\dl{-2.5}{-1.06}{-2.53}{-0.91}\dl{-2.52}{-1.04}{-2.53}{-0.91}\dl{-2.54}{-1.02}{-2.53}{-0.91}
\dl{-2.38}{-1.18}{-2.65}{-1.06}\dl{-2.43}{-1.13}{-2.65}{-1.06}\dl{-2.47}{-1.09}{-2.65}{-1.06}\dl{-2.5}{-1.06}{-2.65}{-1.06}\dl{-2.52}{-1.04}{-2.65}{-1.06}\dl{-2.54}{-1.02}{-2.65}{-1.06}
\dl{-2.77}{-0.5}{-2.53}{-0.68}\dl{-2.7}{-0.52}{-2.53}{-0.68}\dl{-2.64}{-0.54}{-2.53}{-0.68}\dl{-2.6}{-0.55}{-2.53}{-0.68}\dl{-2.57}{-0.56}{-2.53}{-0.68}\dl{-2.55}{-0.57}{-2.53}{-0.68}
\dl{-2.77}{-0.5}{-2.47}{-0.5}\dl{-2.7}{-0.52}{-2.47}{-0.5}\dl{-2.64}{-0.54}{-2.47}{-0.5}\dl{-2.6}{-0.55}{-2.47}{-0.5}\dl{-2.57}{-0.56}{-2.47}{-0.5}\dl{-2.55}{-0.57}{-2.47}{-0.5}
\dl{-2.43}{-0.3}{-2.67}{-0.12}\dl{-2.5}{-0.28}{-2.67}{-0.12}\dl{-2.56}{-0.26}{-2.67}{-0.12}\dl{-2.6}{-0.25}{-2.67}{-0.12}\dl{-2.63}{-0.24}{-2.67}{-0.12}\dl{-2.65}{-0.23}{-2.67}{-0.12}
\dl{-2.43}{-0.3}{-2.73}{-0.3}\dl{-2.5}{-0.28}{-2.73}{-0.3}\dl{-2.56}{-0.26}{-2.73}{-0.3}\dl{-2.6}{-0.25}{-2.73}{-0.3}\dl{-2.63}{-0.24}{-2.73}{-0.3}\dl{-2.65}{-0.23}{-2.73}{-0.3}
\dl{-2.7}{0.21}{-2.46}{0.39}\dl{-2.63}{0.23}{-2.46}{0.39}\dl{-2.57}{0.25}{-2.46}{0.39}\dl{-2.53}{0.26}{-2.46}{0.39}\dl{-2.5}{0.27}{-2.46}{0.39}\dl{-2.48}{0.28}{-2.46}{0.39}
\dl{-2.7}{0.21}{-2.4}{0.21}\dl{-2.63}{0.23}{-2.4}{0.21}\dl{-2.57}{0.25}{-2.4}{0.21}\dl{-2.53}{0.26}{-2.4}{0.21}\dl{-2.5}{0.27}{-2.4}{0.21}\dl{-2.48}{0.28}{-2.4}{0.21}
\dl{-2.5}{0.59}{-2.74}{0.41}\dl{-2.57}{0.57}{-2.74}{0.41}\dl{-2.63}{0.55}{-2.74}{0.41}\dl{-2.67}{0.54}{-2.74}{0.41}\dl{-2.7}{0.53}{-2.74}{0.41}\dl{-2.72}{0.52}{-2.74}{0.41}
\dl{-2.49}{0.59}{-2.8}{0.59}\dl{-2.57}{0.57}{-2.8}{0.59}\dl{-2.63}{0.55}{-2.8}{0.59}\dl{-2.67}{0.54}{-2.8}{0.59}\dl{-2.7}{0.53}{-2.8}{0.59}\dl{-2.72}{0.52}{-2.8}{0.59}
\dl{-2.6}{1.42}{-2.75}{1.15}\dl{-2.64}{1.38}{-2.75}{1.15}\dl{-2.68}{1.34}{-2.75}{1.15}\dl{-2.71}{1.31}{-2.75}{1.15}\dl{-2.74}{1.28}{-2.75}{1.15}\dl{-2.76}{1.26}{-2.75}{1.15}
\dl{-2.6}{1.42}{-2.87}{1.3}\dl{-2.64}{1.38}{-2.87}{1.3}\dl{-2.68}{1.34}{-2.87}{1.3}\dl{-2.71}{1.31}{-2.87}{1.3}\dl{-2.74}{1.28}{-2.87}{1.3}\dl{-2.76}{1.26}{-2.87}{1.3}
\dl{-2.6}{0.98}{-2.45}{1.25}\dl{-2.56}{1.02}{-2.45}{1.25}\dl{-2.52}{1.06}{-2.45}{1.25}\dl{-2.49}{1.09}{-2.45}{1.25}\dl{-2.46}{1.12}{-2.45}{1.25}\dl{-2.44}{1.14}{-2.45}{1.25}
\dl{-2.6}{0.98}{-2.33}{1.1}\dl{-2.56}{1.02}{-2.33}{1.1}\dl{-2.52}{1.06}{-2.33}{1.1}\dl{-2.49}{1.09}{-2.33}{1.1}\dl{-2.46}{1.12}{-2.33}{1.1}\dl{-2.44}{1.14}{-2.33}{1.1}\dl{-2.44}{1.14}{-2.33}{1.1}
\dl{-5.95}{-0.13}{-6.05}{0.13}\dl{-5.95}{-0.06}{-6.05}{0.13}\dl{-5.95}{-0.01}{-6.05}{0.13}\dl{-5.95}{0.03}{-6.05}{0.13}\dl{-5.95}{0.05}{-6.05}{0.13}\dl{-5.95}{0.07}{-6.05}{0.13}
\dl{-5.95}{-0.13}{-5.85}{0.13}\dl{-5.95}{-0.06}{-5.85}{0.13}\dl{-5.95}{-0.01}{-5.85}{0.13}\dl{-5.95}{0.03}{-5.85}{0.13}\dl{-5.95}{0.05}{-5.85}{0.13}\dl{-5.95}{0.07}{-5.85}{0.13}
\dl{-5.04}{-0.88}{-4.87}{-0.67}\dl{-5}{-0.86}{-4.87}{-0.67}\dl{-4.96}{-0.84}{-4.87}{-0.67}\dl{-4.92}{-0.82}{-4.87}{-0.67}\dl{-4.89}{-0.8}{-4.87}{-0.67}\dl{-4.87}{-0.78}{-4.87}{-0.67}
\dl{-5.04}{-0.88}{-4.77}{-0.85}\dl{-5}{-0.86}{-4.77}{-0.85}\dl{-4.96}{-0.84}{-4.77}{-0.85}\dl{-4.92}{-0.82}{-4.77}{-0.85}\dl{-4.89}{-0.8}{-4.77}{-0.85}\dl{-4.87}{-0.78}{-4.77}{-0.85}
\dl{-5.01}{-0.42}{-5.18}{-0.63}\dl{-5.05}{-0.44}{-5.18}{-0.63}\dl{-5.09}{-0.46}{-5.18}{-0.63}\dl{-5.13}{-0.48}{-5.18}{-0.63}\dl{-5.16}{-0.5}{-5.18}{-0.63}\dl{-5.18}{-0.52}{-5.18}{-0.63}
\dl{-5.01}{-0.42}{-5.28}{-0.45}\dl{-5.05}{-0.44}{-5.28}{-0.45}\dl{-5.09}{-0.46}{-5.28}{-0.45}\dl{-5.13}{-0.48}{-5.28}{-0.45}\dl{-5.16}{-0.5}{-5.28}{-0.45}\dl{-5.18}{-0.52}{-5.28}{-0.45}
\dl{-5.24}{0.54}{-5.07}{0.33}\dl{-5.2}{0.52}{-5.07}{0.33}\dl{-5.16}{0.5}{-5.07}{0.33}\dl{-5.12}{0.48}{-5.07}{0.33}\dl{-5.09}{0.46}{-5.07}{0.33}\dl{-5.07}{0.44}{-5.07}{0.33}
\dl{-5.24}{0.54}{-4.97}{0.51}\dl{-5.2}{0.52}{-4.97}{0.51}\dl{-5.16}{0.5}{-4.97}{0.51}\dl{-5.12}{0.48}{-4.97}{0.51}\dl{-5.09}{0.46}{-4.97}{0.51}\dl{-5.07}{0.44}{-4.97}{0.51}
\dl{-4.81}{0.76}{-4.98}{0.97}\dl{-4.85}{0.78}{-4.98}{0.97}\dl{-4.89}{0.8}{-4.98}{0.97}\dl{-4.93}{0.82}{-4.98}{0.97}\dl{-4.96}{0.84}{-4.98}{0.97}\dl{-4.98}{0.86}{-4.98}{0.97}
\dl{-4.81}{0.76}{-5.08}{0.79}\dl{-4.85}{0.78}{-5.08}{0.79}\dl{-4.89}{0.8}{-5.08}{0.79}\dl{-4.93}{0.82}{-5.08}{0.79}\dl{-4.96}{0.84}{-5.08}{0.79}\dl{-4.98}{0.86}{-5.08}{0.79}
\dl{-6.35}{0.13}{-6.45}{-0.13}\dl{-6.35}{0.06}{-6.45}{-0.13}\dl{-6.35}{0.01}{-6.45}{-0.13}\dl{-6.35}{-0.03}{-6.45}{-0.13}\dl{-6.35}{-0.05}{-6.45}{-0.13}\dl{-6.35}{-0.07}{-6.45}{-0.13}
\dl{-6.35}{0.13}{-6.25}{-0.13}\dl{-6.35}{0.06}{-6.25}{-0.13}\dl{-6.35}{0.01}{-6.25}{-0.13}\dl{-6.35}{-0.03}{-6.25}{-0.13}\dl{-6.35}{-0.05}{-6.25}{-0.13}\dl{-6.35}{-0.07}{-6.25}{-0.13}
\dl{-7.31}{-2.11}{-7.05}{-2.06}\dl{-7.27}{-2.08}{-7.05}{-2.06}\dl{-7.24}{-2.05}{-7.05}{-2.06}\dl{-7.21}{-2.03}{-7.05}{-2.06}\dl{-7.18}{-2.01}{-7.05}{-2.06}\dl{-7.15}{-2}{-7.05}{-2.06}
\dl{-7.31}{-2.11}{-7.17}{-1.89}\dl{-7.27}{-2.08}{-7.17}{-1.89}\dl{-7.24}{-2.05}{-7.17}{-1.89}\dl{-7.21}{-2.03}{-7.17}{-1.89}\dl{-7.18}{-2.01}{-7.17}{-1.89}\dl{-7.15}{-2}{-7.17}{-1.89}
\dl{-7.26}{-0.94}{-7.52}{-0.97}\dl{-7.32}{-0.92}{-7.52}{-0.97}\dl{-7.36}{-0.91}{-7.52}{-0.97}\dl{-7.39}{-0.9}{-7.52}{-0.97}\dl{-7.42}{-0.89}{-7.52}{-0.97}\dl{-7.45}{-0.88}{-7.52}{-0.97}
\dl{-7.26}{-0.94}{-7.46}{-0.77}\dl{-7.32}{-0.92}{-7.46}{-0.77}\dl{-7.36}{-0.91}{-7.46}{-0.77}\dl{-7.39}{-0.9}{-7.46}{-0.77}\dl{-7.42}{-0.89}{-7.46}{-0.77}\dl{-7.45}{-0.88}{-7.46}{-0.77}
\dl{-8.44}{-1.7}{-8.53}{-1.75}\dl{-8.44}{-1.68}{-8.53}{-1.75}\dl{-8.45}{-1.65}{-8.53}{-1.75}\dl{-8.45}{-1.62}{-8.53}{-1.75}\dl{-8.46}{-1.59}{-8.53}{-1.75}\dl{-8.46}{-1.55}{-8.53}{-1.75}\dl{-8.47}{-1.5}{-8.53}{-1.75}
\dl{-8.44}{-1.7}{-8.33}{-1.71}\dl{-8.44}{-1.68}{-8.33}{-1.71}\dl{-8.45}{-1.65}{-8.33}{-1.71}\dl{-8.45}{-1.62}{-8.33}{-1.71}\dl{-8.46}{-1.59}{-8.33}{-1.71}\dl{-8.46}{-1.55}{-8.33}{-1.71}\dl{-8.47}{-1.5}{-8.33}{-1.71}
\dl{-7.12}{1.98}{-7.38}{2.03}\dl{-7.16}{2.01}{-7.38}{2.03}\dl{-7.19}{2.04}{-7.38}{2.03}\dl{-7.22}{2.06}{-7.38}{2.03}\dl{-7.25}{2.08}{-7.38}{2.03}\dl{-7.28}{2.09}{-7.38}{2.03}
\dl{-7.12}{1.98}{-7.26}{2.2}\dl{-7.16}{2.01}{-7.26}{2.2}\dl{-7.19}{2.04}{-7.26}{2.2}\dl{-7.22}{2.06}{-7.26}{2.2}\dl{-7.25}{2.08}{-7.26}{2.2}\dl{-7.28}{2.09}{-7.26}{2.2}
\dl{-7.51}{0.86}{-7.25}{0.83}\dl{-7.45}{0.88}{-7.25}{0.83}\dl{-7.41}{0.89}{-7.25}{0.83}\dl{-7.38}{0.9}{-7.25}{0.83}\dl{-7.35}{0.91}{-7.25}{0.83}\dl{-7.32}{0.92}{-7.25}{0.83}
\dl{-7.51}{0.86}{-7.31}{1.03}\dl{-7.45}{0.88}{-7.31}{1.03}\dl{-7.41}{0.89}{-7.31}{1.03}\dl{-7.38}{0.9}{-7.31}{1.03}\dl{-7.35}{0.91}{-7.31}{1.03}\dl{-7.32}{0.92}{-7.31}{1.03}
\dl{-8.46}{1.58}{-8.37}{1.53}\dl{-8.46}{1.6}{-8.37}{1.53}\dl{-8.45}{1.63}{-8.37}{1.53}\dl{-8.45}{1.66}{-8.37}{1.53}\dl{-8.44}{1.69}{-8.37}{1.53}\dl{-8.44}{1.73}{-8.37}{1.53}\dl{-8.43}{1.78}{-8.37}{1.53}
\dl{-8.46}{1.58}{-8.57}{1.57}\dl{-8.46}{1.6}{-8.57}{1.57}\dl{-8.45}{1.63}{-8.57}{1.57}\dl{-8.45}{1.66}{-8.57}{1.57}\dl{-8.44}{1.69}{-8.57}{1.57}\dl{-8.44}{1.73}{-8.57}{1.57}\dl{-8.43}{1.78}{-8.57}{1.57}
%
\dl{-1.95}{-3.1}{-2.26}{-3.1}\dl{-2}{-3.08}{-2.26}{-3.1}\dl{-2.04}{-3.06}{-2.26}{-3.1}\dl{-2.07}{-3.04}{-2.26}{-3.1}\dl{-2.1}{-3.03}{-2.26}{-3.1}\dl{-2.13}{-3.02}{-2.26}{-3.1}
\dl{-1.95}{-3.1}{-2.17}{-2.88}\dl{-2}{-3.08}{-2.17}{-2.88}\dl{-2.04}{-3.06}{-2.17}{-2.88}\dl{-2.07}{-3.04}{-2.17}{-2.88}\dl{-2.1}{-3.03}{-2.17}{-2.88}\dl{-2.13}{-3.02}{-2.17}{-2.88}
\dl{-1.95}{-3.1}{-2.13}{-3.02}
\point{-9.35}{0.37}{\small $u_1$}\point{-9.07}{-2.9}{\small $u_2$}\point{8.87}{-0.62}{\small $u_4$}\point{8.52}{2.67}{\small $u_3$}
\point{-8.97}{2.67}{\small $v_1$}\point{-9.35}{-0.62}{\small $v_2$}\point{8.52}{-2.9}{\small $v_4$}\point{8.87}{0.37}{\small $v_3$}
\point{-6.3}{1.6}{\small $w_1$}\point{-6.3}{-1.8}{\small $w_2$}\point{5.6}{-1.8}{\small $w_4$}\point{5.6}{1.6}{\small $w_3$}
\point{-4.06}{0.38}{\small $x$}
\point{3.77}{0.38}{\small $y$}\point{1.17}{-2.9}{\small $z$}
\ptll{-6.6}{1.37}{1}\ptll{-6.6}{-1.45}{2}\ptld{-4.05}{-0.27}{5}\ptlr{6.6}{-1.45}{4}\ptlr{6.6}{1.37}{3}\ptld{3.9}{-0.15}{6}
\ptld{-1.23}{2.3}{3}\ptld{-1.3}{0.7}{4}\ptlu{-1.23}{-2.37}{4}\ptlu{-1.3}{-0.77}{3}
\ptld{1.23}{2.3}{1}\ptld{1.3}{0.7}{1}\ptlu{1.23}{-2.37}{2}\ptlu{1.3}{-0.77}{2}
\ptlu{-8.31}{0.53}{7}\ptlu{-8.11}{-2.68}{7}
\ptld{-8.11}{2.65}{8}\ptld{-8.31}{-0.63}{8}
\ptld{8.31}{-0.65}{7}\ptld{8.11}{2.6}{7}
\ptlu{8.11}{-2.65}{8}\ptlu{8.31}{0.56}{8}
\ptlu{-10.7}{2.2}{\textup{\normalsize (b)}}
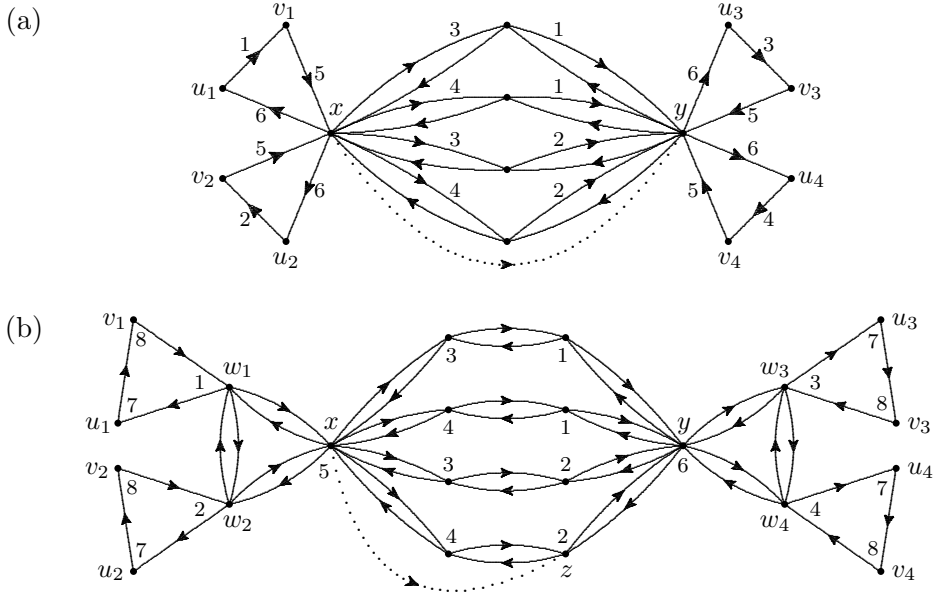
\ptlu{0}{-5.2}{\textup{\small Figure 1. The digraphs $D_1,H_1,D_2$ and $H_2$ in Lemma \ref{lem5}.}}
\]\\[-1ex]
\indent
\end{proof}

In \cite{KY2010}, Krivelevich and Yuster observed that for $rc(G)$ and $rvc(G)$, we cannot generally find an upper bound for one of the parameters in terms of the other. Indeed, let $s\ge 2$. By taking $G=K_{1,s}$, we have $rc(G) = s$ and $rvc(G) = 1$. On the other hand, let $G$ be constructed as follows. Take $s$ vertex-disjoint triangles and, by designating a vertex from each triangle, add a complete graph $K_s$ on the designated vertices. Then $rc(G) \le 4$ and $rvc(G) = s$. The following lemma shows that the same is true about the functions $\overset{\rightarrow}{rc}(D)$ and $\overset{\rightarrow}{rvc}(D)$, and the functions $\overset{\rightarrow}{src}(D)$ and $\overset{\rightarrow}{srvc}(D)$.

\begin{lem}
\indent\\[-3ex]
\begin{enumerate}
\item[(a)] Given $s\ge 4$, there exists a strongly connected digraph $D$ such that $\overset{\rightarrow}{rc}(D)=\overset{\rightarrow}{src}(D)\ge s$ and $\overset{\rightarrow}{rvc}(D)=\overset{\rightarrow}{srvc}(D)=3$.
\item[(b)] Given $s\ge 2$, there exists a strongly connected digraph $D$ such that $\overset{\rightarrow}{rc}(D)=\overset{\rightarrow}{src}(D)=3$ and $\overset{\rightarrow}{rvc}(D)=\overset{\rightarrow}{srvc}(D)=s$.
\end{enumerate}
\end{lem}

\begin{proof}
(a) Let $D$ be the digraph consisting of $t={s-1 \choose 3}+1$ copies of the triangle $\rC_3$, all having one vertex in common. Let $V(D)=\{v,x_1,y_1,\dots,x_t,y_t\}$, and let the arcs of $A(D)$ be $vx_i,x_iy_i,y_iv$, for $1\le i\le t$. Suppose that there is a rainbow connected arc-colouring of $D$ with at most $s-1$ colours. Then clearly, every triangle $vx_iy_i$ must use three different colours. Moreover, some two triangles, say $vx_1y_1$ and $vx_2y_2$, must use the same three colours. Now, $x_1y_1vx_2y_2$ is the unique $x_1-y_2$ path, which has length $4$, and so cannot be rainbow, a contradiction. Hence, $\overset{\rightarrow}{rc}(D)\ge s$. Clearly, the vertex-colouring $c$ of $D$ where $c(v)=1$, $c(x_i)=2$ and $c(y_i)=3$ for $1\le i\le t$, is rainbow vertex-connected, and since diam$(D)=4$, we have $\overset{\rightarrow}{rvc}(D)=3$ by Proposition \ref{pro1}. Finally, we have   $\overset{\rightarrow}{rc}(D)=\overset{\rightarrow}{src}(D)$ and  $\overset{\rightarrow}{rvc}(D)=\overset{\rightarrow}{srvc}(D)$, since for any two vertices $a,b\in V(D)$, there is a unique $a-b$ path.\\[1ex]
\indent(b) Consider the graph $G$ by taking a copy of $K_s$, and adding a pendent edge to each vertex. Let $D=\lrG$. It is easy to see that $\overset{\rightarrow}{rvc}(D)=\overset{\rightarrow}{srvc}(D)=s$. Now in $D$, let $u_1,\dots, u_s$ be the vertices of the $\lrK_s$, and $v_i$ be the vertex where $u_iv_i,v_iu_i\in A(D)$, for $1\le i\le s$. Clearly, the arc-colouring $c$ of $D$ where $c(u_iv_i)=1$, $c(v_iu_i)=2$, and $c(u_iu_j)=3$, for all $1\le i\neq j\le s$, is strongly rainbow connected, so that $\diam(D)\le \overset{\rightarrow}{rc}(D)\le\overset{\rightarrow}{src}(D)\le 3$. Since $\diam(D)=3$, we have $\overset{\rightarrow}{rc}(D)=\overset{\rightarrow}{src}(D)=3$.
\end{proof}

\section{Rainbow connection of some specific digraphs}\label{specsect}

In this section, we shall determine the (strong) rainbow vertex-connection numbers of some specific digraphs. For $n\ge 3$, the \emph{wheel} $W_n$ is the graph obtained by taking the cycle $C_n$, and joining a new vertex $v$ to every vertex of $C_n$. For $t \ge 2$, let $K_{n_1,\dots,n_t}$ denote the complete $t$-partite graph with class-sizes $n_1,\dots,n_t$. The following theorem determines the two parameters for the biorientations of paths, cycles, wheels, and complete multipartite graphs.

\begin{thm}\label{biorthm}
\indent\\[-3ex]
\begin{enumerate}
\item[(a)] For $n \ge 2$, $\overset{\rightarrow}{rvc}(\lrP_n)=\overset{\rightarrow}{srvc}(\lrP_n) = n-2$.
\item[(b)] We have
\[
\overset{\rightarrow}{rvc}(\lrC_n)=\overset{\rightarrow}{srvc}(\lrC_n)=
\left\{
\begin{array}{ll}
\lceil\frac{n}{2}\rceil-2 & \textup{\emph{if} }n=3,5,9\textup{\emph{;}}\\[1ex]
\lceil\frac{n}{2}\rceil-1 & \textup{\emph{if} }n=4,6,7,8,10,12\textup{\emph{;}}\\[1ex]
\lceil\frac{n}{2}\rceil & \textup{\emph{if} }n=14\textup{\emph{ or }}n\ge 16.
\end{array}
\right.
\]
Also, $\overset{\rightarrow}{rvc}(\lrC_n)=\lceil\frac{n}{2}\rceil-1$ and $\overset{\rightarrow}{srvc}(\lrC_n)=\lceil\frac{n}{2}\rceil$ for $n=11,13,15$.
\item[(c)] For $n \ge 4$, $\overset{\rightarrow}{rvc}(\lrWn)=\overset{\rightarrow}{srvc}(\lrWn) = 1$.
\item[(d)] Let $t \ge 2$, and let $K_{n_1,\dots,n_t}$ be a complete $t$-partite graph with $n_i \ge 2$ for some $i$. Then, $\overset{\rightarrow}{rvc}(\lrK_{n_1,\dots,n_t})=\overset{\rightarrow}{srvc}(\lrK_{n_1,\dots,n_t}) = 1$.
\end{enumerate}
\end{thm}

\begin{proof}
For each part, by Proposition \ref{pro3}, it suffices to show that the corresponding graph $G$ has either $rvc(G)$ or $srvc(G)$ equal to the required value. Then, since it is easy to see that $rvc(P_n)=srvc(P_n)=n-2$ for $n\ge 2$, part (a) follows. Parts (c) and (d) follow similarly, and they can also be deduced from Theorem \ref{pro2}(b), since the diameter of each digraph in consideration is $2$. We may also refer to \cite{LMS2012}, Corollary 1.2.

By \cite{LL2011}, Theorem 2.1, we also have the result for $\overset{\rightarrow}{rvc}(\lrC_n)$ in (b). It remains and suffices to determine $srvc(C_n)$. For $n=3,4,5,6,8,9,10,12$, consider a rainbow vertex-connected vertex-colouring of $C_n$, using at most $rvc(C_n)$ colours. For any two vertices $x,y$ of $C_n$, there is a rainbow $x-y$ path with length at most $rvc(C_n)+1\le\lfloor\frac{n}{2}\rfloor$, which therefore is also a rainbow $x-y$ geodesic. This implies that $srvc(C_n)=rvc(C_n)$. For $n=7$, note that the vertex-colouring of $C_7$ where the vertices have colours $1,2,1,2,1,2,3$, in that order around the cycle, is strongly rainbow vertex-connected. It follows that $srvc(C_7)=rvc(C_7)=3$. Finally, let $n=11$ or $n\ge 13$. It is easy to see that the vertex-colouring of $C_n$ where the colours are $1,2,\dots,\lceil\frac{n}{2}\rceil,1,2,\dots,\lfloor\frac{n}{2}\rfloor$, in that order around the cycle, is strongly rainbow vertex-connected. Hence, $srvc(C_n)\le\lceil\frac{n}{2}\rceil$, and for $n\neq 11,13,15$, we have $\lceil\frac{n}{2}\rceil=rvc(C_n)\le srvc(C_n)\le\lceil\frac{n}{2}\rceil$, so that $srvc(C_n)=\lceil\frac{n}{2}\rceil$. For $n=11,13,15$, suppose that we have a vertex-colouring of $C_n$, using at most $\lceil\frac{n}{2}\rceil-1$ colours. Then, there are three vertices with the same colour, and two of them, say $u$ and $v$, are connected by a path $P$ with length at most $\lfloor\frac{n}{3}\rfloor$. Let $x$ and $y$ be the neighbours of $u$ and $v$ not in $P$. Then, the $x-y$ path containing $P$ has length at most $\lfloor\frac{n}{3}\rfloor+2=\lfloor\frac{n}{2}\rfloor$, and is therefore the unique $x-y$ geodesic. It follows that there is no rainbow $x-y$ geodesic, so that $srvc(C_n)\ge\lceil\frac{n}{2}\rceil$. Hence, $srvc(C_n)=\lceil\frac{n}{2}\rceil$.
\end{proof}

We remark that we have the somewhat surprising fact that $\overset{\rightarrow}{srvc}(\lrC_{11})=srvc(C_{11})=6$, which is greater than $\overset{\rightarrow}{srvc}(\lrC_{12})=srvc(C_{12})=5$. Next, we have the following result for directed cycles.

\begin{prop}\label{cyclepro}
Let $n\ge 3$. Then,
\[
\overset{\rightarrow}{rvc}(\rC_n)=\overset{\rightarrow}{srvc}(\rC_n)=
\left\{
\begin{array}{ll}
n-2 & \textup{\emph{if} }n=3,4\textup{\emph{,}}\\
n & \textup{\emph{if} }n\ge 5.
\end{array}
\right.
\]
\end{prop}
\begin{proof}
It is easy to see that $\overset{\rightarrow}{rvc}(\rC_n)=\overset{\rightarrow}{srvc}(\rC_n)=n-2$ for $n=3,4$. For $n\ge 5$,  if we have a vertex-colouring of $\rC_n$ with fewer than $n$ colours, then there exist vertices $u$ and $v$ in $\rC_n$ with the same colour, and with a $u-v$ path $P$ in $\rC_n$ of length at least $3$. If $x$ and $y$ are the neighbours of $v$ and $u$ in $P$, then the unique $x-y$ path in $\rC_n$ contains $u$ and $v$, and thus is not rainbow. It follows that $\overset{\rightarrow}{rvc}(\rC_n)\ge n$, and the result follows from (\ref{pro1eq}).
\end{proof}

Now, we extend the results of Theorem \ref{biorthm}(b) and Proposition \ref{cyclepro}. We shall determine $\overset{\rightarrow}{rvc}(D)$ and $\overset{\rightarrow}{srvc}(D)$, where $D$ is any spanning strongly connected subdigraph of $\lrC_n$ with at least one asymmetric arc. We recall that the analogous problem for $\overset{\rightarrow}{rc}(D)$ and $\overset{\rightarrow}{src}(D)$ was considered by Alva-Samos and  Montellano-Ballesteros \cite{JJ2015}, and their result is the following.

\begin{thm}\textup{\cite{JJ2015}}
Let $D$ be a spanning strongly connected subdigraph of $\lrC_n$ with $k \ge 1$ asymmetric arcs. Then,
\[
\overset{\rightarrow}{rc}(D)=
\left\{
\begin{array}{ll}
n-1 & \textup{\emph{if} }k\le 2\textup{\emph{;}}\\
n & \textup{\emph{if} }k\ge 3.
\end{array}
\right.
\]
Moreover, if $k \ge 3$, then $\overset{\rightarrow}{src}(D)=n$.
\end{thm}

Before we state our results for $\overset{\rightarrow}{rvc}(D)$ and $\overset{\rightarrow}{srvc}(D)$, we make some definitions. For a simple path $P$ and its biorientation $\lrP$, we write $\ell(\lrP)$ for the length of $P$, and say that $\ell(\lrP)$ is the \emph{length} of $\lrP$. Note that $\ell(\lrP)=\textup{diam}(\lrP)=|V(P)|-1$. A vertex-colouring of $\lrP$ is \emph{rainbow} if it is rainbow for $P$. We define the four digraphs $D_1,D_2,D_3,D_4$ as shown in Figure 2.\\[1ex]
\[ \unit = 0.6cm
\varline{850}{4.1}
\ellipse{-9}{2}{2}{2}
\ellipse{-3}{2}{2}{2}
\ellipse{3}{2}{2}{2}
\ellipse{9}{2}{2}{2}
\varline{800}{0.6}
\bez{-10}{3.7}{-9.51}{3.98}{-9}{4}\bez{-8}{0.3}{-7.54}{0.61}{-7.32}{0.98}
\bez{-4}{3.7}{-3.51}{3.98}{-3}{4}\bez{-2}{3.7}{-2.49}{3.98}{-3}{4}\bez{-2}{3.7}{-1.54}{3.39}{-1.32}{3.02}
\bez{2}{3.7}{2.49}{3.98}{3}{4}\bez{4}{3.7}{3.51}{3.98}{3}{4}\bez{4}{0.3}{4.46}{0.61}{4.68}{0.98}
\bez{8}{3.7}{8.49}{3.98}{9}{4}\bez{10}{3.7}{9.51}{3.98}{9}{4}\bez{10}{0.3}{10.46}{0.61}{10.68}{0.98}\bez{10}{0.3}{9.51}{0.02}{9}{0}
\varline{750}{0.6}
\pt{-9}{4}\pt{-10}{3.7}\pt{-7.31}{1}\pt{-8}{0.31}
\pt{-3}{4}\pt{-4}{3.7}\pt{-1.31}{3}\pt{-2}{3.7}
\pt{3}{4}\pt{2}{3.7}\pt{4.69}{1}\pt{4}{3.7}\pt{4}{0.31}
\pt{9}{4}\pt{8}{3.7}\pt{10.69}{1}\pt{10}{3.7}\pt{10}{0.31}\pt{9}{0}
\thnline
\dl{-9.4}{3.94}{-9.68}{3.96}\dl{-9.46}{3.92}{-9.68}{3.96}\dl{-9.5}{3.91}{-9.68}{3.96}\dl{-9.53}{3.9}{-9.68}{3.96}\dl{-9.56}{3.89}{-9.68}{3.96}\dl{-9.59}{3.88}{-9.68}{3.96}
\dl{-9.4}{3.94}{-9.62}{3.77}\dl{-9.46}{3.92}{-9.62}{3.77}\dl{-9.5}{3.91}{-9.62}{3.77}\dl{-9.53}{3.9}{-9.62}{3.77}\dl{-9.56}{3.89}{-9.62}{3.77}\dl{-9.59}{3.88}{-9.62}{3.77}
\dl{-7.7}{0.53}{-7.59}{0.78}\dl{-7.65}{0.58}{-7.59}{0.78}\dl{-7.62}{0.61}{-7.59}{0.78}\dl{-7.6}{0.63}{-7.59}{0.78}\dl{-7.58}{0.65}{-7.59}{0.78}\dl{-7.56}{0.67}{-7.59}{0.78}
\dl{-7.7}{0.53}{-7.45}{0.64}\dl{-7.65}{0.58}{-7.45}{0.64}\dl{-7.62}{0.61}{-7.45}{0.64}\dl{-7.6}{0.63}{-7.45}{0.64}\dl{-7.58}{0.65}{-7.45}{0.64}\dl{-7.56}{0.67}{-7.45}{0.64}
\dl{-3.4}{3.94}{-3.68}{3.96}\dl{-3.46}{3.92}{-3.68}{3.96}\dl{-3.5}{3.91}{-3.68}{3.96}\dl{-3.53}{3.9}{-3.68}{3.96}\dl{-3.56}{3.89}{-3.68}{3.96}\dl{-3.59}{3.88}{-3.68}{3.96}
\dl{-3.4}{3.94}{-3.62}{3.77}\dl{-3.46}{3.92}{-3.62}{3.77}\dl{-3.5}{3.91}{-3.62}{3.77}\dl{-3.53}{3.9}{-3.62}{3.77}\dl{-3.56}{3.89}{-3.62}{3.77}\dl{-3.59}{3.88}{-3.62}{3.77}
\dl{-2.35}{3.87}{-2.63}{3.85}\dl{-2.41}{3.89}{-2.63}{3.85}\dl{-2.45}{3.9}{-2.63}{3.85}\dl{-2.48}{3.91}{-2.63}{3.85}\dl{-2.51}{3.92}{-2.63}{3.85}\dl{-2.54}{3.93}{-2.63}{3.85}
\dl{-2.35}{3.87}{-2.57}{4.04}\dl{-2.41}{3.89}{-2.57}{4.04}\dl{-2.45}{3.9}{-2.57}{4.04}\dl{-2.48}{3.91}{-2.57}{4.04}\dl{-2.51}{3.92}{-2.57}{4.04}\dl{-2.54}{3.93}{-2.57}{4.04}
\dl{-1.53}{3.3}{-1.64}{3.55}\dl{-1.58}{3.35}{-1.64}{3.55}\dl{-1.61}{3.38}{-1.64}{3.55}\dl{-1.63}{3.4}{-1.64}{3.55}\dl{-1.65}{3.42}{-1.64}{3.55}\dl{-1.67}{3.44}{-1.64}{3.55}
\dl{-1.53}{3.3}{-1.78}{3.41}\dl{-1.58}{3.35}{-1.78}{3.41}\dl{-1.61}{3.38}{-1.78}{3.41}\dl{-1.63}{3.4}{-1.78}{3.41}\dl{-1.65}{3.42}{-1.78}{3.41}\dl{-1.67}{3.44}{-1.78}{3.41}
\dl{2.6}{3.94}{2.32}{3.96}\dl{2.54}{3.92}{2.32}{3.96}\dl{2.5}{3.91}{2.32}{3.96}\dl{2.47}{3.9}{2.32}{3.96}\dl{2.44}{3.89}{2.32}{3.96}\dl{2.41}{3.88}{2.32}{3.96}
\dl{2.6}{3.94}{2.38}{3.77}\dl{2.54}{3.92}{2.38}{3.77}\dl{2.5}{3.91}{2.38}{3.77}\dl{2.47}{3.9}{2.38}{3.77}\dl{2.44}{3.89}{2.38}{3.77}\dl{2.41}{3.88}{2.38}{3.77}
\dl{3.65}{3.87}{3.37}{3.85}\dl{3.59}{3.89}{3.37}{3.85}\dl{3.55}{3.9}{3.37}{3.85}\dl{3.52}{3.91}{3.37}{3.85}\dl{3.49}{3.92}{3.37}{3.85}\dl{3.46}{3.93}{3.37}{3.85}
\dl{3.65}{3.87}{3.43}{4.04}\dl{3.59}{3.89}{3.43}{4.04}\dl{3.55}{3.9}{3.43}{4.04}\dl{3.52}{3.91}{3.43}{4.04}\dl{3.49}{3.92}{3.43}{4.04}\dl{3.46}{3.93}{3.43}{4.04}
\dl{4.3}{0.53}{4.41}{0.78}\dl{4.35}{0.58}{4.41}{0.78}\dl{4.38}{0.61}{4.41}{0.78}\dl{4.4}{0.63}{4.41}{0.78}\dl{4.42}{0.65}{4.41}{0.78}\dl{4.44}{0.67}{4.41}{0.78}
\dl{4.3}{0.53}{4.55}{0.64}\dl{4.35}{0.58}{4.55}{0.64}\dl{4.38}{0.61}{4.55}{0.64}\dl{4.4}{0.63}{4.55}{0.64}\dl{4.42}{0.65}{4.55}{0.64}\dl{4.44}{0.67}{4.55}{0.64}
\dl{8.6}{3.94}{8.32}{3.96}\dl{8.54}{3.92}{8.32}{3.96}\dl{8.5}{3.91}{8.32}{3.96}\dl{8.47}{3.9}{8.32}{3.96}\dl{8.44}{3.89}{8.32}{3.96}\dl{8.41}{3.88}{8.32}{3.96}
\dl{8.6}{3.94}{8.38}{3.77}\dl{8.54}{3.92}{8.38}{3.77}\dl{8.5}{3.91}{8.38}{3.77}\dl{8.47}{3.9}{8.38}{3.77}\dl{8.44}{3.89}{8.38}{3.77}\dl{8.41}{3.88}{8.38}{3.77}
\dl{9.65}{3.87}{9.37}{3.85}\dl{9.59}{3.89}{9.37}{3.85}\dl{9.55}{3.9}{9.37}{3.85}\dl{9.52}{3.91}{9.37}{3.85}\dl{9.49}{3.92}{9.37}{3.85}\dl{9.46}{3.93}{9.37}{3.85}
\dl{9.65}{3.87}{9.43}{4.04}\dl{9.59}{3.89}{9.43}{4.04}\dl{9.55}{3.9}{9.43}{4.04}\dl{9.52}{3.91}{9.43}{4.04}\dl{9.49}{3.92}{9.43}{4.04}\dl{9.46}{3.93}{9.43}{4.04}
\dl{9.4}{0.06}{9.68}{0.04}\dl{9.46}{0.08}{9.68}{0.04}\dl{9.5}{0.09}{9.68}{0.04}\dl{9.53}{0.1}{9.68}{0.04}\dl{9.56}{0.11}{9.68}{0.04}\dl{9.59}{0.12}{9.68}{0.04}
\dl{9.6}{0.06}{9.62}{0.23}\dl{9.46}{0.08}{9.62}{0.23}\dl{9.5}{0.09}{9.62}{0.23}\dl{9.53}{0.1}{9.62}{0.23}\dl{9.56}{0.11}{9.62}{0.23}\dl{9.59}{0.12}{9.62}{0.23}
\dl{10.3}{0.53}{10.41}{0.78}\dl{10.35}{0.58}{10.41}{0.78}\dl{10.38}{0.61}{10.41}{0.78}\dl{10.4}{0.63}{10.41}{0.78}\dl{10.42}{0.65}{10.41}{0.78}\dl{10.44}{0.67}{10.41}{0.78}
\dl{10.3}{0.53}{10.55}{0.64}\dl{10.35}{0.58}{10.55}{0.64}\dl{10.38}{0.61}{10.55}{0.64}\dl{10.4}{0.63}{10.55}{0.64}\dl{10.42}{0.65}{10.55}{0.64}\dl{10.44}{0.67}{10.55}{0.64}
\point{-9.35}{-1}{\small$D_1$}\point{-3.35}{-1}{\small$D_2$}\point{2.65}{-1}{\small$D_3$}\point{8.65}{-1}{\small$D_4$}
\point{-11.4}{0.5}{\small$\lrP'$}\point{-7.05}{2.9}{\small$\lrP$}
\point{0.6}{0.5}{\small$\lrP'$}\point{5.1}{2.4}{\small$\lrP$}
\point{6.4}{0.9}{\small$\lrP'$}\point{11.1}{2.4}{\small$\lrP$}
\ptlu{0}{-2.6}{\textup{
\small
\begin{tabular}{c}
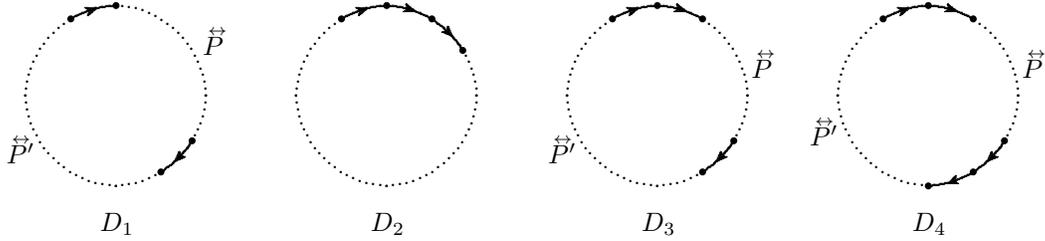
Figure 2. The digraphs $D_1$ to $D_4$.
\end{tabular}
}
}
\]\\[1ex]
In Figure 2, each dotted line represents a bioriented path. The bioriented paths in $D_3$ are non-trivial, but those of $D_1,D_2$ and $D_4$ may be trivial. Note that $D_1,D_3$ and $D_4$ may be one of many possible digraphs, for fixed $n$.

Now, we have the following result. It is worth noting that the existence of an asymmetric arc in $D$ creates a significant difference, with $\overset{\rightarrow}{rvc}(\lrC_n)$ and $\overset{\rightarrow}{srvc}(\lrC_n)$ approximately $\frac{n}{2}$, and $\overset{\rightarrow}{rvc}(D)$ and $\overset{\rightarrow}{srvc}(D)$ approximately $n$.

\begin{thm}\label{cyclethm}
Let $D$ be a spanning strongly connected subdigraph of $\lrC_n$, where $n\ge 3$, and with $k \ge 1$ asymmetric arcs. 
\begin{enumerate}
\item[(a)]
\[
\overset{\rightarrow}{rvc}(D)=
\left\{
\begin{array}{ll}
n-2 & \textup{\emph{if} }k\le 2\textup{\emph{, or} }D=D_2\textup{\emph{, or} }D=D_4\textup{\emph{ with }}n=4\textup{\emph{;}}\\
n-1 & \textup{\emph{if} }D=D_3\textup{\emph{, or} }D=D_4\textup{\emph{ with }}n\ge 5\textup{\emph{;}}\\
n & \textup{\emph{otherwise.}}
\end{array}
\right.
\]
\item[(b)] 
\begin{enumerate}
\item[(i)]
$\overset{\rightarrow}{srvc}(D)=n-2$ if one of the following holds.
\begin{itemize}
\item $k=1$.
\item $D=D_1$ with $n\le 8$, or with $n\ge 9$ and $\ell(\lrP),\ell(\lrP')\le\lfloor\frac{n}{2}\rfloor+1$.
\item $D=D_2$ with $n\le 8$.
\item $D=D_4$ with $n=4$.
\end{itemize}
\item[(ii)] $\overset{\rightarrow}{srvc}(D)=n-1$ if one of the following holds.
\begin{itemize}
\item $D=D_1$ with $n\ge 9$, and $\ell(\lrP)\in\{0,\lfloor\frac{n}{2}\rfloor+2\}$ or $\ell(\lrP')\in\{0,\lfloor\frac{n}{2}\rfloor+2\}$.
\item $D=D_3$ with $5\le n\le 10$, or with $n\ge 11$ and $\ell(\lrP),\ell(\lrP')\le\lfloor\frac{n}{2}\rfloor+1$.
\item $D=D_4$ with $5\le n\le 8$, or with $n\ge 9$ and $\ell(\lrP),\ell(\lrP')\le\lfloor\frac{n}{2}\rfloor$.
\end{itemize}
\item[(iii)] Otherwise, we have $\overset{\rightarrow}{srvc}(D)=n$.
\end{enumerate}
\end{enumerate}
\end{thm}

\begin{proof}
Let $V(D)=V(\lrC_n)=\{v_0,\dots,v_{n-1}\}$, where in $\lrC_n$, there are symmetric arcs between $v_i$ and $v_{i+1}$ for $0\le i\le n-1$ (with $v_n=v_0$). Since $k\ge 1$, we may assume that $v_{n-1}v_0$ is an asymmetric arc of $D$. This implies that $D$ contains all the arcs $v_iv_{i+1}$ for $0\le i\le n-1$, since otherwise there would not exist a $v_0-v_{n-1}$ path in $D$. Thus, the directed cycle $\rC=v_0v_1\cdots v_{n-1}v_0$ is a spanning subdigraph of $D$. Note that $v_0\rC v_{n-1}$ is the only $v_0-v_{n-1}$ path in $D$, so that diam$(D)=n-1$. By (\ref{pro1eq}), we have $n-2\le \overset{\rightarrow}{rvc}(D)\le \overset{\rightarrow}{srvc}(D)\le n$. We now prove two auxiliary claims.

\begin{claim}\label{cycleclm1}
If $D$ is such that $k\neq 1$ and $D\not\in\{D_1,D_2,D_3,D_4\}$, then $\overset{\rightarrow}{rvc}(D)=\overset{\rightarrow}{srvc}(D)=n$.
\end{claim}

\begin{proof}
Firstly, if $(n,k)=(5,5)$, then we have $D=\rC_5$, so that $\overset{\rightarrow}{rvc}(D)=\overset{\rightarrow}{srvc}(D)=5=n$. Now, let $(n,k)\neq (5,5)$. By (\ref{pro1eq}), it suffices to prove that $\overset{\rightarrow}{rvc}(D)\ge n$. Note that $D$ contains three pairwise non-incident asymmetric arcs, say $v_{n-1}v_0$, $v_iv_{i+1}$ and $v_jv_{j+1}$, where $1\le i<j\le n-3$ and $j-i\ge 2$. Suppose that we have a vertex-colouring of $D$ with at most $n-1$ colours. Then two vertices $u,v\in V(D)$ have the same colour. Now $\{u,v\}$ is disjoint from one of $\{v_{n-1},v_0\}$, $\{v_i,v_{i+1}\}$ and $\{v_j,v_{j+1}\}$. Assume without loss of generality that $\{u,v\}$ is disjoint from $\{v_{n-1}, v_0\}$. Then, the only $v_0-v_{n-1}$ path is $v_0\rC v_{n-1}$, which has $u$ and $v$ as internal vertices, and so is not rainbow. Therefore $\overset{\rightarrow}{rvc}(D)\ge n$.
\end{proof}

\begin{claim}\label{cycleclm2}
Let $c$ be a vertex-colouring of $D$ with the following property. For any two vertices $u,v\in V(D)$ such that $c(u)=c(v)$, the subdigraph $D-\{u,v\}$ consists of two components, each of which is a rainbow bioriented path. Then the following hold.
\begin{enumerate}
\item[(a)] $c$ is a rainbow vertex-connected colouring of $D$. 
\item[(b)] If in addition, both bioriented paths of $D-\{u,v\}$ have length at most $\lfloor\frac{n}{2}\rfloor$, then $c$ is a strongly rainbow vertex-connected colouring of $D$.
\end{enumerate}
\end{claim}

\begin{proof}
(a) Let $x,y\in V(D)$, and suppose that there exists an $x-y$ path $P$ in $D$ which is not rainbow. Then, $P$ has internal vertices $u$ and $v$ with $c(u)=c(v)$. The stated property implies that $x$ and $y$ belong to a rainbow bioriented path, and therefore there is a rainbow $x-y$ path. It follows that $c$ is a rainbow vertex-connected colouring of $D$.\\[1ex]
\indent (b) Suppose that $c$ has the additional property as stated. Let $x,y\in V(D)$. Then, either every $x-y$ path in $D$ is rainbow, so there certainly exists a rainbow $x-y$ geodesic. Or, there exists an $x-y$ path in $D$ which is not rainbow. By the previous argument and using the additional property, there is a rainbow $x-y$ path of length at most $\lfloor\frac{n}{2}\rfloor$, which is therefore a geodesic. Hence, $c$ is a strongly rainbow vertex-connected colouring of $D$.
\end{proof}

By Claim \ref{cycleclm1}, it remains to consider $D$ when $k=1$, and when $D\in\{D_1,D_2,D_3,D_4\}$ for the rest of the proof of Theorem \ref{cyclethm}.\\[1ex]
\indent (a) We first consider $D\in\{D_1,D_2\}$. Clearly, $\overset{\rightarrow}{rvc}(D)=1=n-2$ if $n=3$. Now, let $n\ge 4$. We may let the asymmetric arcs of $D_1$ be $v_{n-1}v_0$ and $v_\ell v_{\ell+1}$, for some $0\le\ell\le n-3$, and those of $D_2$ be $v_{n-1}v_0,v_0v_1$ and $v_1v_2$. Define the vertex-colouring $c$ of $D$ with $n-2$ colours as follows. 
\begin{equation}\label{cycleeq1}
\begin{array}{r@{\:\:}c@{\:\:}ll}
c(v_i) & = & i & 
\begin{array}{l}
\textup{for }1\le i\le n-2,
\end{array}
\\
(c(v_{n-1}),c(v_0)) & = &
\left\{
\begin{array}{l}
(\ell,\ell+1)\\
(1,2)
\end{array}
\right.
&
\begin{array}{l}
\textup{if }D=D_1\textup{ and }1\le\ell\le n-3,\\
\textup{otherwise.}
\end{array}
\end{array}
\end{equation}
If $D=D_1$ and $1\le\ell\le n-3$, then both $D_1-\{v_0,v_{\ell+1}\}$ and $D_1-\{v_\ell,v_{n-1}\}$ consists of two rainbow bioriented paths. Otherwise, both $D-\{v_0,v_2\}$ and $D-\{v_1,v_{n-1}\}$ consists of a single vertex and a rainbow bioriented path of length $n-4$. Hence by Claim \ref{cycleclm2}(a), $c$ is a rainbow vertex-connected colouring of $D$. 

Next, if $k=1$, then since $D_1$ is a spanning subdigraph of $D$, by Proposition \ref{pro4}, it follows that $\overset{\rightarrow}{rvc}(D)=n-2$. Also, we have $D_4=\rC_4$ when $n=4$, and $\overset{\rightarrow}{rvc}(D_4)=2=n-2$. This completes the proof of the first part of (a).

Now for the second part, we consider $D_3$ and $D_4$, both with $n\ge 5$. Note that $D_4$ is a spanning subdigraph of $D_3$, if the asymmetric arcs for both digraphs are suitably chosen. By Proposition \ref{pro4}, we have $\overset{\rightarrow}{rvc}(D_3)\le\overset{\rightarrow}{rvc}(D_4)$. For the second part of (a), it suffices to prove that $\overset{\rightarrow}{rvc}(D_3)\ge n-1$ and $\overset{\rightarrow}{rvc}(D_4)\le n-1$. 

For the former, let the three asymmetric arcs of $D_3$ be $v_{n-1}v_0$, $v_0v_1$ and $v_\ell v_{\ell+1}$, for some $2\le\ell\le n-3$. Suppose that we have a rainbow vertex-connected colouring $c$ of $D_3$, with $n-2$ colours. Since the unique $v_0-v_{n-1}$ path is $v_0\rC v_{n-1}$, we may assume that $c(v_i)=i$ for $1\le i\le n-2$. Since the unique $v_1-v_0$ path is $v_1\rC v_0$, this means that we must have $c(v_{n-1})=1$. But then, the unique $v_{\ell+1}-v_\ell$ path is $v_{\ell+1}\rC v_\ell$, which contains $v_1$ and $v_{n-1}$ as internal vertices, and hence is not rainbow, a contradiction. Therefore, $\overset{\rightarrow}{rvc}(D_3)\ge n-1$.

For the latter, let the four asymmetric arcs of $D_4$ be $v_{n-1}v_0$, $v_0v_1$, $v_\ell v_{\ell+1}$ and $v_{\ell+1}v_{\ell+2}$, for some $1\le \ell\le n-3$. Define the vertex-colouring $c$ of $D_4$ where $c(v_i)=i$ for $1\le i\le n-1$, and $c(v_0)=\ell+1$. We can easily apply Claim \ref{cycleclm2}(a), by considering $D_4-\{v_0,v_{\ell+1}\}$, to obtain $\overset{\rightarrow}{rvc}(D_4)\le n-1$. This completes the proof of the second part of (a).

For the third part of (a), if $D$ is exclusive from the first two parts, then $k\neq 1$ and $D\not\in\{D_1,D_2,D_3,D_4\}$. Hence, $\overset{\rightarrow}{rvc}(D)=n$ by Claim \ref{cycleclm1}.\\[1ex]
\indent (b) We observe throughout that for $x,y\in V(D)$, if an $x-y$ path has length at most $\lceil\frac{n}{2}\rceil-1$, then it is the unique $x-y$ geodesic.\\[1ex]
\indent (i) Firstly, for $D=D_4$ with $n=4$, we have $D_4=\rC_4$, and $\overset{\rightarrow}{srvc}(D_4)=2=n-2$. Next, let $k=1$. Clearly, $\overset{\rightarrow}{srvc}(D)=1=n-2$ if $n=3$. Let $n\ge 4$, and $c$ be the vertex-colouring of $D$ where $c(v_i)=i$ for $1\le i\le n-2$, $c(v_0)=\lceil\frac{n}{2}\rceil$, and $c(v_{n-1})=\lceil\frac{n}{2}\rceil-1$. Then, both $D-\{v_0,v_{\lceil n/2 \rceil}\}$ and $D-\{v_{\lceil n/2 \rceil-1},v_{n-1}\}$ consists of two rainbow bioriented paths with length at most $\lceil\frac{n}{2}\rceil-2\le\lfloor\frac{n}{2}\rfloor$.

Now, let $D\in\{D_1,D_2\}$ where in each case, $D$ has the respective condition as stated. Clearly, $\overset{\rightarrow}{srvc}(D)=1=n-2$ if $n=3$. For $n\ge 4$, let the asymmetric arcs of $D$ be as described in the first part of (a). That is, for $D_1$, they are $v_{n-1}v_0$ and $v_\ell v_{\ell+1}$, for some $0\le\ell\le n-3$. Moreover, for $n\ge 9$, the condition $\ell(\lrP),\ell(\lrP')\le\lfloor\frac{n}{2}\rfloor+1$ implies that $\lceil\frac{n}{2}\rceil-3\le\ell\le\lfloor\frac{n}{2}\rfloor+1$. For $D_2$, we have $4\le n\le 8$, and the asymmetric arcs are $v_{n-1}v_0,v_0v_1$ and $v_1v_2$. Consider the colouring $c$ of $D$ as described in (\ref{cycleeq1}). As before, if $D=D_1$ and $1\le\ell\le n-3$, then we consider $D_1-\{v_0,v_{\ell+1}\}$ and $D_1-\{v_\ell,v_{n-1}\}$. Otherwise, consider $D-\{v_0,v_2\}$ and $D-\{v_1,v_{n-1}\}$. If $4\le n\le 8$, then we have two rainbow bioriented paths, each with length at most $n-4\le\lfloor\frac{n}{2}\rfloor$. Otherwise, we have $D=D_1$ with $n\ge 9$ and $\ell(\lrP),\ell(\lrP')\le\lfloor\frac{n}{2}\rfloor+1$, so the two rainbow bioriented paths have length at most $\lfloor\frac{n}{2}\rfloor$. 

In every case, by Claim \ref{cycleclm2}(b), $c$ is a strongly rainbow vertex-connected colouring of $D$. It follows that $\overset{\rightarrow}{srvc}(D)=n-2$.\\[1ex]
\indent (ii) Firstly, let $D=D_1$ with $n\ge 9$ and the stated condition. We may assume that the asymmetric arcs of $D_1$ are $v_{n-1}v_0$ and $v_\ell v_{\ell+1}$, where $\ell\in\{0,\lfloor\frac{n}{2}\rfloor+2\}$. Suppose that we have a strongly rainbow vertex-connected colouring $c$ of $D$, with $n-2$ colours. If $\ell=0$, then as in the argument for $D_3$ in part (a), in view of the unique $v_0-v_{n-1}$ and $v_1-v_0$ paths in $D_1$, we may assume that $c(v_i)=i$ for $1\le i\le n-2$, and $c(v_{n-1})=1$. Now, the unique $v_{n-2}-v_2$ geodesic is $v_{n-2}\rC v_2$, since it has length $4\le\lceil\frac{n}{2}\rceil-1$. But this geodesic is not rainbow, a contradiction. Next, let $\ell=\lfloor\frac{n}{2}\rfloor+2$, and note that $v_{n-1}v_0$ and $v_\ell v_{\ell+1}$ are not incident. Again, the unique $v_0-v_{n-1}$ path in $D_1$ implies that, we may assume $c(v_i)=i$ for $1\le i\le n-2$. Since $v_{\ell+1}\rC v_\ell$ is the unique $v_{\ell+1}-v_\ell$ path, we have $c(v_{n-1})\in\{\ell,\ell+1\}$. Now, the unique $v_{\ell-1}-v_0$ geodesic is $v_{\ell-1}\rC v_0$, since it has length $n-\ell+1=\lceil\frac{n}{2}\rceil-1$. But this geodesic is not rainbow, another contradiction. Hence, $\overset{\rightarrow}{srvc}(D)\ge n-1$. 

Now, let $c$ be the vertex-colouring of $D_1$ where $c(v_i)=i$ for $1\le i\le n-1$, and $c(v_0)=\lfloor\frac{n}{2}\rfloor$ if $\ell=0$, and $c(v_0)=\ell=\lfloor\frac{n}{2}\rfloor+2$ if $\ell=\lfloor\frac{n}{2}\rfloor+2$. If $\ell=0$, then $D_1-\{v_0,v_{\lfloor n/2\rfloor}\}$ consists of two rainbow bioriented paths of length at most $\lceil\frac{n}{2}\rceil-2\le\lfloor\frac{n}{2}\rfloor$. If $\ell=\lfloor\frac{n}{2}\rfloor+2$, then $D_1-\{v_0,v_\ell\}$ consists of two rainbow bioriented paths, one with length $\lfloor\frac{n}{2}\rfloor$, and the other with length $\lceil\frac{n}{2}\rceil-4\le\lfloor\frac{n}{2}\rfloor$. Hence, $\overset{\rightarrow}{srvc}(D_1)\le n-1$ by Claim \ref{cycleclm2}(b).

Next, let $D\in\{D_3,D_4\}$ with $n\ge 5$. By (\ref{pro1eq}) and part (a), we have $\overset{\rightarrow}{srvc}(D)\ge\overset{\rightarrow}{rvc}(D)=n-1$. Now, let $D=D_3$ with the stated condition. Let the asymmetric arcs be $v_{n-1}v_0, v_0v_1$ and $v_\ell v_{\ell+1}$, for some $2\le\ell\le n-3$. Note that for $n\ge 11$, we have $\lceil\frac{n}{2}\rceil-3\le\ell\le\lfloor\frac{n}{2}\rfloor+2$. Let $c$ be the vertex-colouring of $D_3$ where $c(v_i)=i$ for $1\le i\le n-1$. We let $c(v_0)=\ell$, unless if $5\le n\le 10$ and $\ell=2$, or if $n\ge 11$ and $\ell=\lceil\frac{n}{2}\rceil-3$, in which case we let $c(v_0)=\ell+1$. If $c(v_0)=\ell$, then $D_3-\{v_0,v_\ell\}$ consists of two rainbow bioriented paths of length at most $n-5\le\lfloor\frac{n}{2}\rfloor$ if $5\le n\le 10$, and at most $\lfloor\frac{n}{2}\rfloor$ if $n\ge 11$. If $c(v_0)=\ell+1$, then $D_3-\{v_0,v_{\ell+1}\}$ consists of two rainbow bioriented paths, with lengths $1$ and $n-5\le\lfloor\frac{n}{2}\rfloor$ if $5\le n\le 10$, and lengths $\lceil\frac{n}{2}\rceil-4\le\lfloor\frac{n}{2}\rfloor$ and $\lfloor\frac{n}{2}\rfloor$ if $n\ge 11$. Hence, $\overset{\rightarrow}{srvc}(D_3)\le n-1$ by Claim \ref{cycleclm2}(b). Finally, let $D=D_4$ with the stated condition. Let the asymmetric arcs be $v_{n-1}v_0, v_0v_1,v_\ell v_{\ell+1}$ and $v_{\ell+1} v_{\ell+2}$, for some $1\le\ell\le n-3$. Note that for $n\ge 9$, we have $\lceil\frac{n}{2}\rceil-3\le\ell\le\lfloor\frac{n}{2}\rfloor+1$. Let $c$ be the vertex-colouring of $D_4$ where $c(v_i)=i$ for $1\le i\le n-1$, and $c(v_0)=\ell+1$. Then $D_4-\{v_0,v_{\ell+1}\}$ consists of two rainbow bioriented paths of length at most $n-4\le\lfloor\frac{n}{2}\rfloor$ if $5\le n\le 8$, and at most $\lfloor\frac{n}{2}\rfloor$ if $n\ge 9$. Again by Claim \ref{cycleclm2}(b), we have $\overset{\rightarrow}{srvc}(D_4)\le n-1$.\\[1ex]
\indent (iii) By Claim \ref{cycleclm1}, in the remaining cases for $D$, we have $D\in\{D_1,D_2,D_3,D_4\}$. In each case, there is the condition on $D$ which is exclusive from (i) and (ii). In particular, we have $n\ge 9$. Suppose that we have a strongly rainbow vertex-connected colouring $c$ of $D$, using at most $n-1$ colours. Since $v_0\rC v_{n-1}$ is the unique $v_0-v_{n-1}$ path, we may assume that $c(v_i)=i$ for $1\le i\le n-2$. Note that if $D=D_1$, then the two asymmetric arcs are not incident. If $D\in\{D_2,D_3,D_4\}$, then we may assume that $v_0v_1$ is an asymmetric arc. For these cases of $D$, we have $c(v_{n-1})\neq 1$. Otherwise, the unique $v_{n-2}-v_2$ geodesic is $v_{n-2}\rC v_2$, since it has length $4\le\lceil\frac{n}{2}\rceil-1$, but it is not rainbow. Since $v_1\rC v_0$ is the unique $v_1-v_0$ path, we may assume that $c(v_{n-1})=n-1$.

Firstly if $D=D_1$, then we have $n\ge 11$, with $\lrP$ and $\lrP'$ both non-trivial, and either $\ell(\lrP)\ge\lfloor\frac{n}{2}\rfloor+3$ or $\ell(\lrP')\ge\lfloor\frac{n}{2}\rfloor+3$. We may assume that the asymmetric arcs are $v_{n-1}v_0$ and $v_\ell v_{\ell+1}$, where $\lfloor\frac{n}{2}\rfloor+3\le\ell\le n-3$. Note that $c(v_0)\neq c(v_{n-1})$. Otherwise, the unique $v_{n-2}-v_1$ geodesic is $v_{n-2}\rC v_1$, since it has length $3<\lceil\frac{n}{2}\rceil-1$, but it is not rainbow. Since $v_{\ell+1}\rC v_\ell$ is the unique $v_{\ell+1}-v_\ell$ path, it follows that $c(v_0)\in\{\ell,\ell+1\}$ or $c(v_{n-1})\in\{\ell,\ell+1\}$. We have $v_{\ell-1}\rC v_1$ is the unique $v_{\ell-1}-v_1$ geodesic, since it has length $n-\ell+2\le \lceil\frac{n}{2}\rceil-1$. But this geodesic is not rainbow.

Next, if $D=D_2$, then we have $n\ge 9$. Since $v_2\rC v_1$ is the unique $v_2-v_1$ path, it follows that $c(v_0)\in\{1,2\}$. But then, the unique $v_{n-1}-v_3$ geodesic is $v_{n-1}\rC v_3$, since it has length $4\le \lceil\frac{n}{2}\rceil-1$, and it is not rainbow. 

Now, let $D=D_3$, so that $n\ge 11$, and either $\ell(\lrP)\ge\lfloor\frac{n}{2}\rfloor+2$ or $\ell(\lrP')\ge\lfloor\frac{n}{2}\rfloor+2$. We may assume that the asymmetric arcs are $v_{n-1}v_0, v_0v_1$ and $v_\ell v_{\ell+1}$, where either $\lfloor\frac{n}{2}\rfloor+3\le\ell\le n-3$ or $2\le\ell\le \lceil\frac{n}{2}\rceil-4$. Since $v_{\ell+1}\rC v_\ell$ is the unique $v_{\ell+1}-v_\ell$ path, it follows that $c(v_0)\in\{\ell,\ell+1\}$. If $\lfloor\frac{n}{2}\rfloor+3\le\ell\le n-3$, then $v_{\ell-1}\rC v_1$ is the unique $v_{\ell-1}-v_1$ geodesic, since it has length $n-\ell+2\le\lceil\frac{n}{2}\rceil-1$. If $2\le\ell\le \lceil\frac{n}{2}\rceil-4$, then $v_{n-1}\rC v_{\ell+2}$ is the unique $v_{n-1}-v_{\ell+2}$ geodesic, since it has length $\ell+3\le\lceil\frac{n}{2}\rceil-1$. In either case, the unique geodesic in consideration is not rainbow.

Finally, let $D=D_4$, so that $n\ge 9$, and either $\ell(\lrP)\ge\lfloor\frac{n}{2}\rfloor+1$ or $\ell(\lrP')\ge\lfloor\frac{n}{2}\rfloor+1$. We may assume that the asymmetric arcs are $v_{n-1}v_0, v_0v_1,v_\ell v_{\ell+1}$ and $v_{\ell+1}v_{\ell+2}$, where $\lfloor\frac{n}{2}\rfloor+2\le\ell\le n-3$. We may similarly consider the unique $v_{\ell+1}-v_\ell$ and $v_{\ell+2}-v_{\ell+1}$ paths, and deduce that $c(v_0)=\ell+1$. But then, the unique $v_\ell-v_1$ geodesic is $v_\ell \rC v_1$, since it has length $n-\ell+1\le \lceil\frac{n}{2}\rceil-1$, and it is not rainbow. 

In all cases, we have a contradiction, and hence $\overset{\rightarrow}{srvc}(D)=n$.\\[1ex]
\indent This completes the proof of Theorem \ref{cyclethm}.
\end{proof}

Our next aim is to consider the (strong) rainbow vertex-connection numbers of circulant digraphs. For an integer $n\geq 3$ and a set $S\subset \{1,2,\ldots,n-1\}$, the \emph{circulant digraph} $C_{n}(S)$ is defined to be the digraph with vertex set $V(C_{n}(S))=\{v_{0},v_{1},\ldots,v_{n-1}\}$ and arc set $A(C_{n}(S))=\{v_{i}v_{j}: j-i \equiv s \textup{ (mod $n$), }s\in S\}$. Throughout, the subscripts of the vertices are taken cyclically modulo $n$. An arc $v_{i}v_{j}$, where $j-i \equiv s$ (mod $n$) for $s\in S$, is called an \emph{$s$-jump}. Given an integer $k\geq 1$, let $[k]=\{1, 2,\ldots, k\}$. In \cite{JJ2015}, Alva-Samos and Montellano-Ballesteros proved the following result.

\begin{thm}\textup{\cite{JJ2015}}
If $1\le k\le n-2$, then $\overset{\rightarrow}{rc}(C_n([k]))=\overset{\rightarrow}{src}(C_n([k]))=\lceil\frac{n}{k}\rceil$.
\end{thm}

Here, we shall similarly consider $\overset{\rightarrow}{rvc}(C_n([k]))$ and $\overset{\rightarrow}{srvc}(C_n([k]))$. This task turns out to be somewhat more complicated. Note firstly that if $\lfloor\frac{n}{2}\rfloor\le k\le n-2$, then diam$(C_n([k]))=2$, so that by Theorem \ref{pro2}(b), we have $\overset{\rightarrow}{rvc}(C_n([k]))=\overset{\rightarrow}{srvc}(C_n([k]))=1$. Also, when $k=1$, we have $C_n([1])=\rC_n$, so that $\overset{\rightarrow}{rvc}(C_n([1]))=\overset{\rightarrow}{srvc}(C_n([1]))=n-2$ for $n=3,4$, and $\overset{\rightarrow}{rvc}(C_n([1]))=\overset{\rightarrow}{srvc}(C_n([1]))=n$ for $n\ge 5$, by Proposition \ref{cyclepro}. Hence, we may restrict our attention to $2\le k\le\lfloor\frac{n}{2}\rfloor-1$.

\begin{thm}\label{circulantthm}
Let $2\le k\le\lfloor\frac{n}{2}\rfloor-1$.
\begin{enumerate}
\item[(a)] Let $n\not\equiv 0,1\textup{ (mod }k)$. Then $\overset{\rightarrow}{rvc}(C_n([k]))=\lceil\frac{n}{k}\rceil-1$, and $\overset{\rightarrow}{srvc}(C_n([k]))=\lceil\frac{n}{k}\rceil$, where the latter holds for $n$ sufficiently large.
\item[(b)] Let $n=ak+1$, where $a\ge 3$.
\begin{enumerate}
\item[(i)] If $a-1\mid n$, then $\overset{\rightarrow}{rvc}(C_n([k]))=\overset{\rightarrow}{srvc}(C_n([k]))=\frac{n-1}{k}-1$.
\item[(ii)] If $a-1\nmid n$, then $\overset{\rightarrow}{rvc}(C_n([k]))=\frac{n-1}{k}$, and
\[
\overset{\rightarrow}{srvc}(C_n([k]))=
\left\{
\begin{array}{ll}
\frac{n-1}{k} & \textup{\emph{if} }a<k+2\textup{\emph{,}}\\[1ex]
\frac{n-1}{k}+1=\lceil\frac{n}{k}\rceil & \textup{\emph{if} }a>k+2.
\end{array}
\right.
\]
\end{enumerate}
\item[(c)] Let $n=ak$, where $a\ge 3$. 
\begin{enumerate}
\item[(i)] If $a=3,4$, then $\overset{\rightarrow}{rvc}(C_n([k]))=\overset{\rightarrow}{srvc}(C_n([k]))=\frac{n}{k}-1$.
\item[(ii)] If $a\ge 5$, then $\overset{\rightarrow}{rvc}(C_n([k]))\in\{\frac{n}{k}-1,\frac{n}{k}\}$, and $\overset{\rightarrow}{srvc}(C_n([k]))=\frac{n}{k}$.
\end{enumerate}
\end{enumerate}
\end{thm}

\begin{proof}
For convenience, we write $D=C_n([k])$ throughout. We first derive some facts about $D$, and auxiliary bounds for $\overset{\rightarrow}{rvc}(D)$ and $\overset{\rightarrow}{srvc}(D)$. Note that for $x,y\in V(D)$, an $x-y$ geodesic is obtained by taking as many successive $k$-jumps as possible, starting from $x$, until we either meet $y$, or some other vertex $w\neq y$ from which we can reach $y$ with a ``remainder'' $k'$-jump, for some $1\le k'\le k-1$. We denote this geodesic by $G(x,y)$. If $x=v_i$ and $y=v_j$ for some $0\le i,j\le n-1$, then $d(x,y)=\lceil\frac{j-i}{k}\rceil$ if $i\le j$, and $d(x,y)=\lceil\frac{j+n-i}{k}\rceil$ if $i>j$. In both cases, the length of $G(x,y)$ is $d(x,y)$. From this, it follows that diam$(D)=\lceil\frac{n-1}{k}\rceil$.
Next, we define the vertex-colouring $c$ of $D$, using colours $0,1,\dots,\lceil\frac{n}{k}\rceil-1$, where $c(v_i)=\lfloor\frac{i}{k}\rfloor$ for $0\le i\le n-1$. Note that for every $i$, by taking the $k$-jump from $v_i$ to $v_{i+k}$, the colour is increased by $1$, modulo $\lceil\frac{n}{k}\rceil$. For $x,y\in V(D)$, since the number of internal vertices of $G(x,y)$ is at most $\lceil\frac{n-1}{k}\rceil-1<\lceil\frac{n}{k}\rceil$, and these vertices are connected by successive $k$-jumps, we have $G(x,y)$ is a rainbow $x-y$ geodesic. Hence, $c$ is a strongly rainbow vertex-connected colouring for $D$. By (\ref{pro1eq}), we have
\begin{equation}
\Big\lceil\frac{n-1}{k}\Big\rceil-1\le\overset{\rightarrow}{rvc}(D)\le\overset{\rightarrow}{srvc}(D)\le\Big\lceil\frac{n}{k}\Big\rceil.\label{circulanteq1}
\end{equation}

Next, we prove a stronger upper bound for $\overset{\rightarrow}{rvc}(D)$.

\begin{claim}\label{circulantclm1}
If $n\not\equiv 0$\textup{ (mod }$k)$, then
\begin{equation}
\overset{\rightarrow}{rvc}(D)\le\Big\lceil\frac{n}{k}\Big\rceil-1.\label{circulanteq2}
\end{equation}
\end{claim}

\begin{proof}
Let $n=ak+b$ for some $a\ge 2$ and $1\le b\le k-1$, and $\gcd(n,k)=g$. We have $g$ vertex-disjoint directed cycles $\rC^0,\rC^1,\dots,\rC^{g-1}$, where each cycle is a copy of $\rC_{n/g}$, with $\rC^r=v_rv_{r+k}v_{r+2k}\cdots v_{r+(n/g-1)k}v_r$ for every $0\le r\le g-1$. Observe that each $\rC^r$ consists of $\frac{n}{g}$ $k$-jumps of $D$. Furthermore, if we identify the indices of the vertices of $D$ with the cyclic group $\mathbb Z_n$, then the indices of the vertices of $\rC^0$ form the cyclic subgroup $\mathbb Z_{n/g}$, which is generated by $k$, and also by $g$. For $0\le r\le g-1$, the indices of the vertices of $\rC^r$ can be identified with the coset $\mathbb Z_{n/g}+r$. Now, let $c$ be the vertex-colouring of $D$, using colours $0,1,\dots,a-1$, where $c(v_{r+\ell k})\equiv \ell$ (mod $a)$ for $0\le \ell\le \frac{n}{g}-1$ and $0\le r\le g-1$. Since $a=\lceil\frac{n}{k}\rceil-1$, the claim follows if we can show that $c$ is a rainbow vertex-connected colouring for $D$. 

For $0\le r\le g-1$, let $U_r=\{v_{r+(n/g-s)k}: a\ge s\ge 2\}\subset V(\rC^r)$, and $V_r=V(\rC^r)\setminus U_r$. Note that any $a$ consecutive vertices of $\rC^r$ have distinct colours if the first vertex is in $(V_r\setminus\{v_{r+(n/g-1)k}\})\cup\{v_{r+(n/g-a)k}\}$. Let $x,y\in V(D)$, with $x\in V(\rC^r)$ and $xy\not\in A(D)$. Note that $G(x,y)$ has at most $\lceil\frac{n-1}{k}\rceil-1\le a$ internal vertices. Thus if $x\in V_r$, then $G(x,y)$ is a rainbow $x-y$ geodesic. Otherwise, let $x\in U_r$. Consider taking a $(k-g)$-jump from $x$ to $x'$. We have $x'\in V(\rC^r)$, since the indices of $x$ and $x'$ both belong to the coset $\mathbb Z_{n/g}+r$. Also, consider the directed cycle $\rC=v_0v_1\cdots v_{n-1}v_0$, and the positions of the vertices of $U_r\cup\{v_{r+(n/g-1)k}\}$ in $\rC$. Any two consecutive vertices $v_{r+(n/g-s)k}$ and $v_{r+(n/g-s+1)k}$, for some $a\ge s\ge 2$, are connected by $k$ arcs of $\rC$, so that $v_{r+(n/g-1)k}$ is connected to $v_{r+(n/g-a)k}$ by $n-(a-1)k>k$ arcs of $\rC$. It follows that $x'\not\in U_r\cup\{v_{r+(n/g-1)k}\}$, and thus 
$x'\in V_r\setminus\{v_{r+(n/g-1)k}\}$. Now, note that since $g\mid n-ak=b$, $g\mid k$ and $b<k$, we have $b\le k-g$. This means that $G(x',y)-y$ has at most $\lceil\frac{n-1-(k-g)}{k}\rceil=a-1+\lceil\frac{b+g-1}{k}\rceil\le a$ vertices. Therefore, $xx'\cup G(x',y)$ is a rainbow $x-y$ path. This proves Claim \ref{circulantclm1}.
\end{proof}

Now, we continue with the proof of Theorem \ref{circulantthm}.\\[1ex]
\indent (a) Let $n\not\equiv 0,1\textup{ (mod }k)$. Since $\lceil\frac{n-1}{k}\rceil=\lceil\frac{n}{k}\rceil$ if and only if $n\not\equiv 1$ (mod $k)$, we instantly have $\overset{\rightarrow}{rvc}(D)=\lceil\frac{n}{k}\rceil-1$ by (\ref{circulanteq1}) and (\ref{circulanteq2}). To prove the second part, it suffices to show that $\overset{\rightarrow}{srvc}(D)\ge\lceil\frac{n}{k}\rceil$ for $n$ sufficiently large, by (\ref{circulanteq1}). Suppose that we have a vertex-colouring of $D$, using at most $\lceil\frac{n}{k}\rceil-1$ colours. Let $\gcd(n,k)=g$ and $k=k'g$. We have the directed cycle $\rC=v_0v_kv_{2k}\cdots v_{(n/g-1)k}v_0$, which is a copy of $\rC_{n/g}$. Since $k'(\lceil\frac{n}{k}\rceil-1)<\frac{n}{g}$, some colour occurs at least $k'+1$ times in $\rC$. Without loss of generality, $v_k$ and $v_{\ell k}$ have the same colour, for some $2\le \ell\le\lfloor\frac{n}{(k'+1)g}\rfloor+1$. Since $k\nmid n$, we have $v_0v_kv_{2k}\cdots v_{\lfloor n/k\rfloor k}$ is the unique $v_0-v_{\lfloor n/k\rfloor k}$ geodesic. This is not rainbow if $\lfloor\frac{n}{(k'+1)g}\rfloor+1\le\lfloor\frac{n}{k}\rfloor-1$, which is true for $n\ge 2k(k+1)$.\\[1ex]
\indent (b) Let $n=ak+1$, where $a\ge 3$. Note that since $\gcd(n,k)=1$, we have the directed cycle $\rC=v_0v_kv_{2k}\cdots v_{(n-1)k}v_0$, which is a copy of $\rC_n$ in $D$ consisting of all the $k$-jumps. Observe that for $x,y\in V(D)$, the vertices of $G(x,y)-y$ form a consecutive set of at most $\frac{n-1}{k}=a$ vertices of $\rC$, and they are connected by successive $k$-jumps.\\[1ex]
\indent (i) Suppose that $a-1\mid n$. Let $c$ be the vertex-colouring of $D$ using colours $0,1,\dots,a-2$, where $c(v_{\ell k})\equiv \ell$ (mod $a-1)$ for $0\le\ell\le n-1$. Then, as we visit the vertices $v_0, v_k, v_{2k}, \dots$ along $\rC$, the colours appear in blocks of $0, 1, 2, \dots, a-2$, repeating $\frac{n}{a-1}$ times, so that any $a-1$ consecutive vertices of $\rC$ have distinct colours. Now, for $x,y\in V(D)$, the number of internal vertices of $G(x,y)$ is at most $a-1$, so it follows that $G(x,y)$ is a rainbow $x-y$ geodesic. Therefore, $\overset{\rightarrow}{srvc}(D)\le a-1=\frac{n-1}{k}-1$, and part (i) follows from (\ref{circulanteq1}).\\[1ex]
\indent (ii) Let $a-1\nmid n$. Suppose that we have a rainbow vertex-connected colouring of $D$, using $\frac{n-1}{k}-1=a-1$ colours. Note that for any $0\le i\le n-1$, we have $G(v_i,v_{i-1})=v_iv_{i+k}v_{i+2k}\cdots v_{i+ak}$ is the unique $v_i-v_{i-1}$ geodesic, which contains $a-1$ internal vertices, and so must be rainbow. This means that any $a-1$ consecutive vertices of $\rC$ must have distinct colours, which is a contradiction since $a-1\nmid n$. Thus, $\overset{\rightarrow}{srvc}(D)\ge\overset{\rightarrow}{rvc}(D)\ge \frac{n-1}{k}$. Since $\lceil\frac{n}{k}\rceil-1=\frac{n-1}{k}$, we have $\overset{\rightarrow}{rvc}(D)=\frac{n-1}{k}$ by (\ref{circulanteq2}).
%

Now, let $a<k+2$. We define the vertex-colouring $c$ of $D$, using colours $0,1,\dots, a-1$, as follows. First, let $c(v_0)=c(v_{(a-1)k})=c(v_{2(a-1)k})=\cdots=c(v_{(a-1)k^2})=0$. We then colour the remaining vertices with colours $1,2,\dots, a-1, 1,2, \dots, a-1,\dots$, starting from $v_k, v_{2k},\dots$ and going around $\rC$, omitting the vertices already coloured with $0$. Note that colour $0$ appears $k+1$ times, and every other colour appears $k$ times. In $\rC$, the distance from $v_{(a-1)k^2}$ to $v_0$ is $n-(a-1)k=k+1>a-1$. It follows that any two vertices with the same colour have distance at least $a-1$ between them in $\rC$, and hence any $a-1$ consecutive vertices of $\rC$ have distinct colours. As before, if $x,y\in V(D)$, then $G(x,y)$ is a rainbow $x-y$ geodesic. We have $\overset{\rightarrow}{srvc}(D)\le a=\frac{n-1}{k}$. 

Finally, let $a>k+2$. Suppose that we have a vertex-colouring of $D$, using at most $\frac{n-1}{k}=a$ colours. Then, some colour occurs at least $\lceil\frac{n}{a}\rceil=k+1$ times, so that there are two vertices in this colour within a distance of $\lfloor\frac{n}{k+1}\rfloor=\lfloor a-\frac{a-1}{k+1}\rfloor\le a-2$ in $\rC$. Thus, we may assume that $v_k$ and $v_{\ell k}$ have the same colour, for some $2\le \ell\le a-1$. But then, the unique $v_0-v_{n-1}$ geodesic is $G(v_0,v_{n-1})$, which is not rainbow since it contains $v_k$ and $v_{\ell k}$ as internal vertices. Therefore, $\overset{\rightarrow}{srvc}(D)\ge\frac{n-1}{k}+1=\lceil\frac{n}{k}\rceil$, and we have equality by (\ref{circulanteq1}). This completes the proof of part (ii), and of (b).\\[1ex]
\indent (c) Let $n=ak$, where $a\ge 3$.\\[1ex]
\indent (i) Let $a=3,4$. By (\ref{circulanteq1}), it suffices to show that $\overset{\rightarrow}{srvc}(D)\le \frac{n}{k}-1$. Let $V_\ell=\{v_i:\ell k\le i\le (\ell+1)k-1\}$ for $0\le\ell\le a-1$. Define the vertex-colouring $c$ of $D$, using colours $0,1,\dots,a-2$, as follows. For $a=3$, let $c(v_0)=c(v_k)=0$ and $c(v_{2k})=1$. For $a=4$, let $c(v_0)=c(v_{2k})=0$, $c(v_k)=1$ and $c(v_{3k})=2$. Then for $0\le\ell\le a-1$ and $1\le r\le k-1$, let $c(v_{\ell k+r})\equiv c(v_{\ell k})+r$ (mod $a-1)$. We show that $c$ is a strongly rainbow-vertex connected colouring for $D$. Let $v_i,v_j\in V(D)$. We will show that there is a rainbow $v_i-v_j$ geodesic. 

For $a=3$, it suffices to consider $d(v_i,v_j)=3$. Thus, we have $v_j\in\{v_{i+2k+1},\dots,v_{i+3k-1}\}$. Note that a $k$-jump starting at a vertex in $V_1\cup V_2$ connects two vertices of distinct colours. The same is true for a $(k-1)$-jump starting at a vertex in $V_0\setminus\{v_0\}$. If $v_i\in V_0\cup V_1$, we can take $G(v_i,v_j)=v_iv_{i+k}v_{i+2k}v_j$, and note that $v_{i+k}\in V_1\cup V_2$. If $v_i=v_{2k}$, we take $v_{2k}v_{3k-1}v_{4k-1}v_j$, and note that $v_{3k-1}\in V_2$. If $v_i\in V_2\setminus\{v_{2k}\}$, then we take $v_iv_{i+k}v_{i+2k-1}v_j$, and note that $v_{i+k}\in V_0\setminus\{v_0\}$. In every case, we have a rainbow $v_i-v_j$ geodesic. Hence, $\overset{\rightarrow}{srvc}(D)\le 2=\frac{n}{k}-1$.

Now, let $a=4$. If $d(v_i,v_j)=3$, then $G(v_i,v_j)=v_iv_{i+k}v_{i+2k}v_j$ is a rainbow $v_i-v_j$ geodesic. Thus, it remains to consider $d(v_i,v_j)=4$, so that $v_j\in\{v_{i+3k+1},\dots,v_{i+4k-1}\}$. If $v_i\in V_0\cup V_2$, then we take $G(v_i,v_j)=v_iv_{i+k}v_{i+2k}v_{i+3k}v_j$. If $v_i\in\{v_k,v_{3k}\}$, then we take $v_iv_{i+k-1}v_{i+2k-1}v_{i+3k-1}v_j$. If $v_i\in V_1\setminus\{v_k\}$, then we take $v_iv_{i+k}v_{i+2k-1}v_{i+3k-1}v_j$. Note that $c(v_{i+k})\equiv i-k$, $c(v_{i+2k-1})\equiv i-k+1$, and $c(v_{i+3k-1})\equiv i-k-1$ (mod $3)$. If $v_i\in V_3\setminus\{v_{3k}\}$, then we take $v_iv_{i+k}v_{i+2k}v_{i+3k-1}v_j$. Note that $c(v_{i+k})\equiv i$, $c(v_{i+2k})\equiv i+1$, and $c(v_{i+3k-1})\equiv i-1$ (mod $3)$. In every case, we have a rainbow $v_i-v_j$ geodesic. Hence, $\overset{\rightarrow}{srvc}(D)\le 3=\frac{n}{k}-1$.\\[1ex]
\indent (ii) Let $a\ge 5$. Clearly by (\ref{circulanteq1}), we have $\overset{\rightarrow}{rvc}(D)\in\{\frac{n}{k}-1,\frac{n}{k}\}$. It remains and suffices to show that  $\overset{\rightarrow}{srvc}(D)\ge \frac{n}{k}$, by (\ref{circulanteq1}). Let $\rC=v_0v_kv_{2k}\cdots$ $v_{(n/k-1)k}v_0$, which is copy of $\rC_{n/k}$. Suppose that we have a vertex-colouring of $D$, using $\frac{n}{k}-1$ colours. By Proposition \ref{cyclepro}, this colouring is not rainbow vertex-connected for $\rC$ if $n\ge 5k$, so there exist $x,y\in V(\rC)$ with no rainbow $x-y$ path in $\rC$. In $D$, the unique $x-y$ geodesic is $x\rC y$, which is not rainbow. It follows that $\overset{\rightarrow}{srvc}(D)\ge \frac{n}{k}$.
\end{proof}

In Theorem \ref{circulantthm}, we have been unable to determine $\overset{\rightarrow}{rvc}(C_n([k]))$ when $n\equiv 0$ (mod $k)$, and $\overset{\rightarrow}{srvc}(C_n([k]))$ for small $n\not\equiv 0,1$ (mod $k)$. The first of these two tasks appears more interesting, and we leave it as an open problem.

\begin{prob}
Let $n=ak$, where $k\ge 2$ and $a\ge 5$. Determine $\overset{\rightarrow}{rvc}(C_n([k]))$.
\end{prob}

\section{Tournaments}\label{tournamentssect}

We now study the (strong) rainbow vertex-connection numbers of tournaments. Our first aim is to consider the range of values that the (strong) rainbow vertex-connection number can take, over all strongly connected tournaments of a given order $n\ge 3$. Clearly, ${\rC}_3$ is the only such tournament of order $3$, and $\overset{\rightarrow}{rvc}({\rC}_3)=\overset{\rightarrow}{srvc}({\rC}_3)=1$. For $n=4$, there is also a unique tournament, namely $T_4$, which is the union of ${\rC}_4=v_0v_1v_2v_3v_0$ and the arcs $v_0v_2$, $v_1v_3$. We have $\overset{\rightarrow}{rvc}(T_4)=\overset{\rightarrow}{srvc}(T_4)=2$.

For the rainbow connection analogue, we have $\overset{\rightarrow}{rc}({\rC}_3)=\overset{\rightarrow}{rc}(T_4)=3$. For the rest of this section, we mainly restrict our attention to $n\ge 5$. Dorbec et al.~\cite{PIE2014} proved the following results.

\begin{thm}\label{Dorbecetaltourthm1}\textup{\cite{PIE2014}}
If $T$ is a strongly connected tournament with $n\ge 5$ vertices, then $2\le\overset{\rightarrow}{rc}(T)\le n-1$.
\end{thm}

\begin{thm}\label{Dorbecetaltourthm2}\textup{\cite{PIE2014}}
For every $n$ and $k$ such that $3\le k\le n-1$, there exists a tournament $T_{n,k}$ on $n$ vertices such that $\overset{\rightarrow}{rc}(T_{n,k})=k$.
\end{thm}

They also remarked that there does not exist a tournament on $4$ or $5$ vertices with rainbow connection number $2$, and that such a tournament exists if the order is $8$ (mod $12)$. In response, Alva-Samos and Montellano-Ballesteros \cite{JJJ2015} proved the following result.

\begin{thm}\label{ASMBtourthm}\textup{\cite{JJJ2015}}
For every $n\ge 6$, there exists a tournament $T_n$ on $n$ vertices such that $\overset{\rightarrow}{rc}(T_n)=2$.
\end{thm}

In Theorems \ref{tournamentthm1} and \ref{tournamentthm2} below, we have the (strong) rainbow vertex-connection versions of the above results.

\begin{thm}\label{tournamentthm1}
If $T$ is a strongly connected tournament on $n\ge 3$ vertices, then $1\le \overset{\rightarrow}{rvc}(T)\le \overset{\rightarrow}{srvc}(T)\le n-2$.
\end{thm}

\begin{proof}
The theorem holds for $n=3,4$. Now, let $n\ge 5$. Clearly, we have $\overset{\rightarrow}{srvc}(T)\ge\overset{\rightarrow}{rvc}(T)\ge 1$ by (\ref{pro1eq}), so it remains to prove that $\overset{\rightarrow}{srvc}(T)\le n-2$. Let $d=\textup{diam}(T)$. If $d=2$, then $\overset{\rightarrow}{srvc}(T)=1$ by Theorem \ref{pro2}(b), and we are done. Now, let $d\ge 3$. Let $u,v\in V(T)$ be such that $d(u,v)=d$, and let $P$ be a $u-v$ geodesic. Let $u',v'\in V(P)$ be the vertices where $uv', u'v\in A(P)$, and note that $u,u',v,v'$ are distinct. Then, it is easy to see that $d(u,u')=d(v',v)=d-1$. We let $c$ be the vertex-colouring of $T$ such that $c(u)=c(u')=1$, $c(v)=c(v')=2$, and the remaining vertices are given the distinct colours $3,4,\dots,n-2$. We claim that $c$ is a strongly rainbow vertex-connected colouring for $T$. We consider two cases.\\[1ex]
\emph{Case 1.} $d\ge 4$.\\[1ex]
\indent Let $x,y\in V(T)$, and $Q$ be an $x-y$ geodesic. Suppose that $Q$ is not rainbow, so that $|V(Q)|\ge 4$. Then, either $u,u'\in V(Q-\{x,y\})$, or $v,v'\in V(Q-\{x,y\})$. Assume the former; if the latter holds, then a similar argument applies, with $v',v$ in place of $u,u'$. If the vertices $x,u,u',y$ occur in this order along $Q$, then we have $d(u,u')\le d(x,y)-2\le d-2$, which contradicts $d(u,u')=d-1$. Now, let the vertices $x,u',u,y$ occur in this order along $Q$. Since $Q$ is an $x-y$ geodesic, we have $uw\in A(T)$, where $w$ is the vertex with $wu'\in A(Q)$ (note that we may possibly have $w=x$). Then $uwu'$ is a $u-u'$ path of length $2$ in $T$, which is a contradiction to $d(u,u')=d-1\ge 3$. Hence, $Q$ is a rainbow $x-y$ geodesic.\\[1ex]
\emph{Case 2.} $d=3$.\\[1ex]
\indent In this case, we have $P=uv'u'v$, and $u'u, vv',vu\in A(T)$. Let $x,y\in V(T)$, and $Q$ be an $x-y$ geodesic. Suppose that $Q$ is not rainbow, so that $d(x,y)=3$. We claim that there exists another $x-y$ geodesic in $T$ which is rainbow.\\[1ex]
\emph{Subcase 2.1.} $u,u'\in V(Q-\{x,y\})$.\\[1ex]
\indent Since $u'u\in A(T)$, we have $Q=xu'uy$. We have $xu',u'v\in A(T)$, so that $x\neq v$, and $vu,uy\in A(T)$, so that $y\neq v$. Note that we may possibly have $x=v'$ or $y=v'$. Also, we have $vy\in A(T)$, otherwise if $yv\in A(T)$, then $uyv$ is a $u-v$ path of length $2$, a contradiction. Hence, we may take $xu'vy$.\\[1ex]
\emph{Subcase 2.2.} $v,v'\in V(Q-\{x,y\})$.\\[1ex]
\indent Since $vv'\in A(T)$, we have $Q=xvv'y$. We have $xv,vu\in A(T)$, so that $x\neq u$, and $uv',v'y\in A(T)$, so that $y\neq u$. Note that we may possibly have $x=u'$ or $y=u'$. Also, we have $xu\in A(T)$, otherwise if $ux\in A(T)$, then $uxv$ is a $u-v$ path of length $2$, a contradiction. Hence, we may take $xuv'y$.\\[1ex]
\indent In both Subcases 2.1 and 2.2, we have found a rainbow $x-y$ geodesic of length $3$. We conclude that in all cases, $c$ is a strongly rainbow vertex-connected colouring for $T$. Therefore $\overset{\rightarrow}{srvc}(T)\le n-2$, as required.
\end{proof}

\begin{thm}\label{tournamentthm2}
For $n\ge 5$ and $1\leq k \leq n-2$, there exists a tournament $T_{n,k}$ on $n$ vertices  such that 
\begin{equation}\label{tournamenteq1}
\overset{\rightarrow}{rvc}(T_{n,k})=\overset{\rightarrow}{srvc}(T_{n,k})=k.
\end{equation}
\end{thm}

%
%

\begin{proof}
We first consider the case $k=1$. By Theorem \ref{pro2}(b), it suffices to find a tournament $T_{n,1}$ on $n\ge 5$ vertices such that $\overset{\rightarrow}{rvc}(T_{n,1})=1$. By Theorems \ref{pro2}(b) and \ref{ASMBtourthm}, we see that such a tournament $T_{n,1}$ exists for every $n\ge 6$. For $n=5$, we let $T_{5,1}$ be the tournament which is the union of two copies of ${\rC}_5$: $v_0v_1v_2v_3v_4v_0$ and $v_0v_3v_1v_4v_2v_0$. Then we have $\overset{\rightarrow}{rvc}(T_{5,1})=1$.

Now, let $k\ge 2$. We first consider the case $n=k+2$, and construct a tournament $T_{k+2,k}$ such that (\ref{tournamenteq1}) holds. Let $V(T_{k+2,k})=\{v_0,\dots,v_{k+1}\}$, and $A(T_{k+2,k})=\{v_iv_{i+1}:0\le i\le k\}\cup\{v_jv_i:0\le i<j\le k+1$ and $j-i\ge 2\}$. Since $v_0\cdots v_{k+1}$ is the only $v_0-v_{k+1}$ path in $T_{k+2,k}$, we have diam$(T_{k+2,k})\ge k+1$, and thus $\overset{\rightarrow}{rvc}(T_{k+2,k})\ge k$ by (\ref{pro1eq}). Now, consider the vertex-colouring $c$ of $T_{k+2,k}$ where $c(v_i)=i$ for $1\le i\le k$, and $c(v_0)=c(v_{k+1})=1$. Let $v_i,v_j\in V(T_{k+2,k})$. If $j>i$, then $v_iv_{i+1}\cdots v_j$ is the unique $v_i-v_j$ path, and if $j\le i-2$, then $v_iv_j\in A(T_{k+2,k})$. If $j=i-1$, then $v_iv_j\not\in A(T_{k+2,k})$, and a $v_i-v_j$ path is $v_iv_{i-2}v_{i-1}$ if $2\le i\le k+1$, and $v_1v_2v_0$ if $i=1$. Hence, we always have a rainbow $v_i-v_j$ geodesic, so that $\overset{\rightarrow}{srvc}(T_{k+2,k})\le k$. Therefore, (\ref{tournamenteq1}) holds for $T_{k+2,k}$, by (\ref{pro1eq}).

Finally, let $n>k+2$. We construct a tournament $T_{n,k}$ such that (\ref{tournamenteq1}) holds. Let $T$ be a tournament on $n-k-1$ vertices, and let $T_{n,k}=(T_{k+2,k})_{v_0\to T}$, which is obtained from $T_{k+2,k}$ by expanding $v_0$ to $T$. Then, note that we have diam$(T_{n,k})\ge k+1$, which again gives $\overset{\rightarrow}{rvc}(T_{n,k})\ge k$. Now, we extend the colouring $c$ on $T_{k+2,k}$ to a colouring $c'$ on $T_{n,k}$ by letting $c'(v)=c(v)$ if $v\neq v_0$, and $c'(v)=1$ if $v\in V(T)$ in $T_{n,k}$. We claim that $c'$ is a strongly rainbow vertex-connected colouring for $T_{n,k}$. Let $x,y\in V(T_{n,k})$. A similar argument as for $T_{k+2,k}$ shows that, there exists a rainbow $x-y$ geodesic in $T_{n,k}$, if we do not have $x,y\in V(T)$. Now, let $x,y\in V(T)$. Then, we have exactly one of the following three situations.
\begin{itemize}
\item $xy\in A(T)\subset A(T_{n,k})$.
\item $xy\not\in A(T)$, and there exists $z\in V(T)$ such that $xzy$ is a path in $T\subset T_{n,k}$.
\item $xy\not\in A(T)$, and there does not exist $z\in V(T)$ such that $xzy$ is a path in $T\subset T_{n,k}$. Then, $xv_1v_2y$ is a rainbow $x-y$ geodesic in $T_{n,k}$, since $k\ge 2$.
\end{itemize}
In each case, we have a rainbow $x-y$ geodesic in $T_{n,k}$, and the claim holds. Therefore, $\overset{\rightarrow}{srvc}(T_{n,k})\le k$. Again (\ref{tournamenteq1}) holds, by (\ref{pro1eq}).
\end{proof}

In \cite{PIE2014}, Dorbec et al.~also proved the following result, which shows that the rainbow connection number of a tournament is in fact closely related to its diameter.

\begin{thm}\label{Dorbecetaltourthm3}\textup{\cite{PIE2014}}
Let $T$ be a tournament of diameter $d$. We have $d\le\overset{\rightarrow}{rc}(T)\le d+2$.
\end{thm}

Here, we shall use a similar method and prove an analogous result for the rainbow vertex-connection number, as follows.

\begin{thm}\label{tournamentthm3}
Let $T$ be a tournament of diameter $d$. We have $d-1\le\overset{\rightarrow}{rvc}(T)\le d+3$.
\end{thm}

\begin{proof}
The theorem holds for $d=2$, since then we have $\overset{\rightarrow}{rvc}(T)=1$ by Theorem \ref{pro2}(b). Now, let $d\ge 3$. Clearly, the lower bound follows from (\ref{pro1eq}). We prove the upper bound by constructing a vertex-colouring $c$ of $T$, using colours $1,2,\dots,d-1,\alpha,\beta,\gamma,\delta$. As in \cite{PIE2014}, we consider the following decomposition of $T$. Let $a\in V(T)$ be a vertex with eccentricity $d$, i.e., there exists a vertex of $T$ at distance $d$ from $a$. For $1\le i\le d$, let $V_i$ denote the set of vertices at distance $i$ from $a$. Then, every $V_i$ is non-empty. Note that $ua\in A(T)$ for every $u\in V_i$ with $2\le i\le d$, and $uv\in A(T)$ whenever $u\in V_i$, $v\in V_j$ and $i-j\ge 2$. Let $p\in V_1$ have maximum in-degree in $T[V_1]$, and $q\in V_d$ have maximum out-degree in $T[V_d]$. Now, we define the vertex-colouring $c$ as follows.
\begin{itemize}
\item $c(v)=i$ if $v\in V_i$, for $2\le i\le d-1$.
\item $c(p)=\alpha$, and $c(v)=1$ for $v\in V_1\setminus\{p\}$.
\item $c(a)=c(q)=\beta$, and $c(v)=d-2$ for $v\in V_d\setminus\{q\}$.
\end{itemize}
Next, we update $c$ by considering a $p-q$ geodesic $P$. Note that we have $|A(P)|=d-1$ or $|A(P)|=d$. Thus, we consider these two cases.
\begin{enumerate}
\item[(i)] If $|A(P)|=d-1$, then $P$ must contain one arc from $V_i$ to $V_{i+1}$, for every $1\le i\le d-1$. Let $r$ be the vertex of $P$ in $V_{d-2}$. We recolour $r$ by letting $c(r)=\gamma$.
\item[(ii)] If $|A(P)|=d$, then $P$ must contain one arc from $V_i$ to $V_{i+1}$, for every $1\le i\le d-1$, and one arc from $V_k$ to $V_k$, for some $1\le k\le d$. If $k\neq d-2$, then let $r$ be the vertex of $P$ in $V_{d-2}$, and $st$ be the arc of $P$ in $V_k$. If $k=d-2$, then let $sr$ be the arc of $P$ in $V_{d-2}$. In either case, we recolour $r$ and $s$ by letting $c(r)=\gamma$ and $c(s)=\delta$.
\end{enumerate}
The situation in (ii) is shown in Figure 3.\\[1ex]
\[ \unit = 0.6cm
\pt{-9.6}{3}
\pt{-7}{1}\pt{-7}{4}\pt{-7}{5}
\pt{-4.4}{1}\pt{-4.4}{4}\pt{-4.4}{5}
\pt{0}{1}\pt{0}{2}\pt{0}{4}\pt{0}{5}
\pt{4.4}{2}\pt{4.4}{4}\pt{4.4}{5}
\pt{7}{2}\pt{7}{4}\pt{7}{5}
\pt{9.6}{2}\pt{9.6}{4}\pt{9.6}{5}
\point{-10.35}{2.88}{\small $a$}\ellipse{-10.2}{3}{0.3}{0.3}
\point{-7.13}{0.43}{\small $p$}\ellipse{-7}{0.5}{0.3}{0.3}
\point{-0.35}{0.38}{\small $s$}\ellipse{-0.2}{0.5}{0.3}{0.3}
\point{0.09}{2.32}{\small $t$}\ellipse{0.2}{2.5}{0.3}{0.3}
\point{4.27}{1.36}{\small $r$}\ellipse{4.4}{1.5}{0.3}{0.3}
\point{9.95}{1.92}{\small $q$}\ellipse{10.1}{2}{0.3}{0.3}
\point{1.93}{1.36}{\small $P$}
\ptlu{-9.31}{3.22}{\beta}\ptlu{-7}{5}{1}\ptld{-7.21}{3.9}{1}\ptlu{-6.75}{1.08}{\alpha}\ptlu{-4.4}{1.05}{2}\ptld{-4.4}{3.9}{2}\ptlu{-4.41}{5}{2}\ptlr{0.06}{1}{\delta}\ptll{-0.1}{2}{k}\ptld{0}{3.9}{k}\ptlu{0}{5}{k}\ptld{4.4}{3.9}{d-2}\ptlu{4}{5}{d-2}\ptlu{4.4}{2.08}{\gamma}\ptlu{7}{2}{d-1}\ptld{7}{3.9}{d-1}\ptlu{7.4}{5}{d-1}\ptlu{10}{5}{d-2}\ptld{9.6}{3.9}{d-2}\ptlu{9.5}{2.04}{\beta}
\varline{500}{0.6}
\dl{-7.3}{0}{-6.7}{0}\dl{-7.3}{6}{-6.7}{6}\dl{-8}{0.7}{-8}{5.3}\dl{-6}{0.7}{-6}{5.3}\bez{-8}{0.7}{-8}{0}{-7.3}{0}\bez{-6}{0.7}{-6}{0}{-6.7}{0}\bez{-8}{5.3}{-8}{6}{-7.3}{6}\bez{-6}{5.3}{-6}{6}{-6.7}{6}
\dl{-4.7}{0}{-4.1}{0}\dl{-4.7}{6}{-4.1}{6}\dl{-5.4}{0.7}{-5.4}{5.3}\dl{-3.4}{0.7}{-3.4}{5.3}\bez{-5.4}{0.7}{-5.4}{0}{-4.7}{0}\bez{-3.4}{0.7}{-3.4}{0}{-4.1}{0}\bez{-5.4}{5.3}{-5.4}{6}{-4.7}{6}\bez{-3.4}{5.3}{-3.4}{6}{-4.1}{6}
\dl{-0.3}{0}{0.3}{0}\dl{-0.3}{6}{0.3}{6}\dl{-1}{0.7}{-1}{5.3}\dl{1}{0.7}{1}{5.3}\bez{-1}{0.7}{-1}{0}{-0.3}{0}\bez{1}{0.7}{1}{0}{0.3}{0}\bez{-1}{5.3}{-1}{6}{-0.3}{6}\bez{1}{5.3}{1}{6}{0.3}{6}
\dl{4.7}{0}{4.1}{0}\dl{4.7}{6}{4.1}{6}\dl{5.4}{0.7}{5.4}{5.3}\dl{3.4}{0.7}{3.4}{5.3}\bez{5.4}{0.7}{5.4}{0}{4.7}{0}\bez{3.4}{0.7}{3.4}{0}{4.1}{0}\bez{5.4}{5.3}{5.4}{6}{4.7}{6}\bez{3.4}{5.3}{3.4}{6}{4.1}{6}
\dl{7.3}{0}{6.7}{0}\dl{7.3}{6}{6.7}{6}\dl{8}{0.7}{8}{5.3}\dl{6}{0.7}{6}{5.3}\bez{8}{0.7}{8}{0}{7.3}{0}\bez{6}{0.7}{6}{0}{6.7}{0}\bez{8}{5.3}{8}{6}{7.3}{6}\bez{6}{5.3}{6}{6}{6.7}{6}
\dl{9.9}{0}{9.3}{0}\dl{9.9}{6}{9.3}{6}\dl{10.6}{0.7}{10.6}{5.3}\dl{8.6}{0.7}{8.6}{5.3}\bez{10.6}{0.7}{10.6}{0}{9.9}{0}\bez{8.6}{0.7}{8.6}{0}{9.3}{0}\bez{10.6}{5.3}{10.6}{6}{9.9}{6}\bez{8.6}{5.3}{8.6}{6}{9.3}{6}
\point{-7.3}{6.2}{\small $V_1$}\point{-4.7}{6.2}{\small $V_2$}\point{-0.3}{6.2}{\small $V_k$}\point{3.82}{6.2}{\small $V_{d-2}$}\point{6.42}{6.2}{\small $V_{d-1}$}\point{9.3}{6.2}{\small $V_d$}
\point{-2.53}{3.96}{\small $\dots$}\point{1.86}{3.96}{\small $\dots$}
\dl{-5.6}{-1}{6.6}{-1}\bez{-9.6}{3}{-9.6}{-1}{-5.6}{-1}\bez{9.6}{2}{9.6}{-1}{6.6}{-1}
\bez{-9.6}{3}{-9.2}{9.7}{-4.4}{5}
\dl{-0.7}{7}{3.3}{7}\bez{7}{5}{5.3}{7}{3.3}{7}\bez{-4.4}{5}{-2.7}{7}{-0.7}{7}
\dl{-9.6}{3}{-7}{1}\dl{-9.6}{3}{-7}{4}\dl{-9.6}{3}{-7}{5}\dl{-7}{1}{-7}{5}\dl{-7}{5}{-4.4}{5}\dl{-7}{4}{-4.4}{4}\dl{-7}{4}{-4.4}{5}\dl{-4.4}{4}{-4.4}{5}\dl{-7}{4}{-4.4}{1}\dl{-7}{1}{0}{4}\dl{0}{4}{0}{5}\dl{7}{4}{7}{5}\dl{7}{5}{9.6}{5}\dl{7}{5}{9.6}{4}\dl{7}{4}{4.4}{5}\dl{4.4}{5}{7}{5}\dl{4.4}{4}{7}{4}\bez{4.4}{5}{7}{9}{9.6}{5}
\varline{850}{3.8}
\dl{-7}{1}{0}{1}\dl{0}{1}{0}{2}\dl{0}{2}{9.6}{2}
\thnline
\dl{-1.63}{1}{-1.89}{1.1}\dl{-1.67}{1}{-1.89}{1.1}\dl{-1.71}{1}{-1.89}{1.1}\dl{-1.75}{1}{-1.89}{1.1}\dl{-1.79}{1}{-1.89}{1.1}\dl{-1.83}{1}{-1.89}{1.1}
\dl{-1.63}{1}{-1.89}{0.9}\dl{-1.67}{1}{-1.89}{0.9}\dl{-1.71}{1}{-1.89}{0.9}\dl{-1.75}{1}{-1.89}{0.9}\dl{-1.79}{1}{-1.89}{0.9}\dl{-1.83}{1}{-1.89}{0.9}
\dl{-1.63}{1}{-1.83}{1}
\dl{-2.51}{1}{-2.77}{1.1}\dl{-2.55}{1}{-2.77}{1.1}\dl{-2.59}{1}{-2.77}{1.1}\dl{-2.63}{1}{-2.77}{1.1}\dl{-2.67}{1}{-2.77}{1.1}\dl{-2.71}{1}{-2.77}{1.1}
\dl{-2.51}{1}{-2.77}{0.9}\dl{-2.55}{1}{-2.77}{0.9}\dl{-2.59}{1}{-2.77}{0.9}\dl{-2.63}{1}{-2.77}{0.9}\dl{-2.67}{1}{-2.77}{0.9}\dl{-2.71}{1}{-2.77}{0.9}
\dl{-2.51}{1}{-2.71}{1}
\dl{-5.57}{1}{-5.83}{1.1}\dl{-5.61}{1}{-5.83}{1.1}\dl{-5.65}{1}{-5.83}{1.1}\dl{-5.69}{1}{-5.83}{1.1}\dl{-5.73}{1}{-5.83}{1.1}\dl{-5.77}{1}{-5.83}{1.1}
\dl{-5.57}{1}{-5.83}{0.9}\dl{-5.61}{1}{-5.83}{0.9}\dl{-5.65}{1}{-5.83}{0.9}\dl{-5.69}{1}{-5.83}{0.9}\dl{-5.73}{1}{-5.83}{0.9}\dl{-5.77}{1}{-5.83}{0.9}
\dl{-5.57}{1}{-5.77}{1}
\dl{1.89}{2}{1.63}{2.1}\dl{1.85}{2}{1.63}{2.1}\dl{1.81}{2}{1.63}{2.1}\dl{1.77}{2}{1.63}{2.1}\dl{1.73}{2}{1.63}{2.1}\dl{1.69}{2}{1.63}{2.1}
\dl{1.89}{2}{1.63}{1.9}\dl{1.85}{2}{1.63}{1.9}\dl{1.81}{2}{1.63}{1.9}\dl{1.77}{2}{1.63}{1.9}\dl{1.73}{2}{1.63}{1.9}\dl{1.69}{2}{1.63}{1.9}
\dl{1.89}{2}{1.69}{2}
\dl{2.77}{2}{2.51}{2.1}\dl{2.73}{2}{2.51}{2.1}\dl{2.69}{2}{2.51}{2.1}\dl{2.65}{2}{2.51}{2.1}\dl{2.61}{2}{2.51}{2.1}\dl{2.57}{2}{2.51}{2.1}
\dl{2.77}{2}{2.51}{1.9}\dl{2.73}{2}{2.51}{1.9}\dl{2.69}{2}{2.51}{1.9}\dl{2.65}{2}{2.51}{1.9}\dl{2.61}{2}{2.51}{1.9}\dl{2.57}{2}{2.51}{1.9}
\dl{2.77}{2}{2.57}{2}
\dl{5.83}{2}{5.57}{2.1}\dl{5.79}{2}{5.57}{2.1}\dl{5.75}{2}{5.57}{2.1}\dl{5.71}{2}{5.57}{2.1}\dl{5.67}{2}{5.57}{2.1}\dl{5.63}{2}{5.57}{2.1}
\dl{5.83}{2}{5.57}{1.9}\dl{5.79}{2}{5.57}{1.9}\dl{5.75}{2}{5.57}{1.9}\dl{5.71}{2}{5.57}{1.9}\dl{5.67}{2}{5.57}{1.9}\dl{5.63}{2}{5.57}{1.9}
\dl{5.83}{2}{5.63}{2}
\dl{8.43}{2}{8.17}{2.1}\dl{8.39}{2}{8.17}{2.1}\dl{8.35}{2}{8.17}{2.1}\dl{8.31}{2}{8.17}{2.1}\dl{8.27}{2}{8.17}{2.1}\dl{8.23}{2}{8.17}{2.1}
\dl{8.43}{2}{8.17}{1.9}\dl{8.39}{2}{8.17}{1.9}\dl{8.35}{2}{8.17}{1.9}\dl{8.31}{2}{8.17}{1.9}\dl{8.27}{2}{8.17}{1.9}\dl{8.23}{2}{8.17}{1.9}
\dl{8.43}{2}{8.23}{2}
\dl{-0.13}{-1}{0.13}{-0.9}\dl{-0.09}{-1}{0.13}{-0.9}\dl{-0.05}{-1}{0.13}{-0.9}\dl{-0.01}{-1}{0.13}{-0.9}\dl{0.03}{-1}{0.13}{-0.9}\dl{0.07}{-1}{0.13}{-0.9}
\dl{-0.13}{-1}{0.13}{-1.1}\dl{-0.09}{-1}{0.13}{-1.1}\dl{-0.05}{-1}{0.13}{-1.1}\dl{-0.01}{-1}{0.13}{-1.1}\dl{0.03}{-1}{0.13}{-1.1}\dl{0.07}{-1}{0.13}{-1.1}
\dl{1.17}{7}{1.43}{6.9}\dl{1.21}{7}{1.43}{6.9}\dl{1.25}{7}{1.43}{6.9}\dl{1.29}{7}{1.43}{6.9}\dl{1.33}{7}{1.43}{6.9}\dl{1.37}{7}{1.43}{6.9}
\dl{1.17}{7}{1.43}{7.1}\dl{1.21}{7}{1.43}{7.1}\dl{1.25}{7}{1.43}{7.1}\dl{1.29}{7}{1.43}{7.1}\dl{1.33}{7}{1.43}{7.1}\dl{1.37}{7}{1.43}{7.1}
\dl{6.87}{7}{7.13}{6.9}\dl{6.91}{7}{7.13}{6.9}\dl{6.95}{7}{7.13}{6.9}\dl{6.99}{7}{7.13}{6.9}\dl{7.03}{7}{7.13}{6.9}\dl{7.07}{7}{7.13}{6.9}
\dl{6.87}{7}{7.13}{7.1}\dl{6.91}{7}{7.13}{7.1}\dl{6.95}{7}{7.13}{7.1}\dl{6.99}{7}{7.13}{7.1}\dl{7.03}{7}{7.13}{7.1}\dl{7.07}{7}{7.13}{7.1}
\dl{0}{1.63}{0.1}{1.37}\dl{0}{1.56}{0.1}{1.37}\dl{0}{1.51}{0.1}{1.37}\dl{0}{1.47}{0.1}{1.37}\dl{0}{1.45}{0.1}{1.37}\dl{0}{1.43}{0.1}{1.37}
\dl{0}{1.63}{-0.1}{1.37}\dl{0}{1.56}{-0.1}{1.37}\dl{0}{1.51}{-0.1}{1.37}\dl{0}{1.47}{-0.1}{1.37}\dl{0}{1.45}{-0.1}{1.37}\dl{0}{1.43}{-0.1}{1.37}
\dl{0}{1.43}{0}{1.63}
\dl{7}{4.37}{7.1}{4.63}\dl{7}{4.44}{7.1}{4.63}\dl{7}{4.49}{7.1}{4.63}\dl{7}{4.53}{7.1}{4.63}\dl{7}{4.55}{7.1}{4.63}\dl{7}{4.57}{7.1}{4.63}
\dl{7}{4.37}{6.9}{4.63}\dl{7}{4.44}{6.9}{4.63}\dl{7}{4.49}{6.9}{4.63}\dl{7}{4.53}{6.9}{4.63}\dl{7}{4.55}{6.9}{4.63}\dl{7}{4.57}{6.9}{4.63}
\dl{0}{4.37}{0.1}{4.63}\dl{0}{4.44}{0.1}{4.63}\dl{0}{4.49}{0.1}{4.63}\dl{0}{4.53}{0.1}{4.63}\dl{0}{4.55}{0.1}{4.63}\dl{0}{4.57}{0.1}{4.63}
\dl{0}{4.37}{-0.1}{4.63}\dl{0}{4.44}{-0.1}{4.63}\dl{0}{4.49}{-0.1}{4.63}\dl{0}{4.53}{-0.1}{4.63}\dl{0}{4.55}{-0.1}{4.63}\dl{0}{4.57}{-0.1}{4.63}
\dl{-7}{4.37}{-7.1}{4.63}\dl{-7}{4.44}{-7.1}{4.63}\dl{-7}{4.49}{-7.1}{4.63}\dl{-7}{4.53}{-7.1}{4.63}\dl{-7}{4.55}{-7.1}{4.63}\dl{-7}{4.57}{-7.1}{4.63}
\dl{-7}{4.37}{-6.9}{4.63}\dl{-7}{4.44}{-6.9}{4.63}\dl{-7}{4.49}{-6.9}{4.63}\dl{-7}{4.53}{-6.9}{4.63}\dl{-7}{4.55}{-6.9}{4.63}\dl{-7}{4.57}{-6.9}{4.63}
\dl{-4.4}{4.63}{-4.3}{4.37}\dl{-4.4}{4.56}{-4.3}{4.37}\dl{-4.4}{4.51}{-4.3}{4.37}\dl{-4.4}{4.47}{-4.3}{4.37}\dl{-4.4}{4.45}{-4.3}{4.37}\dl{-4.4}{4.43}{-4.3}{4.37}
\dl{-4.4}{4.63}{-4.5}{4.37}\dl{-4.4}{4.56}{-4.5}{4.37}\dl{-4.4}{4.51}{-4.5}{4.37}\dl{-4.4}{4.47}{-4.5}{4.37}\dl{-4.4}{4.45}{-4.5}{4.37}\dl{-4.4}{4.43}{-4.5}{4.37}
\dl{-7}{2.63}{-6.9}{2.37}\dl{-7}{2.56}{-6.9}{2.37}\dl{-7}{2.51}{-6.9}{2.37}\dl{-7}{2.47}{-6.9}{2.37}\dl{-7}{2.45}{-6.9}{2.37}\dl{-7}{2.43}{-6.9}{2.37}
\dl{-7}{2.63}{-7.1}{2.37}\dl{-7}{2.56}{-7.1}{2.37}\dl{-7}{2.51}{-7.1}{2.37}\dl{-7}{2.47}{-7.1}{2.37}\dl{-7}{2.45}{-7.1}{2.37}\dl{-7}{2.43}{-7.1}{2.37}
\dl{-5.57}{5}{-5.83}{5.1}\dl{-5.61}{5}{-5.83}{5.1}\dl{-5.65}{5}{-5.83}{5.1}\dl{-5.69}{5}{-5.83}{5.1}\dl{-5.73}{5}{-5.83}{5.1}\dl{-5.77}{5}{-5.83}{5.1}
\dl{-5.57}{5}{-5.83}{4.9}\dl{-5.61}{5}{-5.83}{4.9}\dl{-5.65}{5}{-5.83}{4.9}\dl{-5.69}{5}{-5.83}{4.9}\dl{-5.73}{5}{-5.83}{4.9}\dl{-5.77}{5}{-5.83}{4.9}
\dl{-5.83}{4}{-5.57}{3.9}\dl{-5.79}{4}{-5.57}{3.9}\dl{-5.75}{4}{-5.57}{3.9}\dl{-5.71}{4}{-5.57}{3.9}\dl{-5.67}{4}{-5.57}{3.9}\dl{-5.63}{4}{-5.57}{3.9}
\dl{-5.83}{4}{-5.57}{4.1}\dl{-5.79}{4}{-5.57}{4.1}\dl{-5.75}{4}{-5.57}{4.1}\dl{-5.71}{4}{-5.57}{4.1}\dl{-5.67}{4}{-5.57}{4.1}\dl{-5.63}{4}{-5.57}{4.1}
\dl{5.57}{4}{5.83}{3.9}\dl{5.61}{4}{5.83}{3.9}\dl{5.65}{4}{5.83}{3.9}\dl{5.69}{4}{5.83}{3.9}\dl{5.73}{4}{5.83}{3.9}\dl{5.77}{4}{5.83}{3.9}
\dl{5.57}{4}{5.83}{4.1}\dl{5.61}{4}{5.83}{4.1}\dl{5.65}{4}{5.83}{4.1}\dl{5.69}{4}{5.83}{4.1}\dl{5.73}{4}{5.83}{4.1}\dl{5.77}{4}{5.83}{4.1}
\dl{8.17}{5}{8.43}{4.9}\dl{8.21}{5}{8.43}{4.9}\dl{8.25}{5}{8.43}{4.9}\dl{8.29}{5}{8.43}{4.9}\dl{8.33}{5}{8.43}{4.9}\dl{8.37}{5}{8.43}{4.9}
\dl{8.17}{5}{8.43}{5.1}\dl{8.21}{5}{8.43}{5.1}\dl{8.25}{5}{8.43}{5.1}\dl{8.29}{5}{8.43}{5.1}\dl{8.33}{5}{8.43}{5.1}\dl{8.37}{5}{8.43}{5.1}
\dl{5.83}{5}{5.57}{5.1}\dl{5.79}{5}{5.57}{5.1}\dl{5.75}{5}{5.57}{5.1}\dl{5.71}{5}{5.57}{5.1}\dl{5.67}{5}{5.57}{5.1}\dl{5.63}{5}{5.57}{5.1}
\dl{5.83}{5}{5.57}{4.9}\dl{5.79}{5}{5.57}{4.9}\dl{5.75}{5}{5.57}{4.9}\dl{5.71}{5}{5.57}{4.9}\dl{5.67}{5}{5.57}{4.9}\dl{5.63}{5}{5.57}{4.9}
\dl{-8.2}{4.08}{-8.46}{4}\dl{-8.26}{4.04}{-8.46}{4}\dl{-8.3}{4.02}{-8.46}{4}\dl{-8.32}{4}{-8.46}{4}\dl{-8.34}{3.98}{-8.46}{4}\dl{-8.36}{3.96}{-8.46}{4}
\dl{-8.2}{4.08}{-8.34}{3.86}\dl{-8.26}{4.04}{-8.34}{3.86}\dl{-8.3}{4.02}{-8.34}{3.86}\dl{-8.32}{4}{-8.34}{3.86}\dl{-8.34}{3.98}{-8.34}{3.86}\dl{-8.36}{3.96}{-8.34}{3.86}
\dl{-8.2}{1.92}{-8.46}{2}\dl{-8.26}{1.96}{-8.46}{2}\dl{-8.3}{1.98}{-8.46}{2}\dl{-8.32}{2}{-8.46}{2}\dl{-8.34}{2.02}{-8.46}{2}\dl{-8.36}{2.04}{-8.46}{2}
\dl{-8.2}{1.92}{-8.34}{2.14}\dl{-8.26}{1.96}{-8.34}{2.14}\dl{-8.3}{1.98}{-8.34}{2.14}\dl{-8.32}{2}{-8.34}{2.14}\dl{-8.34}{2.02}{-8.34}{2.14}\dl{-8.36}{2.04}{-8.34}{2.14}
\dl{-8.18}{3.54}{-8.43}{3.53}\dl{-8.24}{3.52}{-8.43}{3.53}\dl{-8.28}{3.5}{-8.43}{3.53}\dl{-8.31}{3.49}{-8.43}{3.53}\dl{-8.34}{3.48}{-8.43}{3.53}\dl{-8.37}{3.47}{-8.43}{3.53}
\dl{-8.18}{3.54}{-8.37}{3.38}\dl{-8.24}{3.52}{-8.37}{3.38}\dl{-8.28}{3.5}{-8.37}{3.38}\dl{-8.31}{3.49}{-8.37}{3.38}\dl{-8.34}{3.48}{-8.37}{3.38}\dl{-8.37}{3.47}{-8.37}{3.38}
\dl{-5.82}{4.46}{-5.57}{4.47}\dl{-5.76}{4.48}{-5.57}{4.47}\dl{-5.72}{4.5}{-5.57}{4.47}\dl{-5.69}{4.51}{-5.57}{4.47}\dl{-5.66}{4.52}{-5.57}{4.47}\dl{-5.63}{4.53}{-5.57}{4.47}
\dl{-5.82}{4.46}{-5.63}{4.62}\dl{-5.76}{4.48}{-5.63}{4.62}\dl{-5.72}{4.5}{-5.63}{4.62}\dl{-5.69}{4.51}{-5.63}{4.62}\dl{-5.66}{4.52}{-5.63}{4.62}\dl{-5.63}{4.53}{-5.63}{4.62}
\dl{5.82}{4.46}{5.57}{4.47}\dl{5.76}{4.48}{5.57}{4.47}\dl{5.72}{4.5}{5.57}{4.47}\dl{5.69}{4.51}{5.57}{4.47}\dl{5.66}{4.52}{5.57}{4.47}\dl{5.63}{4.53}{5.57}{4.47}
\dl{5.82}{4.46}{5.63}{4.62}\dl{5.76}{4.48}{5.63}{4.62}\dl{5.72}{4.5}{5.63}{4.62}\dl{5.69}{4.51}{5.63}{4.62}\dl{5.66}{4.52}{5.63}{4.62}\dl{5.63}{4.53}{5.63}{4.62}
\dl{8.42}{4.46}{8.17}{4.47}\dl{8.36}{4.48}{8.17}{4.47}\dl{8.32}{4.5}{8.17}{4.47}\dl{8.29}{4.51}{8.17}{4.47}\dl{8.26}{4.52}{8.17}{4.47}\dl{8.23}{4.53}{8.17}{4.47}
\dl{8.42}{4.46}{8.23}{4.62}\dl{8.36}{4.48}{8.23}{4.62}\dl{8.32}{4.5}{8.23}{4.62}\dl{8.29}{4.51}{8.23}{4.62}\dl{8.26}{4.52}{8.23}{4.62}\dl{8.23}{4.53}{8.23}{4.62}
\dl{-5.62}{2.41}{-5.72}{2.64}\dl{-5.67}{2.47}{-5.72}{2.64}\dl{-5.7}{2.51}{-5.72}{2.64}\dl{-5.73}{2.54}{-5.72}{2.64}\dl{-5.75}{2.56}{-5.72}{2.64}
\dl{-5.62}{2.41}{-5.83}{2.53}\dl{-5.67}{2.47}{-5.83}{2.53}\dl{-5.7}{2.51}{-5.83}{2.53}\dl{-5.73}{2.54}{-5.83}{2.53}\dl{-5.75}{2.56}{-5.83}{2.53}
\dl{-3.2}{2.63}{-3}{2.81}\dl{-3.14}{2.65}{-3}{2.81}\dl{-3.09}{2.67}{-3}{2.81}\dl{-3.06}{2.69}{-3}{2.81}\dl{-3.04}{2.7}{-3}{2.81}
\dl{-3.02}{2.71}{-3}{2.81}
\dl{-3.2}{2.63}{-2.94}{2.65}\dl{-3.14}{2.65}{-2.94}{2.65}\dl{-3.09}{2.67}{-2.94}{2.65}\dl{-3.06}{2.69}{-2.94}{2.65}\dl{-3.04}{2.7}{-2.94}{2.65}\dl{-3.02}{2.71}{-2.94}{2.65}
\dl{-8.2}{6.81}{-7.95}{6.82}\dl{-8.14}{6.83}{-7.95}{6.82}\dl{-8.1}{6.85}{-7.95}{6.82}\dl{-8.07}{6.86}{-7.95}{6.82}\dl{-8.04}{6.87}{-7.95}{6.82}\dl{-8.01}{6.88}{-7.95}{6.82}
\dl{-8.2}{6.81}{-8.01}{6.97}\dl{-8.14}{6.83}{-8.01}{6.97}\dl{-8.1}{6.85}{-8.01}{6.97}\dl{-8.07}{6.86}{-8.01}{6.97}\dl{-8.04}{6.87}{-8.01}{6.97}\dl{-8.01}{6.88}{-8.01}{6.97}
\ptlu{0}{-3.5}{\textup{
\small
\begin{tabular}{c}
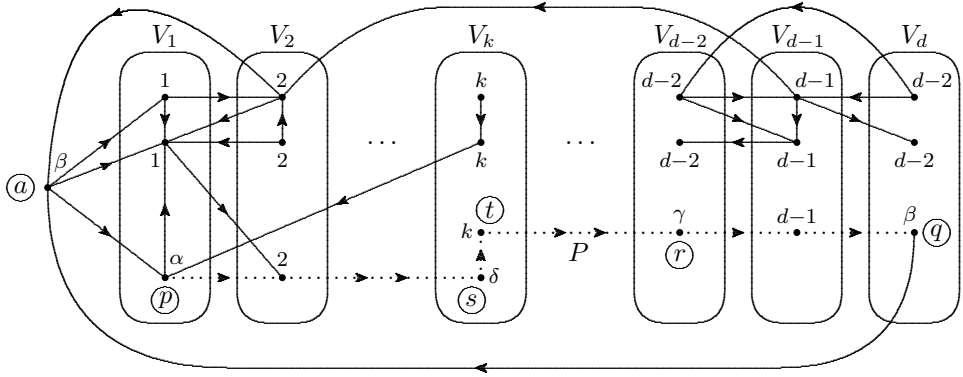
Figure 3. The decomposition and colouring of $T$. The dotted line\\ 
represents the $p-q$ geodesic $P$. Letters that label vertices are circled. 
\end{tabular}
}
}
\]\\[1ex]
\indent Now, we prove that $c$ is a rainbow vertex-connected colouring for $T$, which implies the theorem. Let $x,y\in V(T)$. We show that there always exists a rainbow $x-y$ path. This is true if $x=a$. Let $x\not\in V_1\cup\{a\}$. If $y=a$, then $xy\in A(T)$. Otherwise, $y\in V_j$ for some $1\le j\le d$. Let $Q$ be an $a-y$ geodesic, which contains one vertex in each of $V_1,\dots,V_j$. If $x\in V(Q)$, then $Q$ contains a rainbow $x-y$ path. Otherwise, $x\not\in V(Q)$, and $xa\cup Q$ is a desired path.

Hence, suppose that $x\in V_1$. Note that by the choice of $p$, if $x\neq p$, then either $xp\in A(T)$, or there exists $w\in V_1$ such that $xwp$ is a path. Indeed, if $xp\not\in A(T)$ and no such $w$ exists in $V_1$, then in $T[V_1]$, the set of in-neighbours of $x$ contains $p$, and all in-neighbours of $p$. Thus, $x$ contradicts the choice of $p$. Hence in $T[V_1]$, there is an $x-p$ path $P'$ of length at most $2$. Now, consider $y\not\in V_{d-1}\cup V_d$. If $y\in V(P'\cup P)$, then $P'\cup P$ contains a rainbow $x-y$ path. Otherwise, $P'\cup P\cup qy$ is a desired path. Next, if $y\in V_d$, then by a similar argument as for $T[V_1]$, we have a $q-y$ path $P''$ of length at most $2$ in $T[V_d]$, by the choice of $q$. We have $P'\cup P\cup P''$ is a desired path. Finally, let $y\in V_{d-1}$. If $ry\in A(T)$ in (i) or (ii), or $sy\in A(T)$ in (ii), then $P'\cup P\cup ry$ or $P'\cup P\cup sy$ contains a rainbow $x-y$ path. Otherwise, we can find an in-neighbour of $y$ in $V_{d-2}$, say $z$. Note that $z$ exists by considering an $a-y$ geodesic, and that $z\neq r$ in (i), or $z\not\in\{s,r\}$ in (ii). Then, $P'\cup P\cup qzy$ is a desired path.
\end{proof}

\section*{Acknowledgements}
Hui Lei and Yongtang Shi are partially supported by National Natural Science Foundation of China and Natural Science Foundation of Tianjin (No.~17JCQNJC00300). Shasha Li is partially supported by National Natural Science Foundation of China (No.~11301480), and Natural Science Foundation of Ningbo, China. Henry Liu is supported by International  Interchange Plan of CSU, and China Postdoctoral Science Foundation (Nos.~2015M580695, 2016T90756). Henry Liu also thank Chern Institute of Mathematics, Nankai University, for their generous hospitality. He was able to carry out part of this research during his visit there.

The authors thank the anonymous referees for the careful reading of the manuscript.

\end{document}